\documentclass[a4paper]{amsart}
\usepackage{amsmath,amsthm,amssymb}
\usepackage[T1]{fontenc}
\usepackage{url}
\usepackage{amsmath,amsthm,amssymb}
\usepackage{verbatim}
\usepackage{mathrsfs}
\usepackage{graphics}
\usepackage{graphicx}
\usepackage{subfigure}

\newtheorem{theorem}{Theorem}[section]
\newtheorem{proposition}[theorem]{Proposition}
\newtheorem{lemma}[theorem]{Lemma}
\newtheorem{corollary}[theorem]{Corollary}

\theoremstyle{remark}
\newtheorem{remark}[theorem]{Remark}
\newtheorem{definition}{Definition}

\numberwithin{equation}{section}

\def\be{\begin{equation}}
\def\ee{\end{equation}}
\def\bes{\begin{equation*}}
\def\ees{\end{equation*}}

\newcommand{\vep}{\varepsilon}
\newcommand{\R}{{\mathbb{R}}}

\newcommand{\Q}{{\mathbb{Q}}}
\newcommand{\C}{{\mathbb{C}}}
\newcommand{\Z}{{\mathbb{Z}}}
\newcommand{\N}{{\mathbb{N}}}

\newcommand{\xbm}{(X,\mathcal{B},\mu)}

\newcommand{\Int}{\operatorname{Int}}

\newcommand{\Aff}{\operatorname{Aff}}
\newcommand{\hol}{\operatorname{hol}}

\newenvironment{proofof}[2]{\begin{proof}[Proof of #1 \ref{#2}.]}{\end{proof}}

\let\oldmarginpar\marginpar
\renewcommand\marginpar[1]{\-\oldmarginpar[\raggedleft\footnotesize #1]%
{\raggedright\footnotesize #1}}

\begin{document}
\title
{Non-ergodic $\Z$-periodic billiards and infinite  translation surfaces}
\author[K. Fr\k{a}czek \and C. Ulcigrai]{Krzysztof Fr\k{a}czek \and Corinna Ulcigrai}

\address{Faculty of Mathematics and Computer Science, Nicolaus
Copernicus University, ul. Chopina 12/18, 87-100 Toru\'n, Poland}
 \email{fraczek@mat.umk.pl}
\address{Department of Mathematics\\
University Walk, Clifton\\
Bristol BS8 1TW, United Kingdom}
\email{corinna.ulcigrai@bristol.ac.uk}
\date{}

  \keywords{}
%\thanks{Research partially supported by MNiSzW grant N N201
%384834 and Marie Curie "Transfer of Knowledge" program, project
%MTKD-CT-2005-030042 (TODEQ)}
\maketitle
\begin{abstract}%\marginpar{to adapt}
We give a criterion which allows to prove non-ergodicity  for
certain infinite periodic billiards and directional flows on
$\mathbb{Z}$-periodic translation surfaces. Our criterion applies
in particular to a billiard in an infinite band with periodically
spaced vertical barriers and to the Ehrenfest wind-tree model,
which is a planar billiard with a $\mathbb{Z}^2$-periodic array of
rectangular obstacles. We prove that, in these two examples, both
for a full measure set of parameters of the billiard tables and
for tables with rational parameters, for almost every direction
the corresponding billiard flow is not ergodic and has uncountably
many ergodic components. As another application, we show that for
any recurrent $\mathbb{Z}$-cover of a square tiled surface of
genus two the directional flow is not ergodic and has no invariant
sets of finite measure for a full measure set of directions.
%As an application, we deduce that is not ergodic and has no invariant sets of finite measure for almost every direction.
In the language of essential values, we prove that the
skew-products  which arise as Poincar{\'e} maps of the above
systems are associated to non-regular $\mathbb{Z}$-valued cocycles
for interval exchange transformations.

\begin{comment}
This  paper concerns the ergodic  theory of infinite periodic
rational billiards and of linear flows on $\mathbb{Z}$-periodic
covers of compact  translation surfaces.
%Directional flows on recurrent $\mathbb{Z}$- covers of lattice surfaces were expected to be ergodic for almost every direction.
We exhibit a class of $\mathbb{Z}$-periodic translation surfaces
for which the directional flow  is recurrent, has no invariant
sets of finite measure, but fails to be ergodic for almost every
direction. Examples of applications of the result include the
unfolding of the billiard flow in an infinite tube with periodic
slits.  More precisely, we prove  that for any recurrent
$\mathbb{Z}$-cover of a square tiled surface of genus two, for
almost every direction the directional flow is not ergodic and, in
addition, that for any recurrent  $\mathbb{Z}$-cover of a Veech
surface and almost every direction, the directional flow has no
invariant sets of finite measure.
%As an application, we deduce that is not ergodic and has no invariant sets of finite measure for almost every direction.
In the language of essential values, we prove that the
$\mathbb{Z}$-valued  skew-products over interval exchange
transformations which arise as Poincar{\'e} maps of the above
$\mathbb{Z}$-covers are {non-regular}.
\end{comment}
% arise as , are non-regular, that is the only essential values are $0$ and $\inftny$.
\end{abstract}

\section{Introduction and main results}\label{intro:sec}

The ergodic theory of directional flows %$(\varphi^{\theta}_t)_{t\in\R}$
 on \emph{compact} translation surfaces
(definitions are recalled below) has been a rich and vibrant area
of research in the last decades, in connection with the study of
rational billiards, interval exchange transformations and
Teichm\"uller geodesic flows (see for example the surveys
\cite{Ma, ViB, YoLN, ZoFlat}). On the other hand, very little is
known about the ergodic properties  of directional flows on
\emph{non-compact} translation surfaces, for which the natural
invariant measure is \emph{infinite} (see \cite{Gu}).
%the underlying surface is non compact and

A natural motivation to study infinite translation surfaces, as in
the case of compact ones, come  from billiards. As linear flows on
compact translation surfaces arise for example by \emph{unfolding}
billiard flows in rational polygons, examples of flows on
\emph{infinite} translation surfaces can be obtained by unfolding
periodic rational billiards, for example in a band  (see the
billiard described below, Figure \ref{fig_bil} and
\S\ref{examples}) or in the plane (as the Ehrenfest wind-tree
model, see Figure \ref{Ehrenfestplane} and \S\ref{Ehrenfest:sec}).
The infinite translation surfaces obtained in this way are rich in
symmetry, and turns out to be $\mathbb{Z}^d$-covers (see below for
a definition) of compact translation surfaces. Poincar{\'e} maps
of directional flows on compact surfaces are piecewise isometries
known as interval exchange transformations;  Poincar{\'e} maps of
directional flows $\mathbb{\Z}^d$ covers are
$\mathbb{Z}^d$-extensions of  interval exchange transformations
(see \S \ref{Zcovers:sec} for the definitions of interval exchange
transformations and extensions).

% A first natural class  of infinite translation surfaces to consider are those which are $\mathbb{Z}^d$-covers (also defined below) of a compact translation surface. Such surfaces arise for example naturally from periodic billiards in tubes (see the Example below) or in the plane (as the Ehrenfest wind-tree model).

The ergodic  properties of directional flows on
$\mathbb{Z}^d$-covers and  more  generally of
$\mathbb{Z}^d$-extensions of interval exchange transformations
have been recently a very active area of research (see for example
\cite{Co-Fr, Co-Gu, Gu, Ho1,  Hu-We1, Ho-We, Hu-Le-Tr, Hu-We2}).
Recall that a measurable flow $(\varphi_t)_{t\in\R}$ on the
measurable space $(X,\mathcal{B})$ preserves the measure $\mu $
(where $\mu$ is $\sigma$-finite)  if $\mu( \varphi_t A )= \mu(A) $
for all $t\in \R$, $A \in \mathcal{B}$. The invariant measure
$\mu$  is \emph{ergodic} and we say that $(\varphi_t)_{t\in\R}$ is
ergodic with respect to $\mu$ if for any measurable set $A$ which
is almost invariant, i.e.~such that $\mu( \varphi_t A\triangle
A)=0 $ for all $t\in \mathbb{R}$, either $\mu(A)=0$ or
$\mu(A^c)=0$, where $A^c$ denotes the complement. In the classical
set-up, a celebrated result
%on ergodicity of directional flows on \emph{compact} translation surfaces
by Kerchoff-Masur-Smillie \cite{Ke-Ma-Sm} states that
for \emph{every} compact connected translation surface for
a.e. direction $\theta\in S^1$
% $(M,\omega)$
the directional flow in direction $\theta$ is \emph{ergodic}  with
respect to the Lebesgue measure  and moreover is \emph{uniquely
ergodic}, i.e.~the Lebesgue measure is the unique ergodic
invariant measure up to scaling.
%are direction % $({\varphi}^\theta_t)_{t\in\R}$
Some recent results
%(see \cite{})
concerning ergodicity  are in the direction of proving that also
for some $\mathbb{Z}$-covers  ergodicity holds for a full measure
set of directions,   for example in special cases as
$\mathbb{Z}$-covers of  surfaces of genus $1$ (see \cite{Hu-We1})
or of $\mathbb{Z}$-covers which have the lattice property  (see
Theorem \ref{HW}  quoted below, from \cite{Hu-We2}).
%  and more in general to prove ergodicity for a full measure set of directions also in the set up of $\mathbb{Z}$-covers (see for example \cite{})
Examples of ergodic directions in some infinite translation
surfaces  were also constructed by Hooper \cite{Ho1}).
%  classification of invariant Radon measures.

In contrast, in this paper we give a criterion
(Theorem~\ref{non-ergodicitycriterion}) which allows to show
non-ergodicity for some infinite billiards and $\Z$-covers of
translation surfaces. Our criterion allows us in particular  to
prove that some well-studied infinite periodic billiards, for
example the billiard in a band with barriers and the periodic
Erhenfest-wind tree model are not ergodic both for a full measure
set of parameters and for certain specific values of parameters
(Theorems \ref{stripbilliard} and \ref{Ehrenfestthm}).  Moreover,
we  show that such flows admits uncountably many ergodic
components (defined in \S\ref{Mackey}). The criterion for
non-ergodicity (Theorem~\ref{non-ergodicitycriterion}) requires
several preliminary definitions and it is therefore stated in \S
\ref{nonergodicity:sec}. Here below (\S\S \ref{examples} and
\ref{Ehrenfest:sec}) we formulate the two results just mentioned
about infinite billiards (Theorems \ref{stripbilliard} and
\ref{Ehrenfestthm}), that are based on this criterion. Another
application of the non-ergodicity criterion is given by Theorem
\ref{maintheorem}, which gives a class of $\mathbb{Z}$-covers of
translation surfaces  for which the set of ergodic directions
$\theta $ for the directional flow $({\varphi}^\theta_t)_{t\in\R}$
has \emph{measure zero}  (see \S\ref{staircases:sec}, where we
state Theorem \ref{maintheorem} after the preliminary definitions
in \S\ref{defnmain:sec} and comment on the relations with other
recent results).

Let us remark that our Theorems can be rephrased in the language
of  skew-products and essential values (as explained in \S
\ref{Zcovers:sec} and \S \ref{essentialvalues:sec} below). While
skew-products over rotations are well studied,
% and there is a large literature \marginpar{Add references},
very  few results were previously known for skew-products over
IETs. The first return (Poincar{\'e}) maps of the billiard flows
or of the directional flows considered provide examples of
skew-products associated to \emph{non-regular} cocycles for
interval exchange transformations (see \S\ref{essentialvalues:sec}
for the definition of non-regularity).

%\subsection{Two examples}
\subsection{A billiard in an infinite band}\label{examples}
% $\Z$-periodic tables and infinite staircases.}\label{examples}
Let us consider the infinite band $\R\times [0,1]$ with
periodically placed linear barriers (also called \emph{slits})
handling from the lower side of the band perpendicularly (see
Figure \ref{fig_bil}). We will denote by $T(l)=(\R\times
[0,1])\setminus (\Z\times[0,l])$ the billiard table in which the
length of the slit is given by the parameter $0<l<1$ as shown in
Figure \ref{fig_bil}. Let us recall that a billiard trajectory is
the trajectory of a point-mass which moves freely inside $T(l)$ on
segments of straight lines and undergoes elastic collisions (angle
of incidence equals to the angle of reflection) when it hits the
boundary of $T(l)$. An example of a billiard trajectory is drawn
in Figure \ref{fig_bil}.
\begin{figure}[h]
\includegraphics[width=0.6\textwidth]{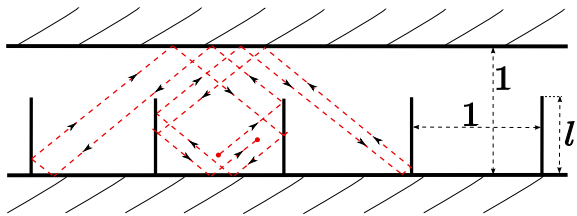}
\caption{Billiard flow on $T(l)$.}\label{fig_bil}
\end{figure}
The billiard  flow $(b_t)_{t \in \R}$ is defined on a full measure
set of points in the phase space $T^1(l)$, that consists of the
subset of points $(x, \theta) \in T(l) \times S^1$ such that if
$x$ belongs to the boundary of $T(l)$ then  $\theta$ is an inward
direction. For $t \in \R$ and  $(x,\theta )$ in the domain of
$(b_t)_{t\in\R}$,  $b_t$ maps $ (x, \theta)$ to $b_t(x,\theta) =
(x',\theta')$, where $x'$ is the point reached after time $t$ by
flowing at unit speed along the billiard trajectory starting at
$x$ in direction $\theta$ and $\theta'$ is the tangent direction
to the trajectory at $x'$.

%For every direction $\theta\in S^1$,
%consider the \emph{billiard flow} $(b_t^\theta)_{t\in\R}$, whose
%trajectories  (well defined for a full measure set of points which
%do not hit the corners of the billiard table) are \emph{billiard
%trajectories}: a point $p\in T(l)$ moves at unit speed on
%segments of straight lines with elastic collisions (angle of
%incidence equals to the angle of reflection) at the boundary of
%$T(l)$.
%\marginpar{Put references to literature on rectangle with barrier and HD non ue directions or not?}

The infinite billiard $(b_t)_{t\in\R}$ is an extension of a finite billiard (in a
rectangle with a barrier), whose fine dynamical properties were
studied in many papers (see \cite{Veech2,  Ch, Ch-Hu-Ma, EMS}).
Let us also remark that a similar billiard in a  semi-infinite
band is known as a \emph{retroreflector} and was studied in
\cite{BKM}.

Since the directions of any billiard trajectory in $T(l)$ are at
most four, the set $T(l)\times \Gamma\theta$, where
$\Gamma\theta:=\{\theta, -\theta, \pi-\theta, \pi + \theta\}$, is
an invariant subset in the phase space $T^1(l) $ for the billiard
flow on $T(l)$. The flow $(b^{\theta}_t)_{t\in\R}$ will denote the
restriction of $(b_t)_{t\in \R}$ to this invariant set. Remark
that  the directional billiard flow $(b^{\theta}_t)_{t\in\R}$
preserves the product of the Lebesgue measure on $T(l)$ and the
counting measure on the orbit $\Gamma\theta$. We say that
$(b_t^\theta)_{t\in\R}$ on $T(l)$ is \emph{ergodic} if it is
ergodic with respect to this natural invariant measure.

\begin{theorem}\label{stripbilliard}
Consider the billiard flow  $(b_t)_{t\in\R}$ on the infinite strip
$T(l)$. There exists a set $\Lambda \subset [0,1]$ of full
Lebesgue measure such that, if either:
\begin{itemize}
\item[(1)] $l$ is a rational number, or
\item[(2)] $l \in \Lambda$, %a full Lebesgue measure set of parameters in $(0,1)$.
\end{itemize}
then for almost every $\theta\in S^1$ the directional billiard
flow $(b^\theta_t)_{t\in\R}$ on $T(l)$ is recurrent and not
ergodic.  Moreover, $(b^\theta_t)_{t\in\R}$ has uncountably many
ergodic components.
\end{theorem}
Let us remark that, even though we prove that the result holds for
a full measure set of parameters  $\Lambda$, the assumption $(1)$
is more precise since it gives \emph{concrete} values of the
parameters for which the conclusion holds. It is natural to ask if
there  exists exceptional directions $\theta \in S^1$ and $l \in
(0,1)$ for which the flow $(b^\theta_t)_{t\in\R}$ is ergodic. In
\cite{Fr-Ul:erg_bil}  it is shown that for all $l=p/q \in
\mathbb{Q}\cap(0,1)$ there exists a positive Hausdorff dimension
set of exceptional directions $\theta \in S^1$ for which
$(b^\theta_t)_{t\in\R}$ on $T(p/q)$ is ergodic.

\subsection{The Ehrenfest wind-tree model}\label{Ehrenfest:sec}
The Ehrenfest wind-tree billiard is a model of a gas particle
introduced in 1912 by P. and T. Ehrenfest. The periodic version,
which was first studied by Hardy and Weber in \cite{Ha-We},
consist of a $\mathbb{Z}^2$-periodic planar array of rectangular
scatterers, whose sides are given by two parameters  $0<a,b<1$
(see Figure \ref{Ehrenfestplane}). The billiard flow  in  the
complement $E_{2}(a,b)$ of the interior of the rectangles is the
Ehrenfest wind-tree billiard, that we will denote by
$(e_t)_{t\in\R}$. An example of a billiard trajectory is also
shown in Figure \ref{Ehrenfestplane}.  Many results on the
dynamics of the periodic wind-tree models, in particular on
recurrence and diffusion times, were proved recently, see
\cite{Co-Gu, Hu-Le-Tr, Trou, Del, DHL}.
\begin{figure}[h]
\includegraphics[width=0.6\textwidth]{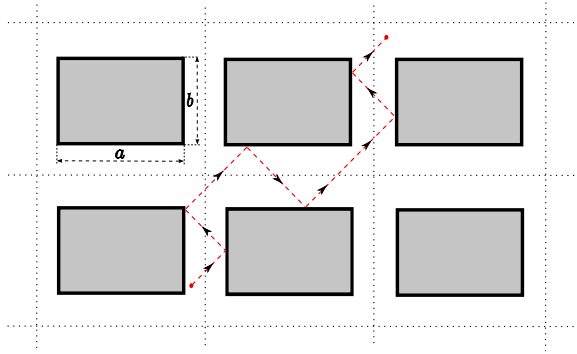}
\caption{Ehrenfest wind-tree billiard on $E_{2}(a,b)$.}\label{Ehrenfestplane}
\end{figure}
\begin{figure}[h]
\includegraphics[width=0.6\textwidth]{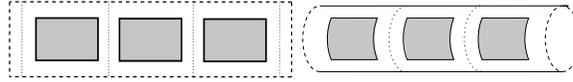}
\caption{Ehrenfest wind-tree billiard on $E_{1}(a,b)$.}\label{Ehrenfesttube}
\end{figure}
One can also consider a one-dimensional version of the periodic
Ehrenfest  wind-tree model, whose configuration space $E_1(a,b)$
is an infinite tube $\mathbb{R}\times (\R/\Z)$ with
$\mathbb{Z}$-periodic rectangular scatterers  (see Figure
\ref{Ehrenfesttube}) of horizontal and vertical sides of lengths
$a$  and $b$ respectively. We will also denote by $(e_t)_{t\in\R}$
the billiard flow  in $E_1(a,b)$.  As for the billiard in a strip
in \S\ref{examples}, any trajectory of $(x,\theta)$ for
$(e_t)_{t\in\R}$ in $E_1(a,b)$ or in $E_2(a,b)$ travels in at most
four directions, belonging to the set $\Gamma \theta :=\{ \pm
\theta, \theta \pm \pi\}$. The restriction of $(e_t)_{t\in \R}$ to
the invariant set $E_i(a,b)\times\Gamma\theta$ for $i=1,2$ will be
denoted by $(e_t^\theta)_{t\in\R}$. The directional billiard flow
$(e_t^\theta)_{t\in\R}$ preserves the product measure $\mu$ of the
Lebesgue measure on $E_1(a,b)$ ($E_2(a,b)$) and the counting
measure on $\Gamma\theta$ and the ergodicity of
$(e_t^\theta)_{t\in\R}$ refers to ergodicity with respect to this
measure $\mu$.

\begin{theorem}\label{Ehrenfestthm}
Consider the billiard flow $(e_t)_{t\in\R}$ in the
$\mathbb{Z}$-periodic  Ehrenfest wind-tree model $E_1(a,b)$. There
exists a set $\mathcal{P} \subset [0,1]^2$ of full Lebesgue
measure such that, if either:
\begin{itemize}
\item[(1)] $a,b\in(0,1)$ are rational numbers, or
\item[(2)]  $a,b\in(0,1)$ can be written as $1/(1-a)=x+y\sqrt{D}$,
$1/(1-b)=(1-x)+y\sqrt{D}$  with $x,y\in \mathbb{Q}$ and $D$ a
positive square-free integer, or
\item[(3)] $(a,b) \in \mathcal{P}$, % belong to a full  Lebesgue measure set of parameters in $[0,1]\times [0,1]$.
\end{itemize}
then for almost every $\theta\in S^1$ the directional billiard
flow $(e^\theta_t)_{t\in\R}$ on $E_1(a,b)$ is recurrent and not
ergodic. Moreover,  $(e^\theta_t)_{t\in\R}$ has uncountably many
ergodic components.
\end{theorem}
As in Theorem \ref{stripbilliard}, the result holds by $(3)$ for the
full measure set of parameters $\mathcal{P}$, but only the assumptions
$(1)$ and $(2)$ give \emph{concrete} values of the parameters $(a,b)$ for
which the conclusion holds.

As a corollary, since $(e_t^\theta)_{t\in\R}$ in $E_2(a,b)$ is  a
cover of $(e_t^\theta)_{t\in\R}$  on $E_1(a,b)$, we have the
following conclusion about the original Ehrenfest periodic model.
\begin{corollary}\label{Ehrenfestcor}
If $(a,b)$ satisfy either $(1), (2)$ or $(3)$ in Theorem
\ref{Ehrenfestthm},  then for almost every $\theta\in S^1$ the
planar periodic Ehrenfest wind tree model $(e_t^\theta)_{t\in\R}$
on $E_2(a,b)$ is  not ergodic and moreover, there are uncountably
many ergodic components.
\end{corollary}

\subsection{Directional flows on translation surfaces and $\Z$-covers}\label{defnmain:sec}
We now recall some basic definitions to then state  (in \S
\ref{staircases:sec}) another application of our non-ergodicity
criterion (Theorem~\ref{non-ergodicitycriterion}) for a class of
$\mathbb{Z}$-covers of translation surfaces. Let $M$ be an
oriented surface  (not necessarily compact). A translation surface
$(M,\omega)$ is a complex structure on $M$ together with an
nonzero \emph{Abelian differential} $\omega$, that is a non-zero
holomorphic $1$-form. Since given an Abelian differential  there
exists an uniquely complex structure which is compatible with it,
the translation structure on $M$ is uniquely defined by $\omega$.
Let $\Sigma=\Sigma_\omega\subset M$ be the set of zeros of
$\omega$. For every $\theta\in S^1 = \R/2\pi \Z $ %\{ \theta \in \mathbb{C}, |\theta|=1\}$
denote by $X_\theta=X^{\omega}_\theta$ the directional vector
field in direction $\theta$ on $M\setminus\Sigma$, defined by
$i_{X_\theta}\omega =e^{i\theta}$. Then the corresponding
directional flow
$(\varphi^{\theta}_t)_{t\in\R}=(\varphi^{\omega,\theta}_t)_{t\in\R}$
(also known as \emph{translation flow}) on $M\setminus\Sigma$
preserves the volume form
$\nu_{\omega}=\frac{i}{2}\omega\wedge\overline{\omega}=\Re(\omega)\wedge\Im(\omega)$.
We will use the notation $(\varphi^{v}_t)_{t\in\R}$ and $X_h$ for
the \emph{vertical flow and vector field} (corresponding to
$\theta = \frac{\pi}{2}$) and $(\varphi^{h}_t)_{t\in\R}$ and $X_h$
for the \emph{horizontal flow and vector field} respectively
($\theta = 0$). We will always consider translation surfaces of
\emph{area one}, that is renormalized so that
$A(\omega):=\nu_\omega (M)$ is equal to one. We will denote by
$M_\theta$ the set of regular points for the directional flow
$(\varphi^{\theta}_t)_{t\in\R}$, i.e. the set of point for which
the orbit of the flow may be defined for all times $t\in\R$. Then
$M_\theta$ is a Borel subset of $M$ with $\nu_\omega(M\setminus
M_\theta)=0$ and $(\varphi^{\theta}_t)_{t\in\R}$ restricted to
$M_\theta$ is a well defined Borel flow.

Let $(M,\omega)$ be a compact connected translation surface.
Recall that a {\em $\Z$-cover} of $M$ is a %connected
 manifold $\widetilde{M}$ with a free totally discontinuous action of the group $\Z$
%($\,\cdot:\Z\times\widetilde{M}\to\widetilde{M}$)
such that the quotient manifold $\widetilde{M}/\Z$ is homeomorphic
to $M$. The map $p:\widetilde{M}\to M$ obtained by  composition of
the projection $\widetilde{M}\to\widetilde{M}/\Z$ and the
homeomorphism $\widetilde{M}/\Z\to M$ is called a {\em covering
map}. Denote by $\widetilde{\omega}$ the pullback of the form
$\omega$ by the map $p$. Then $(\widetilde{M},\widetilde{\omega})$
is a translation surface as well. As we recall at the beginning of
Section \ref{Zcovers:sec},  $\Z$-covers of M up to isomorphism are
in one-to-one correspondence with homology classes in $H_1(M,\Z)$.
%\[\operatorname{Hom}(H_1(M,\Z),\Z)\longleftrightarrow\{\Z\text{-covers of }M\}/\text{isomorphism},\]
%which we recall  in
For every $\gamma\in H_1(M,\Z)$ we will denote by
$(\widetilde{M}_\gamma,\widetilde{\omega}_\gamma)$ the translation
surface associated to the $\Z$-cover given by $\gamma$.
%The translation surface$(\widetilde{M}_\gamma,\widetilde{\omega}_\gamma)$ will be called%a $\Z$-cover of $(M,\omega)$.

For any $\Z$-cover $(\widetilde{M},\widetilde{\omega})$ of the
translation surface $(M,\omega)$ and $\theta\in S^1$ denote by
$({\varphi}^\theta_t)_{t\in\R}$ and
$(\widetilde{\varphi}^\theta_t)_{t\in\R}$ the volume-preserving
directional flows on $(M,\nu_{\omega})$ and
$(\widetilde{M},\nu_{\widetilde{\omega}})$ respectively.
Recall
that a measure-preserving flow $(\varphi_t)_{t\in\R}$ on
$(X,\mathcal{B},\mu)$ ($\mu$ is $\sigma$-finite) is {\em
recurrent} if  for any $A\in\mathcal{B}$ with $\mu(A)>0$, for a.e.
$x\in A$ there is $t_n \to\infty$ such that $\varphi_{t_n}x\in A$.

Denote by $\hol:H_1(M,\Z)\to \C$ the {\em holonomy map}, i.e.\
$\hol(\gamma)=\int_\gamma\omega$ for every $\gamma\in H_1(M,\Z)$.
As recently shown by Hooper and Weiss (see Proposition~15   in
\cite{Ho-We}) a curve $ \gamma$ on $({M},{\omega})$ has
$\hol(\gamma)=0$ if and only if for every $\theta\in S^1$ such
that $({\varphi}^\theta_t)_{t\in\R}$ is ergodic, the flow
$(\widetilde{\varphi}^\theta_t)_{t\in\R}$ on the $\Z$-cover
$(\widetilde{M}_\gamma,\widetilde{\omega}_\gamma)$  is
\emph{recurrent}. Thus, following Hooper and Weiss, we adopt the
following definition:

\begin{definition}[see \cite{Ho-We}]
The $\Z$-cover $(\widetilde{M}_\gamma,\widetilde{\omega}_\gamma)$
of the translation surface $(M,\omega)$ given by $\gamma \in
H_1(M,\Z)$
 is called {\em recurrent}
if $\hol(\gamma)=0$.
\end{definition}
\begin{comment}
\begin{proposition}[Proposition~15 in \cite{Ho-We}]
A $\Z$-cover $(\widetilde{M},\widetilde{\omega})$ of
$(M,\omega)$ is recurrent if and only if  for every $\theta\in
S^1$ the flow $(\widetilde{\varphi}^\theta_t)_{t\in\R}$ is
recurrent provided that $({\varphi}^\theta_t)_{t\in\R}$ is
ergodic.
\end{proposition}
\end{comment}

Recall that a  translation surface $(M,\omega)$ is {\em
square-tiled}   if there exists a ramified cover $p:M\to\R^2/\Z^2$
unramified outside $0\in\R^2/\Z^2$ such that $\omega=p^*(dz)$.
Square tiled surfaces are also known as \emph{origamis}. Examples of square
tiled surface $(M,\omega)$ can be realized by gluing finitely (or
infinitely) many  squares of equal sides in $\mathbb{R}^2$ by
identifying each left vertical side of a square with a right vertical
side of some square and each top horizontal side with a bottom
horizontal side via translations.

%Our main result produces examples of recurrent $\Z$-covers $(\widetilde{M},\widetilde{\omega})$ for which the directional flow
%$(\widetilde{\varphi}^\theta_t)_{t\in\R}$ is \emph{ergodic} for a Lebesgue measure zero set of $\theta \in S^1$.
\subsection{$\Z$-covers of genus two square tiled surfaces and staircases}\label{staircases:sec}
Another application of the non-ergodicity criterion
(Theorem~\ref{non-ergodicitycriterion}) is the following.
\begin{theorem}\label{maintheorem}
If $(M,\omega)$ is square-tiled translation surface of  genus $2$,
for any recurrent $\Z$-cover
$(\widetilde{M}_\gamma,\widetilde{\omega}_\gamma)$ given by a non
trivial $\gamma \in H_1(M, \mathbb{Z})$ and for a.e. $\theta\in
S^1$ the directional flow
$(\widetilde{\varphi}^\theta_t)_{t\in\R}$ is {not} ergodic.
Moreover, it has no invariant sets of positive measure and has
uncountably many ergodic components.
%Moreover, for for a.e.
%$\theta\in S^1$ the directional flow
%$(\widetilde{\varphi}^\theta_t)_{t\in\R}$ has no invariant sets of positive measure.
\end{theorem}
Let us give an example to which Theorem \ref{maintheorem} applies.
Consider the  infinite staircase in Figure \ref{staircase} and let
us denote by $Z^\infty_{(3,0)}$ the surface obtained by
identifying the opposite parallel sides belonging to the boundary  by
translations (the notation $Z^\infty_{(3,0)}$ refers to
\cite{Hu-Sch}). The surface $Z^\infty_{(3,0)}$ inherits from
$\mathbb{R}^2$ a translation surface structure and thus one can
consider the directional flows $(\varphi^\theta_t)_{t\in\R}$ in
direction $\theta$ on $Z^\infty_{(3,0)}$.
\begin{figure}[h]
\centering \subfigure[Translation surface
$Z^\infty_{(3,0)}$.]{\label{staircase}
\includegraphics[height=0.15\textheight]{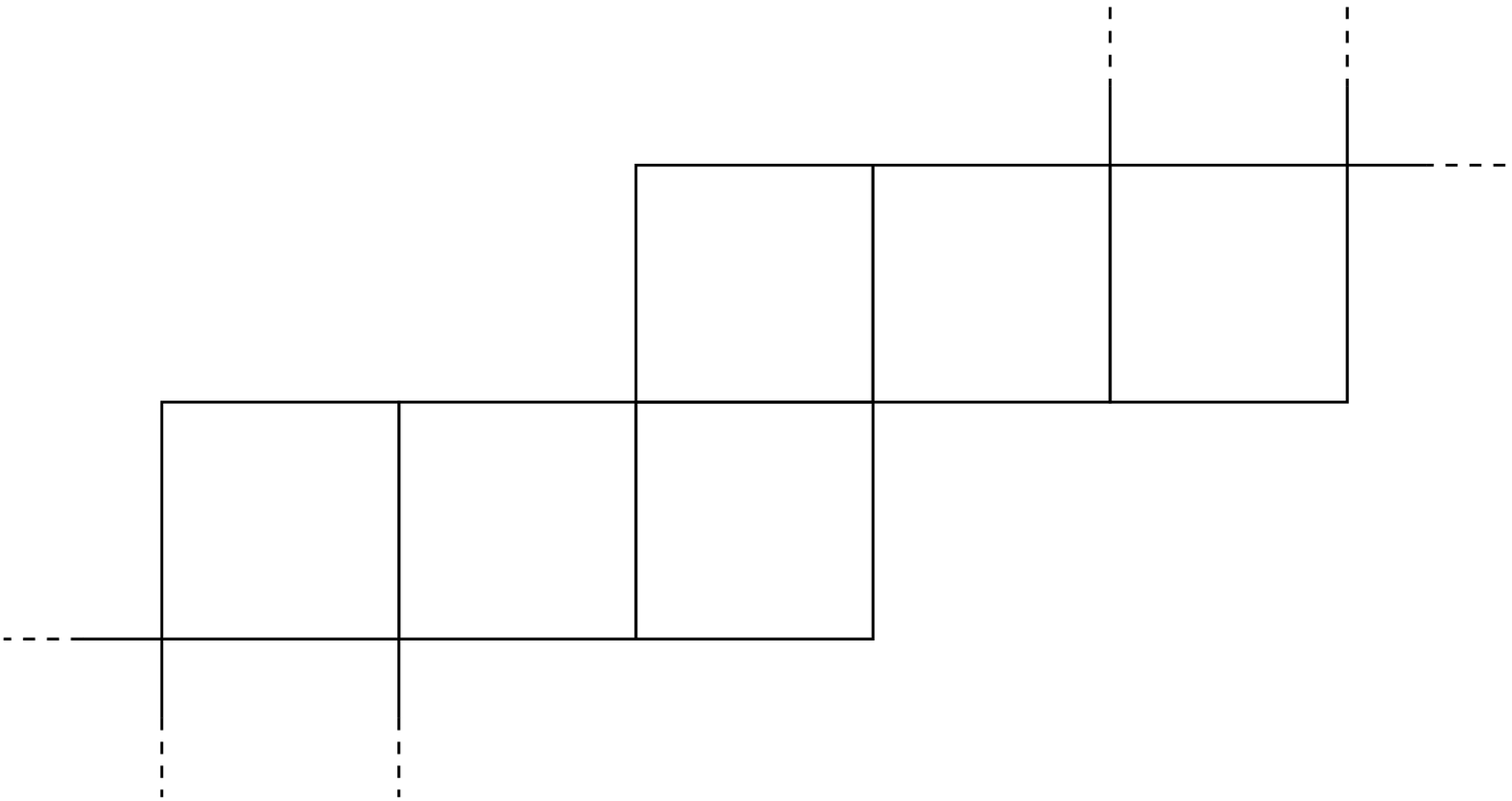}}
\hspace{2mm} \subfigure[Translation surface
$Z_{(3,0)}$.]{\label{example3}
\includegraphics[width=0.3\textwidth]{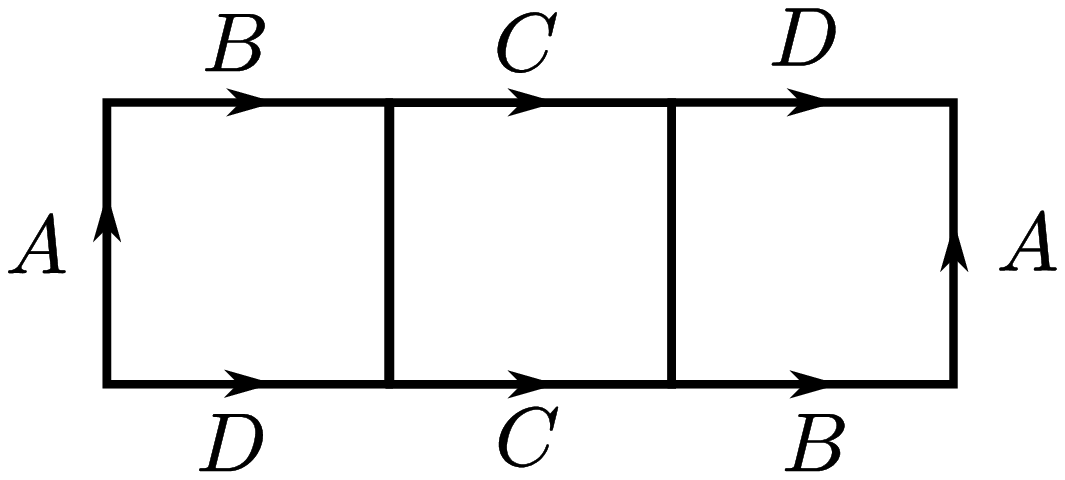}}
\caption{The infinite staircase translation surface
$Z^\infty_{(3,0)}$.}\label{Rauzy}
\end{figure}
%\begin{figure}[h]
%\includegraphics[height=0.18\textheight]{stair3.eps}
%\caption{The infinite staircases translation surface
%$Z^\infty_{(3,0)}$.}\label{staircase}
%\end{figure}
\begin{comment}
\begin{figure}[h]
\centering
\subfigure[$Z^\infty_{(3,0)}$]{\label{staircase}
\includegraphics[height=0.18\textheight]{stair3.eps}}
\hspace{6mm}
\subfigure[$Z^\infty_{(2,0)}$]{\label{stair2}
\includegraphics[height=0.6\textheight]{stair3.eps}}
\caption{Infinite staircases translation surfaces}
\end{figure}
\end{comment}
One can see that this infinite translation surface is a
$\mathbb{Z}$-cover of the genus two square-tiled surface
$Z_{(3,0)}$ shown in Figure \ref{example3}. Thus, as a consequence
of Theorem~\ref{maintheorem} we get:
\begin{corollary}\label{non-divergence}
The set of directions $\theta\in S^1$ such that the directional
flow $(\varphi^\theta_t)_{t\in\R}$ on the infinite staircase
$Z^\infty_{(3,0)}$ is ergodic  has Lebesgue measure zero.
Moreover, for almost every $\theta\in S^1$,
$(\varphi^\theta_t)_{t\in\R}$ has no invariant sets of finite
measure.
\end{corollary}
More generally, a countable family of staircases translation
surfaces $Z^\infty_{(a,b)}$ depending on the natural parameters
$a\geq 2, b\geq 0$ was defined and studied by Hubert and
Schmithüsen in \cite{Hu-Sch}. For $a>2$, these translation
surfaces are $\mathbb{\mathbb{Z}}$-covers of genus $2$
square-tiled surfaces.  Thus, Corollary \ref{non-divergence} holds
for any  $Z^\infty_{(a,b)}$ with $a>2, b\geq 0$.

On the other hand, we remark  that, if  one starts from the
staircase  in Figure \ref{staircase2} and obtains the  translation
surface known as $Z^\infty_{(2,0)}$ by identifying opposite
parallel sides belonging to the boundary, the set of directions
$\theta$ such that  the directional flow
$(\varphi^\theta_t)_{t\in\R}$ on the infinite staircase
$Z^\infty_{(2,0)}$ is ergodic  has \emph{full} Lebesgue measure
(see \cite{Hu-We2}). This difference is related to the fact that
$Z^\infty_{(2,0)}$   is \emph{not}  a $\mathbb{Z}$-cover of a
genus $2$ surface and the study of the directional flows on
$Z^\infty_{(2,0)}$ can be reduced to well-know results of
ergodicity of skew products over rotations (see \cite{Hu-We2} for
references).

\begin{figure}[h]
\includegraphics[height=0.13\textheight]{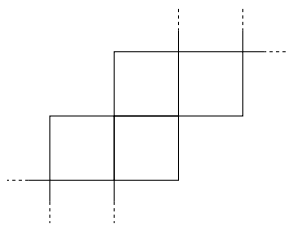}
\caption{The infinite staircases translation surface
$Z^\infty_{(2,0)}$.}\label{staircase2}
\end{figure}
Let us comment on the relation between  Corollary
\ref{non-divergence} of our theorem and another recent result by
Hubert and Weiss. In section \S \ref{Veech:sec} we recall the
definition of the Veech group $SL(M, \omega) < SL(2, \mathbb{R})$
of a translation surface. We say that a translation surface $({M},
\omega)$ (compact or not) is a \emph{lattice surface} if the Veech
group is a lattice in $SL(2, \mathbb{R})$. We say that a
(infinite) translation surface $(\widetilde{M},
\widetilde{\omega}) $  has an \emph{infinite strip} if  there
exists a subset of  $\widetilde{M}$ isometric to the strip
$\mathbb{R}\times (-a,a)$ for some $a>0$ (with respect to the flat
metric induced by $\widetilde{\omega}$ on $\widetilde{M}$).
\begin{theorem}[Hubert-Weiss, \cite{Hu-We2}]\label{HW}
Let $(\widetilde{M}, \widetilde{\omega})$ be a $\mathbb{Z}$-cover
that is a lattice surface and has an infinite strip. Then the
directional flow $(\varphi^\theta_t)_{t\in\R}$ on
$(\widetilde{M},\widetilde{\omega})$ is ergodic for a.e. $\theta
\in {S}^1$.
\end{theorem}
One can easily check that  $Z_{(3,0)}^\infty$ has an
\emph{infinite strip} (for example in the direction $\theta =
\frac{\pi}{4}$). On the other hand, as it was proved in
\cite{Hu-Sch}, the Veech group $SL(Z_{(3,0)}^\infty)$ is of the
first kind, is infinitely generated and is \emph{not} a lattice.
Thus, our result shows that the assumption that $SL(\widetilde{M},
\widetilde{\omega})$ (and not only $SL( {M}, {\omega})$) is a
lattice is essential for the conclusion of Theorem \ref{HW} to
hold.
%Remark that all directions in the conical limit set of $SL(Z_3^\infty, \mathbb{R})$ are non-divergent. Thus, from Corollary \ref{non-divergence} implies that the conical limit set of $SL(Z_3^\infty, \mathbb{R})$ has Lebesgue $0$, or equivalently, that $Z_3^\infty$ is of convergence type.

% Moreover, $Z_3^\infty$ has an
%infinite strip. On the other hand, by Theorem~\ref{maintheorem},
%the directional flow on $Z_3^\infty$ is not ergodic for almost
%every direction which gives a counter-example to Theorem~1 in
%\cite{Hu-We}.

%Both Corollaries are proved in Section \ref{applicationssec}.

\subsection{Outline and structure of the paper}
The Sections from \ref{Zcovers:sec} to \ref{Veech:sec} contain
background material and preliminary results. In \S
\ref{Zcovers:sec} we recall the construction of
$\mathbb{Z}$-covers associated to a homology class and  the
definitions of interval exchange transformations (IETs)
%(\S \ref{IET:sec})#
and $\mathbb{Z}$-extensions. We also explain how the study of directional flows on $\mathbb{Z}$ covers can be reduced to the study of $\mathbb{Z}$-extensions of IETs. We then
present some definitions and results used in the proofs about
the theory of essential values (Section
\ref{essentialvalues:sec}), the Kontsevich-Zorich cocycle  (Section
\ref{Teich:sec}) and lattice surfaces (Section \ref{Veech:sec}).

The heart of the paper is contained in
Section~\ref{nonergodicity:sec}, where the  criterion for
non-ergodicity (Theorem \ref{non-ergodicitycriterion}) is both
stated and proved.
%While Theorem \ref{non-ergodicitycriterion} provides the grounds for the non-ergodicity conclusion of the proofs of the Theorems stated in the introduction, the conclusion on the presence of uncoutably many ergodic invariant measures is based on
In Section~\ref{nonregularity:sec} we state and prove  Theorem
\ref{inv_fin_meas} (on the absence of invariant sets of finite
measure), which provides another crucial ingredient to prove the
presence of  uncountably many ergodic components in the various
applications.

The proofs  of the results stated in the introduction is finally
given in Section \ref{Fubini:sec} and follows from Theorems
\ref{non-ergodicitycriterion} and \ref{inv_fin_meas} essentially
from Fubini-type arguments. The first Fubini argument presented
applies to  Veech surfaces  and appears in
\S\ref{applications:sec}, where we  prove Theorem
\ref{maintheorem} and Corollary \ref{non-divergence}. In
\S\ref{strip} and \S\ref{Ehrenfest:finalsec} we prove respectively
Theorem \ref{stripbilliard} on the billiard in a strip  and
Theorem \ref{Ehrenfestthm} and Corollary \ref{Ehrenfestcor} on the
Ehrenfest wind-tree models.

In the Appendix we include the  proof of two technical results used in the proof of the non-ergodicity criterion and stated in Section \ref{Teich:sec}, i.e.~Lemma
\ref{comparisontimes:lemma} and Theorem \ref{cohthm}, which relates coboundaries with the unstable space of the Kontsevich-Zorich.

\section{$\Z$-covers and extensions of interval exchange transformations}\label{Zcovers:sec}
%\subsection{Definitions of $\Z$-covers, IETs and cocycle extensions.}
\subsubsection*{$\Z$-covers}%\label{Zcovers}
Let $(M,\omega)$ be a compact connected translation surface  and
$\widetilde{M}$  a  {\em $\Z$-cover} of $M$ (see \S
\ref{intro:sec}). Let us show that there is a one-to-one
correspondence between $H_1(M,\Z)$ and the set of $\Z$-covers, up
to isomorphism\footnote{Let us remark that here we consider only unramified $\mathbb{Z}$-covers. More generally, one can consider ramified covers determined by elements in the relative homology $H_1(M, \Sigma, \mathbb{Z})$, see \cite{Ho-We}.}.
%  is a connected topological
%manifold $\widetilde{M}$ with totally disconnected action of the
%group $\Z$ ($\,\cdot:\Z\times\widetilde{M}\to\widetilde{M}$) such
%that the quotient manifold $\widetilde{M}/\Z$ is homeomorphic to
%$M$. Then the composition of the projection
%$\widetilde{M}\to\widetilde{M}/\Z$ and the homeomorphism
%$\widetilde{M}/\Z\to M$, denoted by $p:\widetilde{M}\to M$, is
%called a {\em cover map}.
%Denote by $\widetilde{\omega}$ the
%pullback of the form $\omega$ by the map $p$. Then
%$(\widetilde{M},\widetilde{\omega})$ is a translation surface as
%well.
Let us first recall  that we have the following isomorphism  (we
refer for example to Proposition~14.1 in \cite{Ful}):
%homomorphisms from $\pi_1(M,x)$ to $\Z$ and
\[\operatorname{Hom}(\pi_1(M,x),\Z)\longleftrightarrow\{\Z\text{-covers of }M\}/\text{isomorphism}.\]
%\begin{proposition}[see Proposition~14.1 in \cite{Ful}]
%There is a one-to-one correspondence between the set of
%%homomorphisms from $\pi_1(M,x)$ to $\Z$ and the set of
%$\Z$-covers, up to isomorphism:
%\[\operatorname{Hom}(\pi_1(M,x),\Z)\longleftrightarrow\{\Z\text{-covers of }M\}/\text{isomorphism}.\]
%\end{proposition}
In view of Hurewicz theorem $\pi_1(M,x)/[\pi_1(M,x),\pi_1(M,x)]$
and $H_1(M,\Z)$ are isomorphic, so
$\operatorname{Hom}(\pi_1(M,x),\Z)$ and
$\operatorname{Hom}(H_1(M,\Z),\Z)$ are isomorphic as well.  This
yields a one-to-one correspondence
\[\operatorname{Hom}(H_1(M,\Z),\Z)\longleftrightarrow\{\Z\text{-covers of }M\}/\text{isomorphism}.\]
The space $H_1(M,\Z)$ is isomorphic to
$\operatorname{Hom}(H_1(M,\Z),\Z)$ via the map
$\gamma\mapsto\phi_\gamma: H_1(M,\Z)\to\Z$,
$\phi_\gamma(\gamma')=\langle \gamma,\gamma'\rangle $, where
$\langle \,\cdot\,,\,\cdot\,\rangle :H_1(M,\R)\times
H_1(M,\R)\to\R$ is the intersection form (see for example
Proposition~18.13 in \cite{Ful}). This gives the next
correspondence
\begin{equation}\label{correspondence}
H_1(M,\Z)\longleftrightarrow\{\Z\text{-covers of
}M\}/\text{isomorphism}.
\end{equation}
The  $\Z$-cover $\widetilde{M}_\gamma$   determined by $\gamma\in
H_1(M,\Z)$ under the correspondence (\ref{correspondence}) has the
following properties.  Remark that $\langle \,\cdot\,,\,\cdot\,
\rangle $ restricted to $ H_1(M,\Z) \times  H_1(M,\Z)$ coincides
with the algebraic intersection number. If $\sigma$  is a close
curve in $M$ and  $n:=\langle  \gamma, [\sigma] \rangle \in
\mathbb{Z}$ ($[\sigma]\in H_1(M,\Z)$), then $\sigma$  lifts to a
path $\widetilde{\sigma}: [t_0, t_1]\to \widetilde{M}_\gamma$ such
that $\sigma(t_1) = n \cdot \sigma(t_0)$, where $\cdot$ denotes
the action of $\Z$ on
$(\widetilde{M}_\gamma,\widetilde{\omega}_\gamma)$ by deck
transformations. Conversely, if $v:[t_0,t_1]\to\widetilde{M}$ is a
curve such
\begin{equation}\label{lift_curve}
v(t_1)=n\cdot v(t_0)\text{ for some }n\in\Z, \ \text{ then
}\langle \gamma , [p\circ v]\rangle =n,
\end{equation}
where $[p\circ v]\in H_1(M,\Z)$ is the homology class of the
projection  of $v$ by $p:\widetilde{M}_\gamma \to M$.

\subsubsection*{Interval exchange transformations}%\label{IET:sec}
Let us recall the definition of interval exchange transformations
(IETs), with the presentation and notation from \cite{Vi0} and
\cite{ViB}. Let $\mathcal{A}$ be a $d$-element alphabet and let
$\pi=(\pi_0,\pi_1)$ be a pair of bijections
$\pi_\vep:\mathcal{A}\to\{1,\ldots,d\}$ for $\vep=0,1$. Denote by
$\mathcal{S}_{\mathcal{A}}$ the set of all such pairs. Let us
consider $\lambda=(\lambda_\alpha)_{\alpha\in\mathcal{A}}\in
\R_+^{\mathcal{A}}$, where $\R_+=(0,+\infty)$. Set
$|\lambda|=\sum_{\alpha\in\mathcal{A}}\lambda_\alpha$,
$I=\left[0,|\lambda|\right)$ and, for $\epsilon=0,1$, let
\[I^\epsilon_{\alpha}=[l^\epsilon_\alpha,r^\epsilon_\alpha),\quad\text{ where }\quad
l^\epsilon_\alpha=\sum_{\pi_\epsilon(\beta)<\pi_\epsilon(\alpha)}\lambda_\beta,\;\;\;r^\epsilon_\alpha
=\sum_{\pi_\epsilon(\beta)\leq\pi_\epsilon(\alpha)}\lambda_\beta.\]
Then $|I_\alpha^\epsilon|=\lambda_\alpha$ for
$\alpha\in\mathcal{A}$. Given $(\pi,\lambda)\in
\mathcal{S}_{\mathcal{A}}\times\R_+^\mathcal{A}$, let
$T_{(\pi,\lambda)}:[0,|\lambda|)\rightarrow[0,|\lambda|)$ stand
for the {\em interval exchange transformation} (IET) on $d$
intervals $I_\alpha$, $\alpha\in\mathcal{A}$, which isometrically
maps each $I^0_\alpha$ to $I^1_\alpha$, i.e.\ $T_{(\pi,\lambda)}
(x) = x+w_\alpha$ with $w_\alpha:=l^1_\alpha - l^0_\alpha$, for $x
\in I^0_\alpha$, $\alpha\in \mathcal{A}$.

\subsubsection*{Cocycles and skew-product extensions}%\label{cocycles:sec}
Let $T$ be an ergodic automorphism of standard probability space
$\xbm$. Let $G$ be a locally compact abelian second countable
group. Each measurable function $\psi:X \rightarrow G$ determines
a
{\em cocycle} $\psi^{(\,\cdot\,)}$ %:\Z\times I \to \mathbb{R}$
for $T$ by the formula
\begin{equation}\label{cocycledef}
\psi^{(n)}(x)=\left\{
\begin{array}{ccl}
\psi(x)+\psi(Tx)+\ldots+\psi(T^{n-1}x) & \mbox{if} & n>0 \\
0 & \mbox{if} & n=0\\
-(\psi(T^nx)+\psi(T^{n+1}x)+\ldots+\psi(T^{-1}x)) & \mbox{if} &
n<0,
\end{array}
\right.
\end{equation}
the function $\psi$ is also called a cocycle. The {\it skew
product extension} associated to the cocycle $\psi$ is the map
$T_\psi: X\times G \rightarrow  X\times G $
\[ T_{\psi}(x,y)=(Tx,y+{\psi}(x)). \]
Clearly $T_\psi$ preserves the product of $\mu$ and  the Haar
measure $m_G$ on $G$.

\subsection{Reduction to $\mathbb{Z}$-extensions over IETs}\label{reduction:sec}
Let us explain how the question of ergodicity for directional
flows for $\Z$-covers of a compact translation surface
$(M,\omega)$ reduces to the study of $\Z$-valued cocycles for
interval exchange transformations (IETs). Let
$(\widetilde{\varphi}^\theta_t)_{t\in\R}$ be a directional flows
for a  $\Z$-cover $(\widetilde{M},\widetilde{\omega})$ of
$(M,\omega)$ such that the flow $({\varphi}^\theta_t)_{t\in\R}$ on
$M$ is ergodic. Let $I\subset M\setminus\Sigma$ be an interval
transversal to the direction $\theta$ with no self-intersections.
The Poincar\'e return map $T:I\to I$ is a minimal ergodic IET (eif
$({\varphi}^\theta_t)_{t\in\R}$ is ergodic), whose numerical data
will be denoted by
$(\pi,\lambda)\in\mathcal{S}_{\mathcal{A}}\times\R_+^{\mathcal{A}}$
(see for example \cite{ViB, YoLN}). Let $\tau:I\to\R_+$ be the
function which assigns to $x \in I$ the first return time
$\tau(x)$ of $x$ to $I$ under the flow. The function $\tau$ is
constant and equal to some $\tau_\alpha$ on  each exchanged
interval $I_\alpha$. The flow $(\varphi^{\theta}_t)_{t\in\R}$ is
hence measure-theoretically isomorphic to the special flow built
over the IET $T:I\to I$ and under the roof function
$\tau:I\to\R_+$. For every $\alpha\in\mathcal{A}$ we will denote
by $\gamma_\alpha\in H_1(M,\Z)$ the homology class of any loop
$v_x$ formed by the segment of orbit for
$(\varphi^\theta_t)_{t\in\R}$ starting at any $x\in\Int I_\alpha$
and ending at $Tx$ together with the segment of $I$ that joins
$Tx$ and $x$, that we will denote by $[Tx,x]$.

Let us now define a cross-section for the flow
$(\widetilde{\varphi}^\theta_t)_{t\in\R}$ and describe the
corresponding Poincar{\'e} map. Let $\widetilde{I}$ be the
preimage of the interval $I$ via the covering map
$p:\widetilde{M}\to M$. Fix $I_0\subset \widetilde{I}$ a connected
component of $\widetilde{I}$. Then $p|_{I_0}:I_0\to I$ is a
homomorphism and $\widetilde{I}$ is homeomorphic to $I\times\Z$ by
the map
\begin{equation}\label{rho}
I\times\Z\ni (x,n)\mapsto \varrho(x,n):= n\cdot
(p|_{I_0})^{-1}(x)\in \widetilde{I}.
\end{equation}
Denote
by $\widetilde{T}:\widetilde{I}\to \widetilde{I}$ the the
Poincar\'e return map to $\widetilde{I}$ for the flow
$(\widetilde{\varphi}^\theta_t)_{t\in\R}$.

%Recall that $\langle \, ,\, \rangle :H_1(M,\Z)\times
%H_1(M,\Z)\to\Z$ gives the intersection number.
\begin{lemma}\label{lem_flow_auto}
Suppose that
$(\widetilde{M},\widetilde{\omega})=(\widetilde{M}_\gamma,\widetilde{\omega}_\gamma)$
for some $\gamma\in H_1(M,\Z)$ is a  $\Z$-cover. Then the
Poincar\'e  return map $\widetilde{T}$ is isomorphic (via the map
$\varrho$ given in (\ref{rho})) to a skew product
$T_\psi:I\times\Z\to I\times\Z$ of the form
$T_\psi(x,n)=(Tx,n+\psi(x))$, where $\psi=\psi_\gamma:I\to\Z$ is a
piecewise constant function given by
\[\psi_\gamma(x)=\langle \gamma,\gamma_\alpha\rangle  \quad\text{ if }\quad x\in I_\alpha\quad\text{ for each }\alpha\in\mathcal{A}\]
and $T$ and $\gamma_\alpha$ for $\alpha\in\mathcal{A}$ are as
above.
\end{lemma}

\begin{proof}
Let us first remark that
\begin{equation}\label{prrho}
p(\varrho(x,n))=x\text{ and }m\cdot
\varrho(x,n)=\varrho(x,m+n)\text{ for all }x\in I,\;m,n\in\Z.
\end{equation}
Moreover, if $\varrho(x,n), \varrho(x',n')\in\widetilde{I}$ are
joined by a curve in $\widetilde{I}$ then the points belong to the
same connected component of $\widetilde{I}$, hence $n=n'$.
Fix $(x,n)\in\Int I_\alpha\times\Z$ and denote by $v_{x,n}$ the
lift of the loop $v_x$ which starts from the point
$\varrho(x,n)\in\widetilde{I}$. Setting
$\varrho(x,n_e)\in\widetilde{I}$ by its endpoint, by
\eqref{lift_curve} and \eqref{prrho}, we have
\[
\varrho(x,n_e)=\langle  \gamma, [v_x]\rangle \cdot \varrho(x,n)=
\langle \gamma, \gamma_\alpha\rangle  \cdot
\varrho(x,n)=\varrho\left(x,n+\langle \gamma,\gamma_\alpha\rangle
\right),
\]
so $n_e=n+\langle \gamma, \gamma_\alpha\rangle $. Since $v_{x,n}$
is a lift of the curve formed by the segment of orbit for
$(\varphi^\theta_t)_{t\in\R}$ starting at $x\in\Int I_\alpha$ and
ending at $Tx$ together with the segment of $I$ that joins $Tx$
and $x$, $v_{x,n}$ is formed by the segment of orbit for
$(\widetilde{\varphi}^\theta_t)_{t\in\R}$ starting at
$\varrho(x,n)\in\widetilde{I}$ and ending at
$\widetilde{T}\varrho(x,n)$ together with a curve in
$\widetilde{I}$ that joins $\widetilde{T}\varrho(x,n)$ and
$\varrho(x,n_e)$. As $p(\widetilde{T}\varrho(x,n))=Tx$ and the
points $\widetilde{T}\varrho(x,n)$ and $\varrho(x,n_e)$ belong to
the same connected component of $\widetilde{I}$, it follows that
\[\widetilde{T}\varrho(x,n)=\varrho(Tx,n_e)=\varrho\left(Tx,n+\langle  \gamma, \gamma_\alpha\rangle \right),\]
which completes the proof.
\end{proof}

\begin{remark}\label{redtoskew}
The ergodicity of the flow
$(\widetilde{\varphi}^\theta_t)_{t\in\R}$ on
$(\widetilde{M}_\gamma,\widetilde{\omega}_\gamma)$ is equivalent
to the ergodicity of its Poincar{\'e} map $\widetilde{T}$ and
thus,  by Lemma \ref{lem_flow_auto}, it is equivalent to the
ergodicity of the skew product $T_{\psi_\gamma}:I\times \Z\to
I\times\Z$.
\end{remark}

We now recall some properties of this reduction for a special
choice of the section $I$, which will be useful in
\S\ref{Fubini:sec}. For simplicity let  $\theta = \pi/2$ and
assume in addition that the vertical flow
$(\widetilde{\varphi}^v_t)_{t\in\R}$  has no vertical saddle
connections, i.e.~none of its trajectory joins two points of
$\Sigma$, and that the interval $I$ is horizontal and it is chosen
so that one endpoint belongs to the singularity set $\Sigma$ and
the other belongs to an incoming or outgoing separatix, that is to
a trajectory which ends or begins at a point of $\Sigma$. In this
case the IET $T$ has the minimal possible number of exchanged
intervals and the special flow representation is known as
\emph{zippered rectangle} (see \cite{ViB} or \cite{YoLN} for more
details). Recall that each discontinuity of $T$ belongs to an
incoming separatrix (and, by choice, also the endpoints of $I$
belong to separatrices). For each $\alpha \in \mathcal{A}$, let
$\sigma_{l,\alpha}\in\Sigma$ (respectively
$\sigma_{r,\alpha}\in\Sigma$) be the singularity of the separatrix
through the left (right) endpoint of $I_{\alpha}$.

While homology classes  $\{\gamma_\alpha:\alpha\in\mathcal{A}\}$
defined at the beginning of this \S\ref{reduction:sec} generate
the homology $H_1(M, \Z)$ (Lemma 2.17, \S 2.9 in \cite{ViB}), one
can construct a base of the relative homology $H_1(M,\Sigma,\Z)$
as follows. For each $\alpha \in \mathcal{A}$ denote by
$\xi_\alpha\in H_1(M,\Sigma,\Z)$ the relative homology class of
the path which joins $\sigma_{l,\alpha}$ to $\sigma_{r,\alpha}$,
obtained juxtaposing the segment of separatrix starting from
$\sigma_{l,\alpha}$ up to the left endpoint of $I_\alpha$, the
interval $I_\alpha$, and the segment of  separatrix starting from
the right endpoint of $I_\alpha$ and ending at
$\sigma_{r,\alpha}$. Then  $\{\xi_\alpha:\alpha\in\mathcal{A}\}$
establishes a basis of the relative homology $H_1(M,\Sigma, \Z)$
(see \cite{YoLN}). This basis allows us to explicitly compute the
vectors $(\lambda_{\alpha})_{\alpha \in \mathcal{A}}$ and
$(w_{\alpha})_{\alpha \in \mathcal{A}}$ defining $T$ and the
return times $(\tau_{\alpha})_{\alpha \in \mathcal{A}}$ as follows
(see \cite{ViB} or \cite{YoLN}): \be\label{parameters}
\lambda_\alpha=\int_{\xi_\alpha}\Re\omega, \qquad
w_\alpha=\int_{\gamma_\alpha}\Re\omega, \qquad
\tau_\alpha=\int_{\gamma_\alpha}\Im\omega \qquad \text{ for all }
\quad \alpha\in \mathcal{A}. \ee
%\begin{lemma}\label{remark_wal}
%For any $\alpha \in \mathcal{A}$ we have $\lambda_\alpha=\int_{\xi_\alpha}\Re\omega$,
%$w_\alpha=\int_{\gamma_\alpha}\Re\omega$, and  $\tau_\alpha=\int_{\gamma_\alpha}\Im\omega$, $\alpha\in
%\mathcal{A}$.
%\end{lemma}

\section{Essential values of cocycles}\label{essentialvalues:sec}

We give here a brief overview of the tools needed to prove the
non-ergodicity of the skew product $T_\psi$ (see
Section~\ref{reduction:sec}) and describe its ergodic components.
For further background material concerning skew products and
infinite measure-preserving dynamical systems we refer the reader
to \cite{Aa} and \cite{Sch}.

\subsection{Cocycles for transformations and essential values.}
Given  an ergodic automorphism $T$ of standard probability space
$\xbm$,  a locally compact abelian second countable group $G$ and
a cocycle $\psi:X \rightarrow G$ for $T$, consider the
skew-product extension $T_\psi: (X\times
G,\mathcal{B}\times\mathcal{B}_G,\mu\times m_G) \rightarrow
(X\times G,\mathcal{B}\times\mathcal{B}_G,\mu\times m_G)$
($\mathcal{B}_G$ is the Borel $\sigma$-algebra on $G$) given by $
T_{\psi}(x,y)=(Tx,y+{\psi}(x))$.

Two cocycles $\psi_1,\psi_2: X\to G$ for $T$ are called {\em
cohomologous} if there exists a measurable function $g:X\to G$
(called the {\em transfer function}) such that
$\psi_1=\psi_2+g-g\circ T$. Then the corresponding skew products
$T_{\psi_1}$ and $T_{\psi_2}$ are measure-theoretically isomorphic
via the map $(x,y)\mapsto(x,y+g(x))$. A cocycle $\psi:X\to \R$ is
a {\em coboundary} if it is cohomologous to the zero cocycle.

Denote by $\overline{G}$ the one point compactification of the
group $G$. An element $g\in \overline{G}$ is said to be an {\em
essential value} of $\psi$, if for each open neighborhood $V_g$ of
$g$ in $\overline{G}$ and an arbitrary set $B\in\mathcal{B}$,
$\mu(B)>0$, there exists  $n\in\Z$ such that
\begin{eqnarray}
\mu(B\cap T^{-n}B\cap\{x\in X:\psi^{(n)}(x)\in V_g\})>0.
\label{val-ess}
\end{eqnarray}
The set of essential values of $\psi$ will be denoted by
$\overline{E}_G(\psi)$ and put
${E}_G(\psi)=G\cap\overline{E}_G(\psi)$. Then ${E}_G(\psi)$ is a
closed subgroup of $G$.

A cocycle $\psi: X \to G$ is recurrent  if for each open
neighborhood $V_0$ of $0$, (\ref{val-ess}) holds for some $n\neq
0$. This is equivalent to the recurrence of the skew product
$T_\psi$ (cf. \cite{Sch}). In the particular case $G \subset \R$
and $\psi:X\to G$ integrable we have that  the recurrence of $\psi$ is
equivalent to $\int_X\psi\,d\mu=0$.

We recall below some properties of $\overline{E}_G(\psi)$ (see
\cite{Sch}).

\begin{proposition}\label{basicessentialvalues}
If $H$ is a closed subgroup of $G$ and $\psi:X\to H$ then
$E_G(\psi)=E_H(\psi)\subset H$. If $\psi_1,\psi_2:X\to G$  are
cohomologous then $\overline{E}_G(\psi_1)=\overline{E}_G(\psi_2)$.
\end{proposition}

Consider the quotient cocycle $\psi^*:X\to G/{E}(\psi)$ given by
$\psi^*(x)=\psi(x)+{E}(\psi)$. Then
$E_{G/{E}(\psi)}(\psi^*)=\{0\}$. The cocycle $\psi:X\to G$ is
called {\em regular} if $\overline{E}_{G/{E}(\psi)}(\psi^*)=\{0\}$
and {\em non--regular} if
$\overline{E}_{G/{E}(\psi)}(\psi^*)=\{0,\infty\}$. Recall that if
$\psi:X\to G$ is regular then it is cohomologous to a cocycle
$\psi_0:X\to E(\psi)$ such that $E(\psi_0)=E(\psi)$.

The following classical Proposition gives a criterion  to prove
ergodicity and check if a cocycle is a coboundary using essential
values.
\begin{proposition}[see \cite{Sch}]\label{proposition:schmidt}
Suppose that $T:(X,\mu)\to(X,\mu)$ is an ergodic automorphism and
$\psi:X\to G$ be a cocycle for $T$. The skew product
$T_{\psi}:X\times G\to X\times G$ is ergodic if and only if
$E_G(\psi)=G$. The cocycle is a coboundary if and only if
$\overline{E}_G(\psi)=\{0\}$.
\end{proposition}

We also recall the following characterization of coboundaries.
\begin{proposition}[see \cite{Be-dJu-Le-Ro}]\label{proposition:cocycle}
If $T:(X,\mu)\to(X,\mu)$ is an ergodic automorphism then the
cocycle $\psi:X\to G$  for $T$ is a coboundary if and only if the
skew product $T_\psi:X\times G\to X\times G $  has an invariant
set of positive finite measure.
\end{proposition}

\subsection{Ergodic decomposition and Mackey action}\label{Mackey}
If the skew product $T_\psi:X\times G\to X\times G$ is not ergodic
then the structure of its ergodic components (defined below) can be studied by
looking at properties of the so called Mackey action.

Let $(\tau_g)_{g\in G}$ denote the $G$-action on $(X\times
G,\mathcal{B}\times\mathcal{B}_G,\mu\times m_G)$ given by
$\tau_g(x, h) = (x, h+g)$. Then $(\tau_g)_{g\in G}$ commutes with
the skew product $T_\psi$. Fix a probability Borel measure $m$ on
$G$ equivalent to the Haar measure $m_G$. Then the probability
measure $\mu\times m$ is quasi-invariant under $T_\psi$ and
$(\tau_g)_{g\in G}$, i.e.~$(T_\psi)_*(\mu\times m)$ and
$(\tau_g)_*(\mu\times m) $ for any $g\in G$ are equivalent to
$\mu\times m$ (or, in other words, $T_\psi$ and $(\tau_g)_{g\in
G}$ are non-singular actions on $(X\times
G,\mathcal{B}\times\mathcal{B}_G,\mu\times m)$). Denote by
$\mathcal{I}_\psi\subset\mathcal{B}\times\mathcal{B}_G$ the
$\sigma$-algebra  of $T_\psi$-invariant subsets. Since $(X\times
G,\mathcal{B}\times\mathcal{B}_G,\mu\times m)$ is a standard
probability Borel space, the quotient space $((X\times
G)/\mathcal{I}_\psi,\mathcal{I}_\psi,\mu\times
m|_{\mathcal{I}_\psi})$ is well-defined (and is also standard).
This space is called the \emph{space of ergodic components} and it
will be denoted by $(Y, \mathcal{C}, \nu)$. Since $(\tau_g)_{g\in
G}$ preserves $\mathcal{I}_\psi$ it also acts on $(Y, \mathcal{C},
\nu)$. This non-singular $G$-action is called the \emph{Mackey
action} (and is denoted by $(\tau^\psi_g)_{g\in G}$) associated to
the skew product $T_\psi$, and it is always ergodic. Moreover,
there exists a measurable map $Y\ni y\mapsto \overline{\mu}_y$
taking values in the space of probability measures on $(X\times
G,\mathcal{B}\times\mathcal{B}_G)$ such that
\begin{itemize}
\item $\mu\times m=\int_{Y}\overline{\mu}_y\,d\nu(y)$;
\item $\overline{\mu}_y$ is quasi-invariant and ergodic under $T_\psi$ for
$\nu$-a.e. $y\in Y$;
\item $\overline{\mu}_y$ is equivalent to
a $\sigma$-finite measure ${\mu}_y$ invariant under $T_\psi$.
\end{itemize}
Then $T_\psi$ on $(X\times
G,\mathcal{B}\times\mathcal{B}_G,{\mu}_y)$ for $y\in Y$ are called
\emph{ergodic components} of $T_\psi$.
\begin{theorem}[\cite{Sch,Zi}]\label{thm_ergdeco}
Suppose that $T:(X,\mu)\to(X,\mu)$ is ergodic and let $\psi:X\to
G$ be a cocycle. Then
\begin{itemize}
\item[(i)] $\psi$ is recurrent if and only if
the measure $\mu_y$ is continuous for $\nu$-a.e. $y\in Y$;
\item[(ii)]  $\psi$ is non-recurrent if and only if
 $\mu_y$ is purely atomic for $\nu$-a.e. $y\in Y$;
\item[(iii)] $\psi$ is regular if and only if the Mackey
action $(\tau^\psi_g)_{g\in G}$ is strictly transitive, i.e.\ the
measure $\nu$ is supported on a single orbit of
$(\tau^\psi_g)_{g\in G}$.
\end{itemize}
\end{theorem}
If $\psi$ is not recurrent then almost every ergodic component
$T_\psi:(X\times G, \mu_y)\to (X\times G, \mu_y)$ is trivial, i.e.
it is strictly transitive.

If $\psi$ is regular then the structure of ergodic components is
trivial, i.e. if we fix  one ergodic component then every other
ergodic component is the image of the fixed component by a
transformation $\tau_g$. In particular, all ergodic components are
isomorphic.

As a immediate consequence of Theorem~\ref{thm_ergdeco} we obtain
that if a cocycle is recurrent and non-regular then the structure
of ergodic components of the skew product and the dynamics inside
ergodic components are highly non-trivial.

\begin{corollary}\label{corol_nonregul}
Let $T:(X,\mu)\to(X,\mu)$ be an ergodic automorphism and
$\psi:X\to\Z$ is a recurrent non-regular cocycle. Then the
measures $\nu$ and $\mu_y$ for $\nu$-a.e. $y\in Y$ are continuous.
In particular, the skew product $T_\psi$ has uncountably  many
ergodic components and almost every ergodic component is not
supported by a countable set.
\end{corollary}
\begin{proof} Since the measure $\nu$ is ergodic for the
the Mackey $\Z$-action, it is either continuous or purely
discrete. If $\nu$ is discrete then, by ergodicity, $\nu$ is
supported by a single orbit,  in contradiction with $(iii)$.
Consequently, $\nu$ is continuous. The continuity of almost every
measure $\mu_y$ follows directly from $(i)$. The second part of
the corollary is a direct consequence of the continuity of these
measures.
\end{proof}

\begin{remark}
Let $(\widetilde{\varphi}^\theta_t)_{t\in\R}$ be a directional
flows on a $\Z$-cover $(\widetilde{M},\widetilde{\omega})$ of
$(M,\omega)$ such that the flow $({\varphi}^\theta_t)_{t\in\R}$ on
$(M,\omega)$ is ergodic. Suppose that its Poincar\'e return map is
isomorphic to a skew product $T_\psi:I\times\Z\to I\times \Z$ (as
in Section~\ref{reduction:sec}) and the cocycle $\psi$ is
recurrent and non-regular. Then the flow
$(\widetilde{\varphi}^\theta_t)_{t\in\R}$ is not ergodic and it
has uncountably many ergodic components and almost every such
ergodic component is not supported on a single orbit of the flow.
\end{remark}

\subsection{Cocycles for flows}
Let $(\varphi_t)_{t\in\R}$ be a Borel flow on a standard
probability Borel space $\xbm$. A \emph{cocycle} for
the flow $(\varphi_t)_{t\in\R}$ is a Borel function
$F:\R\times X\to\R$ such that
\[F(t+s,x)=F(t,\varphi_sx)+F(s,x)\text{ for all }s,t\in\R\text{ and }x\in X.\]

\begin{definition}Two cocycles $F_1,F_2:\R\times X\to\R$ are called
{\em cohomologous} if there exists a Borel function $u:X\to\R$ and
a Borel $(\varphi_t)_{t\in\R}$-invariant subset $X_0\subset X$
with $\mu(X_0)=1$ such that
\[F_2(t,x)=F_1(t,x)+u(x)-u(\varphi_tx)\text{ for all }x\in X_0\text{ and }t\in\R.\]
A cocycle $F:\R\times X\to\R$ is said to be a {\em cocycle} if it
is cohomologous to the zero cocycle.
\end{definition}

\begin{remark}\label{coboundarycondition}
Let us recall a simple condition on a cocycle $F$ guaranteeing
that it is a coboundary: if there exist a Borel
$(\varphi_t)_{t\in\R}$-invariant subset $X_0\subset X$ with
$\mu(X_0)=1$ such that the map $\R_+\ni t\mapsto F(t,x)\in\R$ is
continuous and bounded for every $x\in X_0$ then $F$ is a
coboundary. Moreover, the transfer function $u:X\to\R$ is given by
\[u(x):=\limsup_{s\to+\infty}F(s,x)=\limsup_{s\in\Q,\,s\to+\infty}F(s,x)\quad\text{ for }\quad x\in X_0.\]
Indeed, for every $t\geq 0$ and $x\in X_0$ we have
\[u(\varphi_tx)=\limsup_{s\to+\infty}F(s,\varphi_tx)=\limsup_{s\to+\infty}F(s+t,x)-F(t,x)=u(x)-F(t,x).\]
\end{remark}

\subsubsection*{Cocycles for translation flows.}
Let $(M,\omega)$ be a compact translation surface and  let
$\theta\in S^1$. For every $x\in M\setminus \Sigma$ denote by
$I^\theta(x)\subset \R$ the maximal open interval for which
$\varphi^\theta_tx$ is well defined whenever $t\in
I^\theta(x)\subset \R$. If $x\in M_\theta$ then $I^\theta(x)=\R$.
For any smooth bounded function $f:M\setminus\Sigma\to\R$ let
\begin{equation}\label{Fdef}
F^\theta_f(t,x):=\int_0^tf(\varphi^\theta_sx)\,ds\text{ if }t\in I^\theta(x).
\end{equation}
Thus $F^\theta_f$ is well defined on $\R\times M_\theta$ and it is
a cocycle for the directional flow $(\varphi^\theta_t)_{t\in\R}$
considered on $(M_\theta,\nu_\omega)$.

Assume that the directional flow $(\varphi^\theta_t)_{t\in\R}$  is
minimal and let $I_\theta \subset M$ be an interval transverse to
$(\varphi^\theta_t)_{t\in\R}$. The first return (Poincar{\'e}) map
of $(\varphi^\theta_t)_{t\in\R}$ to $I_\theta $ is an interval
exchange transformation $T_\theta$. Let $\psi_f^\theta:I\to\R$ be the
cocycle for $T_\theta$ defined as follows. Let $\tau: I_\theta \to
\R^+$ be the piecewise constant function which gives the first
return time $\tau(x)$ of $x $ to $I_\theta$ under the flow
$(\varphi^\theta_t)_{t\in\R}$. Then
\[\psi^\theta_f(x)= F^\theta_f(\tau(x) , x) =
\int_0^{\tau(x)}f(\varphi^\theta_s x)\,ds,  \qquad x \in I_\theta.\]
The following standard equivalence holds (see for example \cite{Fr-Ulc1}).
\begin{lemma}\label{equivalentcoboundaries}
The cocycle  $F^\theta_f$ is a coboundary for the flow
$(\varphi^\theta_t)_{t\in\R}$ if and only if the cocycle
$\psi^\theta_f$ is a coboundary for the interval exchange
transformation $T_\theta$.
\end{lemma}

\begin{remark}\label{remark:cocycle-hom-cohom}
For every smooth closed form $\rho\in\Omega^1(M)$ let us consider
the smooth bounded function $f:M\setminus\Sigma\to\R$,
$f=i_{X_{\theta}}\rho$ and let $\psi_\rho:I\to\R$ be the
corresponding cocycle for $T$ defined by
$\psi_\rho(x)=\int_0^{\tau(x)}f(\varphi^\theta_sx)\,ds$. Let
$\gamma:=\mathcal{P}^{-1} [\rho] \in H_1(M,\R)$, where
$\mathcal{P}:H_1(M,\R)\to H^1(M,\R)$ the Poincar\'e duality, see
\eqref{def_duality} for definition. Then, recalling the
definitions of $\gamma_\alpha$, $v_x$ and $[x,Tx]$ in
\S\ref{reduction:sec} and applying \eqref{prop_duality}, for every
$x\in I_\alpha$
\begin{equation*}
%\begin{aligned}
 \langle  \gamma_\alpha ,\gamma \rangle = \int_{\gamma_\alpha} \rho =
 \int_{v_x}\rho =
 \int_{0}^{\tau(x)}i_{X_\theta}\rho(\varphi^\theta_sx)\,ds+\int_{[Tx,x]}\rho
 =\psi_\rho(x)+g(x)-g(Tx),
%\end{aligned}
\end{equation*}
where $g:I\to\R$ is given by $g(x)=\int_{[x_0,x]}\rho$ ($x_0$ is
the left endpoint of the interval $I$). Consequently, denoting by
$\psi_\gamma:I\to\R$ the cocycle $\psi_\gamma(x)=\langle \gamma,
\gamma_\alpha \rangle$ if $x\in I_\alpha$ for
$\alpha\in\mathcal{A}$, we conclude that the cocycle
$\psi_\rho+\psi_\gamma$ is a coboundary.
\end{remark}

\section{The Teichm\"uller flow and the Kontsevich-Zorich cocycle}\label{Teich:sec}
Given  a connected oriented surface $M$
%of genus $g\geq 1$
and a discrete countable set  $\Sigma\subset M$, denote by
$\operatorname{Diff}^+(M,\Sigma)$ the group of
orientation-preserving homeomorphisms of $M$ preserving $\Sigma$.
Denote by $\operatorname{Diff}_0^+(M,\Sigma)$ the subgroup of
elements $\operatorname{Diff}^+(M,\Sigma)$ which are isotopic to
the identity. Let us denote by
$\Gamma(M,\Sigma):=\operatorname{Diff}^+(M,\Sigma)/\operatorname{Diff}_0^+(M,\Sigma)$
 the {\em mapping-class} group.
% By a {\em translation surface} $(M,\omega)$ we mean a complex
%structure on $M$ and a nonzero holomorphic $1$-form $\omega$ with
%respect to this complex structure. Let
%$\Sigma=\Sigma_\omega\subset M$ be the set of zeros of $\omega$.
%For every $\theta\in S^1$ denote by $X_\theta=S^{\omega}_\theta$
%the directional vector field on $M\setminus\Sigma$ defined by
%$i_{X_\theta}=\theta$. Then the corresponding directional (local)
%flow
%$(\varphi^{\theta}_t)_{t\in\R}=(\varphi^{\omega,\theta}_t)_{t\in\R}$
%on $M\setminus\Sigma$ preserves the volume form
%$\nu_{\omega}=\frac{i}{2}\omega\wedge\overline{\omega}=\Re(\omega)\wedge\Im(\omega)$.
%We will denote by $M_\theta$ the set of regular points for the
%directional flow $(\varphi^{\theta}_t)_{t\in\R}$, i.e. the set  of point for which the orbit of the
%flow may be defined for all times $t\in\R$. Then $M_\theta$ is a
%Borel subset of $M$ with $\nu_\omega(M\setminus M_\theta)=0$ and
%$(\varphi^{\theta}_t)_{t\in\R}$ restricted to $M_\theta$ is a well
%defined Borel flow. We denote by $T$ and $S$ the vertical and
%horizontal tangent vector field on $M\setminus\Sigma$, i.e.
%$i_{T}\omega=i$ and $i_S\omega=1$.
We will denote by $\mathcal{Q}(M)$ (respectively
$\mathcal{Q}^{(1)}(M)$ ) the {\em Teichm\"uller space of Abelian
differentials } (respectively of unit area Abelian differentials),
that is the space of orbits of the natural action of
$\operatorname{Diff}_0^+(M,\emptyset)$ on the space of all Abelian
differentials on $M$ (respectively, the ones with total area
$A(\omega)=\int_M\Re(\omega)\wedge \Im(\omega) =1$). We will
denote by $\mathcal{M}(M)$ ($\mathcal{M}^{(1)}(M)$) the {\em
moduli space of (unit area) Abelian differentials}, that is the
space of orbits of the natural action of
$\operatorname{Diff}^+(M,\emptyset)$ on the space of (unit area)
Abelian differentials on $M$.
%We will denote by $\mathcal{M}^{(1)}(M)$ the moduli space of Abelian differentials \emph{of area one}, that is such that $A(\omega)=1$.
Thus  $\mathcal{M}(M)=\mathcal{Q}(M)/\Gamma(M,\emptyset)$ and
$\mathcal{M}(M)^{(1)}=\mathcal{Q}^{(1)}(M)/\Gamma(M,\emptyset)$.

The group $SL(2,\R)$ acts naturally on  $\mathcal{Q}^{(1)}(M)$ and
$\mathcal{M}^{(1)}(M)$ by postcomposition on the charts defined by
local primitives of the holomorphic $1$-form. We will denote by
$g\cdot \omega$ the new Abelian differential obtained acting by $g
\in SL(2,\R)$ on $\omega$. The {\em Teichm\"uller flow}
$(G_t)_{t\in\R}$ is the restriction of this action to the diagonal
subgroup $(\operatorname{diag}(e^t,e^{-t}))_{t\in\R}$ of
$SL(2,\R)$ on $\mathcal{Q}^{(1)}(M)$ and $\mathcal{M}^{(1)}(M)$.
%Level sets of the area form $A(\omega)= \frac{1}{2} \int \Re \omega \wedge \Im \omega$ are invariant under this action. We will denote by $\mathcal{M}^{(1)}(M)^1$ the level set $A^{-1}(1)$.
Remark that the $SL(2,\R)$ action preserves the zeros of $\omega$
and their degrees.

Let  $M$ be compact and of genus $g$ and let $\kappa$ be the
number of zeros of $\omega$.  If $k_i$, $1\leq i\leq \kappa$ is
the degrees of each zero, one has $2g-2 = \sum_{i=1}^\kappa k_i$.
Let us denote by $\mathcal{H} (\underline{k}) =\mathcal{H} (k_1,
\dots, k_\kappa) $ the \emph{stratum} consisting of all $(M,
\omega)$ such that $\omega$ has $\kappa$ zeros of degrees $k_1,
\dots, k_\kappa$.  Each stratum is invariant under the $SL(2,\R)$
action and the connected components of this action were classified
in \cite{Ko-Zo}. Let $\mathcal{H}^{(1)} (\underline{k}) =
\mathcal{H} (\underline{k}) \cap \mathcal{M}^{(1)}(M)$. Each
stratum $\mathcal{H}^{(1)}=\mathcal{H}^{(1)} (\underline{k}) $
carries a canonical $SL(2,\R)$-invariant measure
$\mu_\mathcal{H}^{(1)}$ that can defined as follows. Let
$\{\gamma_1, \dots, \gamma_{n}\}$ be a basis of the
\emph{relative} homology $H_1(M, \Sigma, \Z)$. Remark that for
each $\gamma_i$, $\int_{\gamma_i} \omega \in \mathbb{C} \approx
\R^2$. The \emph{relative periods} $(\int_{\gamma_1}\omega, \dots,
\int_{\gamma_1}\omega) \in \R^{2n}$ are local coordinates on the
stratum $\mathcal{H} (\underline{k}) $. Consider the pull-back by
the relative periods of the Lebesgue measure on $\R^{2n}$. This
measure induces a conditional measure on the hypersurface
$\mathcal{H}^{(1)} (\underline{k})\subset \mathcal{H}
(\underline{k})$.  Since this measure is finite (see \cite{Ma:IET,
Ve1}), we can  renormalize it to get a probability measure that we
will denote by $\mu_{\mathcal{H}}^{(1)}$. The measure
$\mu_{\mathcal{H}}^{(1)}$ is $SL(2,\R)$-invariant and ergodic for
the Teichm\"uller flow.

\subsubsection*{The Kontsevich-Zorich cocycle.}
Assume that $M$ is compact.  The {\em
Kontsevich-Zorich cocycle} $(G^{KZ}_t)_{t\in\R}$ is the quotient
of the trivial cocycle
\[G_t\times\operatorname{Id}:\mathcal{Q}^{(1)}(M)\times H^1(M,\R)\to\mathcal{Q}^{(1)}(M)\times H^1(M,\R)\]
by the action of the mapping-class group
$\Gamma(M):=\Gamma(M,\emptyset)$. The mapping class group acts on
the fiber $H^1(M,\R)$ by pullback. The cocycle
$(G^{KZ}_t)_{t\in\R}$ acts on the cohomology vector bundle
\[\mathcal{H}^1(M,\R)=(\mathcal{Q}^{(1)}(M)\times H^1(M,\R))/\Gamma(M)\]
(known as \emph{Hodge bundle}) over the Teichm\"uller flow
$(G_t)_{t \in \R}$ on the moduli space
$\mathcal{M}^{(1)}(M)=\mathcal{Q}^{(1)}(M)/\Gamma(M)$. We will
denote by $H^1(M_{\omega},\R)$ the fiber at $\omega$. Clearly
$H^1(M_{\omega},\R)=H^1(M,\R)$.   The space $H^1(M,\R)$ is endowed
with symplectic (intersaction) form given by
\[\langle c_1,c_2\rangle := %=%[c_1]\wedge_{\omega} [c_2]
%\frac{1}{A(\omega)}
\int_M c_1\wedge c_2
\text{ for }c_1,c_2\in H^1(M,\R).\] %where
%$A(\omega):=\int_M\Re(\omega)\wedge \Im(\omega)$.
This symplectic structure  is preserved by the action of the
mapping-class group and
%the intersection form $\wedge$ $\omega$
hence is invariant under the action of $SL(2,\R)$.

Denote by $\mathcal{P}:H_1(M,\R)\to H^1(M,\R)$ the Poincar\'e
duality, i.e.
\begin{equation}\label{def_duality}
\mathcal{P}\sigma=c\quad\text{ iff } \quad\int_\sigma c'=\langle
c,c'\rangle \quad\text{ for all }\quad c'\in
H^1(M,\R).
\end{equation}
Since the Poincar{\'e} duality
$\mathcal{P}:H_1(M,\R)\to H^1(M,\R)$ intertwines the intersection
forms $\langle \,\cdot\, , \,\cdot\,\rangle $ on $H_1(M, \R)$ and
$ H^1(M,\R)$ respectively, that is $\langle \sigma,\sigma'\rangle
=\langle \mathcal{P}\sigma,\mathcal{P}\sigma'\rangle $ for all
$\sigma,\sigma' \in H_1(M, \R)$, we have
\begin{equation}\label{prop_duality}
\langle \sigma,\sigma'\rangle =\langle
\mathcal{P}\sigma,\mathcal{P}\sigma'\rangle=\int_\sigma
\mathcal{P}\sigma'\quad\text{ for all }\quad \sigma,\sigma' \in
H_1(M, \R).
\end{equation}

Each fiber  $ H^1 (M_\omega, \R)$ of the vector bundle
$\mathcal{H}^1(M,\R)$ is endowed with a natural norm, called the
\emph{Hodge norm}, defined as follows (see \cite{For-dev}).  Given
a cohomology class $c\in H^1(M,\R)$, there exists a unique
holomophic one-form $\eta$, holomorphic with respect to the
complex structure induced by $\omega$, such that $c = [\Re \eta]$.
%closed real form $*\rho$, known as the Hodge-star dual, such that $\rho = \Re \eta$ and $*\rho = \Im \eta$ for a
The  Hodge norm of $\| c \|_\omega$ is then defined as
$\big(\frac{i}{2}\int_M \eta \wedge \overline{\eta}\big)^{1/2}$.

\subsubsection*{Lyapunov exponents and Oseledets splittings.}
Let $\mu$ be a probability measure on $\mathcal{M}^{(1)}(M)$ which
is invariant for the Teichm\"uller flow and ergodic.  Since the
Hodge norm of the Kontsevich-Zorich cocycle at time $t$ is
constant and equal to $e^t$ (see \cite{For-dev}) and $\mu$ is a
probability measure, the  Kontsevich-Zorich cocycle is
log-integrable with respect to $\mu$.
%, that is \[  \int_{\mathcal{M}^{(1)}(M)} \sum_{-1\leq t \leq 1 } \frac{\log^+ \| g_t^* \omega \|}{t} \ud \mu(\omega) < + \infty. \]
Thus,  it follows from Osedelets' theorem that there exists
Lyapunov exponents with respect to the measure $\mu$.   As the
action of the  Kontsevich-Zorich cocycle is symplectic, its
Lyapunov exponents with respect to the measure $\mu$ are:
\[1=\lambda^\mu_1>\lambda^\mu_2\geq\ldots\geq\lambda^\mu_g\geq-
\lambda^\mu_g\geq\ldots\geq-\lambda^\mu_2>-\lambda^\mu_1=-1,\] the
inequality $\lambda^\mu_1>\lambda^\mu_2$ was proven in
\cite{For-dev}.
 The measure $\mu$ is called {\em KZ-hyperbolic} if
$\lambda_g^\mu>0$. When $g=2$, it follows from a result by
Bainbridge\footnote{In \cite{Ba} Bainbridge actually computes the
explicit value of $\lambda_2$ for any $\mu$ probability measure
invariant  for the Teichm\"uller flow in the genus  two strata
$\mathcal{H}(2)$ and $\mathcal{H}(1,1)$. The positivity of the
second exponent for $g=2$ also follows by the thesis of Aulicino
\cite{Au}, in which it is shown that no $SL(2,\mathbb{R})$-orbit
in $\mathcal{H}(1,1)$ or $\mathcal{H}(2)$ has completely
degenerate spectrum. } that:
%, which will play a crucial role in the proof of our main result:

\begin{theorem}[Bainbridge]\label{Bain}
If $M$ is surface with genus $g=2$ then for any  probability
measure $\mu$ on $\mathcal{M}^{(1)}(M)$ which is invariant for the
Teichm\"uller flow and ergodic its second Lyapunov exponent
$\lambda_2$ is strictly positive. Thus, $\mu$ is KZ-hyperbolic.
\end{theorem}
If a measure $\mu$ is KZ-hyperbolic, by Oseledets' theorem, for
$\mu$-almost every $\omega\in\mathcal{M}^{(1)}(M)$ (such points
will be called Oseledets regular points), the fiber
$H^1(M_\omega,\R)$  of the bundle $\mathcal{H}^1(M,\R)$ at
$\omega$ has a  direct splitting
\[ H^1(M_\omega,\R)=E_\omega^+(M,\R)\oplus E_\omega^-(M,\R),\]
where the \emph{unstable space} $E_\omega^+(M,\R)$  (the
\emph{stable space}  $E_\omega^-(M,\R)$ resp.) is the subspace
of cohomology classes with positive (negative resp.) Lyapunov
exponents, i.e.
\begin{align}
E_\omega^+(M,\R)&=\Big\{c\in
H^1(M_\omega,\R):\lim_{t\to+\infty}\frac{1}{t}\log\|c\|_{G_{-t}\omega}<0\Big\},
\label{stabledef}
\\
E_\omega^-(M,\R)&=\Big\{c\in
H^1(M_\omega,\R):\lim_{t\to+\infty}\frac{1}{t}\log\|c\|_{G_{t}\omega}<0\Big\}.\nonumber
\end{align}

Let  $\mu$  be an $SL(2,\R)$-invariant probability measure which
is ergodic for the Teichm\"uller flow and let $\mathscr{L}_\mu $
be the support of $\mu$, which is an $SL(2,\R)$-invariant closed
subset of $\mathcal{M}^{(1)}(M)$. Let $\mathbb{F}$ be a field
(we will deal only with fields $\R$ and $\Q$). Let us consider
vector subbundles $\mathcal{K}^1$ of the cohomology bundle
(respectively, vector subbundles of the homology bundle
$\mathcal{K}_1$) defined over $\mathscr{L}_\mu$. Each  such a
subbundle is determined by a collection of subfibers of the
cohomology (or homology) fibers over $\mathscr{L}_\mu$, that is
$\mathcal{K}^1=\bigcup_{\omega\in\mathscr{L}_\mu}\{\omega\}\times
K^1(\omega)$, where $K^1(\omega)\subset H^1(M_\omega,\mathbb{F})$
is a linear subspace (respectively
$\mathcal{K}_1=\bigcup_{\omega\in\mathscr{L}_\mu}\{\omega\}\times
K_1(\omega)$, where $K_1(\omega)\subset
H_1(M_\omega,\mathbb{F})$).  We will call a subbundle
$\mathcal{K}^1$ ($\mathcal{K}_1$) of this form an \emph{invariant
subbundle} over $\mathscr{L}_\mu$ if:
\begin{itemize}
\item[(i)] $K^1(g\cdot\omega)=K^1(\omega)$ ($K_1(g\cdot\omega)=K_1(\omega)$)
for every $g\in SL(2,\R)$ and $\omega\in\mathscr{L}_{\mu}$;
\item[(ii)] if
$\omega_1,\omega_2\in\mathcal{Q}^{(1)}(M)$ are two representatives of the same point
$\omega_1 \Gamma =\omega_2 \Gamma \in\mathscr{L}_\mu$ and
$\phi\in\Gamma(M)$ is an element of the mapping-class group such
that $\phi^*(\omega_1)=\omega_2$ then
$\phi^*K^1(\omega)=K^1(\omega)$ ($\phi_*K_1(\omega)=K_1(\omega)$).
\end{itemize}
\noindent Moreover, we say that an invariant subbundle $\mathcal{K}^1$
($\mathcal{K}_1$) is (locally) constant if the map $\omega\mapsto
K^1(\omega)$ ($\omega\mapsto K_1(\omega)$) is (locally) constant.

For any cohomological invariant subbundle $\mathcal{K}^1$ with
$K^1(\omega)\subset H^1(M,\R)$ for $\omega \in \mathscr{L}_\mu$
one can consider the Kontsevich-Zorich cocycle $(G^{KZ}_t)_{t\in
\R}$ restricted to the subbundle $\mathcal{K}^1$ over the
Teichm\"uller flow on $\mathscr{L}_{\mu}$. The Lyapunov exponents
of the reduced cocycle $(G^{KZ,\mathcal{K}^1}_t)_{t\in \R}$ with
respect to the measure $\mu$ will be called \emph{the Lyapunov
exponents of the subbundle $\mathcal{K}^1$}.

A splitting $\{ H^1(M_\omega,\mathbb{F})=K^1(\omega)\oplus
K^1_\perp(\omega)$, $\omega\in \mathscr{L}_\mu\}$ (respectively
$\{ H_1(M_\omega,\mathbb{F})=K_1(\omega)\oplus K_1^\perp(\omega)$,
$\omega\in \mathscr{L}_\mu\}$) is called an \emph{orthogonal
invariant splitting} if both corresponding subbundles
$\mathcal{K}^1 = \bigcup_{\omega \in \mathscr{L}_\mu} \{ \omega\}
\times K^1(\omega)  $  and $\mathcal{K}^1_\perp = \bigcup_{\omega
\in \mathscr{L}_\mu} \{ \omega\} \times K^1_\perp(\omega)$
(respectively $\mathcal{K}_1$ and $\mathcal{K}_1^\perp$) are
invariant and $K^1(\omega)$, $K^1_\perp(\omega)$ (respectively
$K_1(\omega)$, $K_1^\perp(\omega)$) are orthogonal with respect to
the symplectic form $\langle \,\cdot\, , \,\cdot\, \rangle $ for
every $\omega\in \mathscr{L}_\mu$.

Let  $\{ H^1(M_\omega,\mathbb{R})=K^1(\omega)\oplus
K^1_\perp(\omega)$, $\omega\in \mathscr{L}_\mu\}$ be an orthogonal
invariant splitting. Since the Poincar{\'e} duality
$\mathcal{P}:H_1(M,\R)\to H^1(M,\R)$ intertwines the intersection
forms $\langle \,\cdot\, , \,\cdot\,\rangle $ on $H_1(M, \R)$ and
$ H^1(M,\R)$ respectively, one also has a dual invariant
orthogonal splitting given fiberwise by
\[H_1(M_\omega,\R)=K_1(\omega)\oplus K_1^\perp(\omega)\text{ with
}K_1(\omega):=\mathcal{P}^{-1}K^1(\omega),\,
K^\perp_1(\omega):=\mathcal{P}^{-1}K^1_\perp(\omega).\] The
Lyapunov exponents of the reduced cocycle
$(G^{KZ,\mathcal{K}^1}_t)_{t\in \R}$ with respect to the measure
$\mu$ will be also called the Lyapunov exponents of
$\mathcal{K}_1$.

For any $\omega\in\mathcal{M}^{1}(M)$  denote by
$H^1_{st}(M_\omega,\R)$ the subspace of $H^1(M,\R)$ generated by
$[\Re(\omega)]$ and $[\Im(\omega)]$. Set
\[H^1_{(0)}(M_\omega,\R):=H^1_{st}(M_\omega,\R)^\perp=\{c\in H^1(M_\omega,\R):\, \forall_{c'\in H^1_{st}(M_\omega,\R)}\, \langle c, c'\rangle =0\}.\]
Then one has the following  orthogonal invariant splitting
\[\{ H^1(M_\omega,\R)=H^1_{st}(M_\omega,\R)\oplus H^1_{(0)}(M_\omega,\R) , \quad \omega\in\mathcal{M}^{(1)}(M)\}, \]
Let $\mathcal{H}_1^{st}$ (where $st$ stays for \emph{standard}) and $\mathcal{H}_1^{(0)}$ (also known as \emph{reduced Hodge bundle}) be the corresponding subbundles.
The Lyapunov exponents of the subbundle $\mathcal{H}_1^{(0)}$ are exactly $\{ \pm \lambda_2^\mu, \dots, \pm \lambda_g^\mu\}$
(see the proof of Corollary~2.2 in \cite{For-dev}).
Correspondingly, one also has also the dual orthogonal invariant splitting
\begin{eqnarray*}\{ H_1(M,\R)&=&H_1^{st}(M_\omega,\R)\oplus H_1^{(0)}(M_\omega,\R), \quad \omega\in\mathcal{M}^{(1)}(M)\}, \ \text{where} \\
H_1^{(0)}(M_\omega,\R)&=&\{\sigma\in H_1(M,\R):\int_\sigma c=0\text{ for all }c\in H^1_{st}(M_\omega,\R)\}; \\
H_1^{st}(M_\omega,\R)&=&\{\sigma\in H_1(M,\R):\langle
\sigma,\sigma'\rangle =0\text{ for all }\sigma\in
H_1^{(0)}(M_\omega,\R)\}.
\end{eqnarray*}

\subsubsection*{Coboundaries and unstable space}
If  $\mu$ is a  KZ-hyperbolic probability measure on
$\mathcal{M}^{(1)}(M)$, on a full measure set of Oseledets regular
$\omega \in \mathcal{M}(M)$ one can relate coboundaries for the
vertical flow with the stable space $E^-_\omega(M,\R)$ of the
Kontsevich-Zorich cocycle as stated in Theorem \ref{cohthm} below.
%Let us denote by $X_v$ and $X_h$ the vertical and the horizontal
%tangent vector field on $M\setminus\Sigma$, given by
%$i_{X_v}\omega=e^{\frac{\pi}{2}i}$ and $i_{X_h}\omega=1$. Then $X_v$ and $X_h$ generate  the vertical and horizontal flows
%$(\varphi^{{v}}_t)_{t\in\R}$ and $(\varphi^{{h}}_t)_{t\in\R}$ respectively.

Recall that given a smooth bounded function
$f:M\setminus\Sigma\to\R$ we denote by $F^\theta_f$ the cocycle
over the directional flow $(\varphi^\theta_t)_{t\in\R}$ given by
\[F^\theta_f(t,x):=\int_0^tf(\varphi^\theta_sx)\,ds, \qquad x \in
M_\theta, \quad t\in \R.\]

\begin{theorem}\label{cohthm}
Let $\mu$ be any  $SL(2,\R)$-invariant probability measure on
$\mathcal{M}^{(1)}(M)$ ergodic for the Teichm\"uller flow.
There exists a set $\mathcal{M}' \subset \mathcal{M}^{(1)}(M)$
with $\mu(\mathcal{M}')=1$, such that any $\omega\in\mathcal{M}'$
is Oseledets regular and for any smooth closed form
$\rho\in\Omega^1(M)$, if $[\rho]\in E_\omega^-(M,\R)$, then the
cocycle $F^{v}_f$ with $f:=i_{X_{v}}\rho$ for the vertical flow
$(\varphi^{{v}}_t)_{t\in\R}$ is a coboundary.

Moreover,  if we assume in addition that $\mu$ is KZ-hyperbolic, we
also have, conversely, that if  $[\rho]\notin E_\omega^-(M,\R)$,
then $F^{v}_f$ is \emph{not} a coboundary for the vertical flow
$(\varphi^{{v}}_t)_{t\in\R}$.
%\begin{itemize}
%%\item[(1)] $[\rho]\in E_\omega^-(M,\R)$;
%\item[(2)]
%\end{itemize}
\end{theorem}
%\begin{theorem}\label{cohthm}
The  main technical tools to prove Theorem \ref{cohthm} are
essentially present in the literature\footnote{Theorem
\ref{cohthm} could be deduced from the recent work of Forni in
\cite{For-sob}, in which  much deeper and more technical results
on the cohomological equation are proved.
%This is shown in the Appendix \ref{currents:sec}, where we recall the language of currents and invariant distributions to then explain how to deduce Theorem \ref{cohthm} from the results in \cite{For-coheq}.
The crucial point in the proof of Theorem \ref{cohthm} is the
control on deviations of ergodic averages from the stable space,
which first appears in the work by Zorich \cite{Zo:how} in the
special case in which $\mu$ is the canonical Masur-Veech measure
on a stratum.
%\begin{comment} in which it is shown that  for  $\mu$-almost every
%$\omega\in\mathcal{M}(M)$ and any cohomology class  $\rho\in E_\omega^-(M,\R)$, if $\widetilde{\gamma_t} \in H_1(M, \R)$ denotes the homology class obtained by closing up a trajectory of lenght $t$ for the vertical flow  on $(M, \omega)$ by an arc of bounded lenght, then $\widetilde{\gamma_t} \wedge \rho$ stays at bounded distance from $ E_\omega^-(M,\R)$  for any $t\geq0$ (see Theorem 2 and Conditional Theorem 3 in \cite{Zo:how}, where we remark that Conjecture 1 and 2 are now proved by Avila and Viana \cite{Av-Vi}).
%\end{comment}
Very recently, an adaptation of the proof of Zorich's deviation
result for any $SL(2, \R)$-invariant measure has appeared in the
preprint \cite{DHL}.
%From the upper bounds on deviations, the
%first part of the Theorem can be deduced by Gottshalk-Hedlund
%methods, while the second part exploits the lower bounds.
}.
For completeness, in the Appendix \ref{coh:sec} we  include a
self-contained proof of Theorem \ref{cohthm}. In the same Appendix
we also prove the following Lemma, which is used in the proof of
Theorem \ref{cohthm} and that will also be used in the proof of
non-regularity in Section \ref{nonregularity:sec}.
\begin{lemma}\label{comparisontimes:lemma}
Let $\mu$ be any $SL(2,\R)$-invariant probability measure on
$\mathcal{M}^{(1)}(M)$ ergodic for the Teichm\"uller flow.
Then for $\mu$-almost every $\omega\in\mathcal{M}^{(1)}(M)$, there
exists a sequence of times $(t_k)_{k \in \N}$ with
$t_k\to+\infty$, $m \in \N $, a constant $c>1$  and a sequence
$\{\gamma_1^{(k)}, \dots, \gamma_{m}^{(k)}\}_{k\in \N} $ of
elements of $ H_{1}(M, \mathbb{Z})$ such that, for any  $\rho \in
H^1(M, \R)$ one has
\begin{equation}\label{base0}
\frac{1}{c} \vert \vert \rho  \vert \vert_{G_{t_k} \omega} \leq
\max_{1\leq j\leq m} \left|\int_{\gamma^{(k)}_j } \rho \right|
\leq c \vert \vert \rho \vert \vert_{G_{t_k} \omega} .
\end{equation}
\end{lemma}

\section{Veech surfaces and square-tiled surfaces}\label{Veech:sec}
The {\em affine group } $\Aff(M,\omega)$ of $(M,\omega)$ is the
group of orientation preserving homeomorphisms of $M$ and
preserving $\Sigma$ which are given by affine maps in regular
adopted coordinates. The set of differentials of these maps is
denoted by $SL(M,\omega)$ and it is a subgroup of $SL(2,\R)$.
 A translation surface $(M,\omega)$ is called a {\em lattice surface} (or a {\em Veech surface})
if $SL(M,\omega)\subset SL(2,\R)$ is a lattice.

If $(M,\omega_0)$ is a lattice surface, the $SL(2,\R)$-orbit of
$(M,\omega_0)$  in
$\mathcal{M}^{(1)}(M)$, which will be denoted by $\mathscr{L}_{\omega_0}$,
is closed and can be identified with the
homogeneous space $SL(2,\R)/SL(M,\omega_0)$. The identification is
given by the map
$\Phi: SL(2,\R) \to \mathscr{L}_{\omega_0}\subset\mathcal{M}^{(1)}(M)$ that
sends $g \in SL(2,\R)$ to $g\cdot \omega_0\in \mathscr{L}_{\omega_0}$,
whose kernel is exactly the Veech group $SL(M,\omega_0)$.
Thus $\Phi$ can be treated a map from $SL(2,\R)/SL(M,\omega_0)$ to
$\mathscr{L}_{\omega_0}$. Therefore, $\mathscr{L}_{\omega_0}$
carry a canonical $SL(2, \R)$-invariant measure  $\mu_0$, which is
the image of the Haar measure on $SL(2,R)/SL(M,\omega_0)$ by the map
$\Phi: SL(2,\R)/SL(M,\omega_0) \to \mathcal{M}^{(1)}(M)$.
%Remark that since the $SL(2,\R)$-orbit of
%$(M,\omega_0)$ in $\mathcal{M}^{(1)}(M)$ is closed,  $\mu_0$ is
%finite and ergodic.
We will refer to $\mu_0$ as the \emph{canonical measure} on $\mathscr{L}_{\omega_0}$.
Since the
homogeneous space $SL(2,\R)/SL(M,\omega_0)$ is the unit tangent
bundle of a surface of constant negative curvature, the
(Teichm{\"u}ller) geodesic flow on $SL(2,\R)/SL(M,\omega_0)$ is
ergodic. Thus, $\mu_0$ is ergodic.
%the  $SL(2,\R)$-orbit of $(M,\omega_0)$  can be identified with $SL(2,\R)/SL(M,\omega)$, it  is a classical fact that $\mu_0$ is ergodic.

% Moreover, the Teichmueller geodesic flow restricted to the $SL(2,\R)$-orbit of $(M,\omega)$ coincide with the geodesic flow on the unit tangent bundle of the  hyperbolic surface $SO(2)\backslash SL(2,\R)/SL(M,\omega)$.

All {\em square-tiled} translation surfaces are examples of
lattice surfaces.  If $(M,\omega_0)$ is square-tiled, the Veech
group $SL(M,\omega_0)$ is indeed a finite index subgroup of
$SL(2,\Z)$.
\begin{comment}
If  $k$ is the index, $k$ is also the cardinality of the $SL(2,
\Z))$-orbit of $(M,\omega_0)$ in  $SL(2,\R)/SL(M,\omega_0)$ and
$SL(2,\R)/SL(M,\omega)$ is a finite cover of $SL(2,\R)/SL(2, \R) $
of degree $k$.
\end{comment}
%\subsubsection*{A splitting of the homology for square tiled surfaces}
Let $(M,\omega_0)$ be square-tiled and let $p:M\to\R^2/\Z^2$ be a
ramified cover unramified outside $0\in\R^2/\Z^2$ such that
$\omega_0=p^*(dz)$.
%Let us describe a splitting  of the homology of any square tiled surface  $(M,\omega)$.
% is {\em square-tiled} if there exists a ramified cover $p:M\to\R^2/\Z^2$ unramified outside $0\in\R^2/\Z^2$ such that $\omega=p^*(dz)$.
Set $\Sigma'=p^{-1}(\{0\})$.
%Denote by $Sq(M,\omega_0)$ the set of squares tiled $(M,\omega)$.
For $i$-th square of $(M,\omega_0)$, let $\sigma_i,\zeta_i\in
H_1(M,\Sigma',\Z)$ be the relative homology class of the path in
the $i$-th square from the bottom left corner to the bottom right
corner and to the upper left corner, respectively. Let
$\sigma=\sum\sigma_i\in H_1(M,\Z)$ and $\zeta=\sum\zeta_i\in
H_1(M,\Z)$.

\begin{proposition}[see \cite{Ma-Yo}]\label{H0Veech}
The space $H_1^{(0)}(M_\omega,\R)$ is the kernel of the
homomorphism $p_*:H_1(M,\R)\to H_1(\R^2/\Z^2,\R)$. Moreover,
$H_1^{st}(M_\omega,\R)=\R\sigma\oplus\R\zeta$.
\end{proposition}
\begin{remark}\label{basisexists}

Let $H_1^{(0)}(M,\Q)$ stand for the kernel of $p_*:H_1(M,\Q)\to
H_1(\R^2/\Z^2,\Q)$ and let
$H_1^{st}(M,\Q):=\Q\sigma\oplus\Q\zeta$. In view of
Proposition~\ref{H0Veech},
\[H_1(M,\Q)=H^{(0)}_1(M,\Q)\oplus H_1^{st}(M,\Q)\]
is an orthogonal decomposition. Since $H^{(0)}_1(M,\Q)$ is
invariant under the action on mapping-class group on
$\mathscr{L}_{\omega_0}= SL(2,R)\cdot\omega_0\subset\mathcal{Q}^{(1)}(M)$,  this yields the following
orthogonal invariant splitting, which
is constant on $\mathscr{L}_{\omega_0}$:
\[\{ H_1(M_\omega,\Q)=H^{(0)}_1(M,\Q)\oplus
H_1^{st}(M,\Q)\text{, }\omega\in \mathscr{L}_{\omega_0}\}.\]
%under the $SL(2,\R)$-action on $\mathscr{L}_{\omega_0}$ and is
%preserved by the action on homology induced by the
%$SL(2,\R)$-action restricted to $\mathscr{L}_{\omega_0}$.
\end{remark}

Note that the for every $\gamma\in H_1(M,\R)$   the { holonomy}
$\hol(\gamma)=\int_\gamma\omega$ satisfies
\[\hol(\gamma)=\int_\gamma p^*dz=\int_{p_*\gamma}dz.\]
Since $\Re dz$ and $\Im dz$ generate $H^1(\R^2/\Z^2,\R)$,
$\hol(\gamma)=0$ implies $p_*\gamma=0$.  Thus $\ker\hol\subset
H_1^{(0)}(M,\R)$. Moreover, since both spaces have codimension
two, the previous inclusion is an equality:
\begin{equation}\label{kerhol}
\ker(\hol) =  H_1^{(0)}(M,\R).
\end{equation}

\section{Non-ergodicity}\label{nonergodicity:sec}
In this section we state and prove our main criterion for non-ergodicity.
%\begin{definition}
%If $\mathcal{L}\subset \mathcal{M}^{(1)}(M)$ is an
%$SL(2,\R)$-invariant subset (\emph{locus}) of
%$\mathcal{M}^{(1)}(M)$, let us say that a splitting $H_1 (M,
%\mathbb{Q}) = H \oplus H^\perp$ is \emph{$SL(2,\R)$-invariant over
%$\mathscr{L}$} if for any $\omega \in \mathscr{L}$ and any $A \in
%SL(2, \R)$, the action of homology $A^*: H_1(M, \mathbb{Q}) \to
%H_1(M, \mathbb{Q})$ induced by the  $SL(2, \R)$ action of $A$ on
%$(M, \omega)$ preserves the splitting, that is $A^* \gamma \in H$
%for any $\gamma \in H$ and $A^* \gamma \in H^\perp$ for any
%$\gamma \in H^\perp$.
%\end{definition}
\begin{theorem}\label{non-ergodicitycriterion}
Let $\mu$ be an $SL(2,\R)$-invariant  probability measure on
$\mathcal{M}^1(M)$ ergodic for the Teichm\"uller flow. Let
$\mathscr{L}\subset \mathcal{M}^{(1)}(M)$ stand for the support of
$\mu$. Assume that
\[\{ H_1 (M_\omega,\mathbb{Q}) = K_1 \oplus K_1^\perp, \; \omega\in\mathscr{L}\}\]
is an invariant orthogonal splitting which is constant on
$\mathscr{L}$.  Let $\mathcal{K}_1 = \bigcup_{\omega \in
\mathscr{L}} \{\omega\} \times K_1$ denote the corresponding
invariant subbundle. Suppose that $\dim_{\Q}\mathcal{K}_1=2$ and
the Lyapunov exponents of the Kontsevich-Zorich cocycle on
$\R\otimes_{\Q}\mathcal{K}_1$ are non-zero.

%there exists a subspace $K_1 \subset H_1(M, \mathbb{Z})$ of
%dimension $2$ and a corresponding orthogonal splitting o$H_1 (M,
%\mathbb{Q}) = K_1 \oplus K_1^\perp$ ver $\mathbb{Q}$  (orthogonal
%with respect to the symplectic intersection form)
%%\begin{equation*}
%%H_1 (M, \mathbb{Q}) = H \oplus H^\perp
%%\end{equation*}
%which is $SL(2,\R)$-invariant over $\mathscr{L}$ and such t hat
%the Lyapunov exponents $\lambda^\mu$ corresponding to $K_1$ are
%non-zero.

Then, for $\mu$ almost every $\omega \in \mathscr{L}$, for  any
$\mathbb{Z}$-cover $(\widetilde{M}_\gamma,
\widetilde{\omega}_\gamma)$ of $(M, \omega)$ given by a homology
class $\gamma \in K_1\cap H_1(M,\Z)$, the vertical flow
$(\widetilde{\varphi}^v_t)_{t\in \R}$ on  $(\widetilde{M}_\gamma,
\widetilde{\omega}_\gamma)$ is not ergodic.
\end{theorem}

As a Corollary of the previous theorem, in this section we also
prove the following.

\begin{corollary}\label{theointro}
Let $(M,\omega_0)$ be a square-tiled compact translation surface
of genus $2$ and let $\mu_0$  be the canonical measure  on the
$SL(2,\R)$-orbit of $(M,\omega_0)$ (see \S\ref{Veech:sec}).
 For $\mu_0$-almost every  $(M,\omega)$ the vertical flow of each
recurrent $\Z$-cover $(\widetilde{M},\widetilde{\omega})$  is
not ergodic.
\end{corollary}

Before giving the proof of Theorem  \ref{non-ergodicitycriterion}
and Corollary \ref{theointro}, we state and prove an auxiliary
Lemma. Let $\mathscr{L}, \mu$ and $ K_1, K_1^\perp$ be as in the
assumptions of Theorem~\ref{non-ergodicitycriterion}. Remark that
since $\{ H_1 (M_\omega, \mathbb{R}) = (\R\otimes_{\Q}K_1) \oplus
(\R\otimes_{\Q}K_1^\perp), \omega \in \mathscr{L}\}$ is an
orthogonal splitting, by Poincar{\'e} duality, we also have a dual
constant orthogonal invariant splitting
\begin{equation}\label{dualsplitting}
\{ H^1 (M_\omega, \mathbb{R}) = {K}^1 \oplus K^1_\perp, \ \omega
\in \mathscr{L}\},  \ {K}^1:=\mathcal{P} (\R\otimes_{\Q}K_1), \
{K}^1_\perp: =\mathcal{P}(\R\otimes_{\Q}K_1^\perp).
\end{equation}
\begin{lemma}\label{irrationality}
Let $\omega \in \mathscr{L}$ be Oseledets regular for $\mu$ for
which the conclusion of Lemma~\ref{comparisontimes:lemma} holds.
Let $\rho\in K^1 \subset H^1(M, \mathbb{R})$ be such that $\rho\in
E_\omega^-(M,\R)\setminus\{0\}$. For any $\Q$-basis  $\{ \sigma_1,
\sigma_2\} \subset H_1(M, \mathbb{Z})$ of $K_1$ the periods
$\left(\int_{\sigma_1}\rho,\int_{\sigma_2}\rho \right)\in\R^2$ do
not belong to $\R\cdot(\mathbb{Q}\times \mathbb{Q})$.
\end{lemma}

\begin{proof}
First note that $\big(\int_{\sigma_1}\rho,\int_{\sigma_2}\rho
\big)\neq (0,0)$. Indeed, if
$\int_{\sigma_1}\rho=\int_{\sigma_2}\rho=0$ then
$\langle\mathcal{P}\sigma,\rho\rangle=\int_{\sigma}\rho=0$ for
every $\sigma\in\R\otimes_{\Q}K_1$. By the definition of $K^1$, it
follows that the symplectic form is degenerated on $K^1$,  which
is a contradiction.
%Given a $\Q$-basis  $\{ \sigma_1, \sigma_2\} \subset H_1(M,
%\mathbb{Z})$ of $K_1$, let us consider  the map $\Upsilon: K^1 \to
%\R^2$ given by
%\[\Upsilon(\nu):=\Big(\int_{\sigma_1}\nu,\int_{\sigma_2}\nu\Big)\in\R^2,\]
%which establishes an isomorphism between $K_1$ and $\R^2$.

Denote by $p_{K_1}:H_1(M,\Q)\to K_1$ the orthogonal projection.
Since the splitting is over $\mathbb{Q}$, by writing the image by
$p_{K_1}$ of each element of a basis of $H_1(M, \mathbb{Z})$ as a
linear combination over $\mathbb{Q}$ of $\sigma_1, \sigma_2$, one
can show that there exists $q\in\N$ (the least common multiple of
the denominators) such that
\begin{equation}\label{warzut}
p_{K_1}(H_1(M, \mathbb{Z}))\subset (\Z\sigma_1\oplus\Z\sigma_2)/q.
\end{equation}
Suppose that, contrary to the claim in the Lemma,
$\big(\int_{\sigma_1}\rho,\int_{\sigma_2}\rho\big)\in
\R\cdot(\Q\times\Q)$. Then there exists $a\in\R\setminus\{0\}$
such that $\int_{\sigma_1}\rho,\int_{\sigma_2}\rho\in a\Z$. Thus,
since $\rho \in K_1 $, by the definition of $K^1$ and
\eqref{warzut}, for every $\sigma\in H_1(M,\Z)$ we have \be
\label{integralperiods}\int_{\sigma}\rho=\int_{p_{K_1}\sigma}\rho\in\frac{1}{q}\Big(\Z\int_{\sigma_1}\rho+\Z\int_{\sigma_2}\rho\Big)\in\frac{a}{q}
\Z . \ee By Lemma \ref{comparisontimes:lemma} (which we can apply
by assumption), there exists a constant $c>0$, a sequence of times
$(t_k)_{k \in \N}$, $t_k\to+\infty$ and a sequence
$\{\gamma_1^{(k)}, \dots, \gamma_{m}^{(k)}\}_{k\in\N}\subset
H_{1}(M, \mathbb{Z})$,  such that
\begin{equation}\label{wniosnorm}
0<\frac{1}{c} \| \rho \|_{G_{t_k} \omega}
\leq\widehat{\rho}_k:=\max_{1\leq j\leq m}
\Big|\int_{\gamma^{(k)}_j } \rho \Big| \leq c \| \rho \|_{G_{t_k}
\omega} \text{ for any }k\in\N,
\end{equation}
Thus, by (\ref{integralperiods}),  $\widehat{\rho}_k\in
\frac{a}{q}\Z \setminus\{0\}$ for every natural $k$. On the other
hand, since $\rho\in E^-_\omega(M,\R)$, $\| \rho \|_{G_{t_k}
\omega}\to 0$ as $k\to\infty$. In view of \eqref{wniosnorm},
$\widehat{\rho}_k\to 0$ as $k\to\infty$, which gives a
contradiction.
\end{proof}

\begin{proofof}{Theorem}{non-ergodicitycriterion}
Let  $\mathscr{L}'$ be the set  of Oseledets regular $\omega \in
\mathscr{L}$ for which the conclusion of Theorem~\ref{cohthm} and
Lemma~\ref{comparisontimes:lemma} hold and, in addition, for which
the vertical and the horizontal flows on $(M,\omega)$ are ergodic.
In view of Theorem \ref{cohthm}, Lemma~\ref{comparisontimes:lemma}
and \cite{Ma:IET},  $ \mu(\mathscr{L}')=1$. For any $\omega \in
\mathscr{L}'$ let us consider a $\Z$-cover
$(\widetilde{M},\widetilde{\omega})$ of $(M,\omega)$ associated to
a non-trivial homology class $\gamma \in H_1(M,\Z)\cap K_1$.

%If the $\Z$-cover $(\widetilde{M},\widetilde{\omega})$ is not
%connected then every its directional flows is not ergodic. So we
%will assume that $(\widetilde{M},\widetilde{\omega})$ is
%connected.

Consider the invariant orthogonal splitting of cohomology in
(\ref{dualsplitting}). By assumption, the Lyapunov exponents of the
reduced  Kontsevich-Zorich cocycle
$(G^{KZ,\mathcal{K}^1}_t)_{t\in\R}$ are non-zero. Since the
cocycle $(G^{KZ,\mathcal{K}^1}_t)_{t\in\R}$ preserves the
symplectic structure on $\mathcal{K}^1$ given by the intersection
form, it follows that the exponents of the subbundle $\mathcal{K}^1$ are
one positive and one negative. Thus, the stable space
$E_\omega^-(M,\R)$ intersects $K^1$ exactly in a one dimensional
space. Let $\rho\in\Omega^1(M)$ be a smooth closed form such that
$[\rho]\in E_\omega^-(M,\R)\cap K^1$.

Let $\{\sigma_1,\sigma_2\}\subset H_1(M, \Z)$ be a $\Q$-basis of
$K_1$ such that $\sigma_1=\gamma$ and choose any $\Q$-basis
$\{\sigma_3,\ldots,\sigma_{2g}\}\subset H_1(M, \Z)$ of
$K_1^\perp$. Clearly, $\{\sigma_1,\ldots,\sigma_{2g}\}\subset
H_1(M, \Z)$ is a $\Q$-basis of $H_1(M, \Q)$.
%Then there exists a natural number $q$ such that $H_1(M,
%\Z)\subset(\Z\sigma_1\oplus\ldots\oplus\Z\sigma_{2g})/q$.
%Let us complete $\gamma \in H_1(M,\Z)\cap K_1$ to a base $\{
%\sigma_1:=\gamma, \sigma_2\}\subset H_1(M, \Z)$ of $K_1$ over
%$\mathbb{Q}$.
By Lemma \ref{irrationality}, the periods
$\Upsilon([\rho])=(\int_{\sigma_1}\rho, \int_{\sigma_2}\rho)$ do
not belong to $\R\cdot(\mathbb{Q } \times \mathbb{Q})$. Therefore,
$\int_{\sigma_1}\rho\neq 0\neq \int_{\sigma_2}\rho$ and
$\int_{\sigma_1}\rho/\int_{\sigma_2}\rho\in\R\setminus\Q$. Thus,
since $\langle \gamma , \sigma_2\rangle  = \langle \sigma_1,
\sigma_2\rangle  \in \mathbb{Z}\setminus\{0\}$, up to multiplying
$\rho$ by a non-zero real constant (more precisely, by $\langle
\gamma, \sigma_2\rangle / \int_{\sigma_2}\rho$), we can assume
that
\begin{equation}\label{periods0}
\Upsilon([\rho])=(\alpha, \langle\gamma, \sigma_2\rangle), \qquad
\text{where} \ \alpha\in\R\setminus\Q.
\end{equation}
Choose a transverse horizontal interval $I\subset M$ and let
$T:I\to I$  be the IET obtained as  Poincar\'e return map and let
$I_j$, $j\in\mathcal{A}=\{1,\ldots,m\}$, be the exchanged
subintervals. Then the homology classes $\gamma_j$,
$j\in\mathcal{A}$ generates $H_1(M,\Z)$ (as in
\S\ref{reduction:sec}, $\gamma_j=[v_x]$ where $v_x$ is obtained by
closing up the first return trajectory of the vertical flow
$(\varphi^v_t)_{t\in\R}$ of any $x \in I_j$ by a horizontal
interval). Since the vertical flow $(\varphi^v_t)_{t\in\R}$ on
$(M,\omega)$ is ergodic, $T$ is ergodic as well. By
Lemma~\ref{lem_flow_auto}, the vertical flow
$(\widetilde{\varphi}^{{v}}_t)$ on
$(\widetilde{M},\widetilde{\omega})$ is isomorphic to a special
flow built over the skew product $T_\psi:I\times\Z\to I\times\Z$,
where $\psi=\psi_\gamma$ is given by
\begin{equation}\label{psi}
\psi_\gamma(x)=\langle \gamma,\gamma_j\rangle  \quad\text{ if
}\quad x\in I_j.
\end{equation}
Let us consider the smooth bounded function
$f:M\setminus\Sigma\to\R$, $f=i_{X_{v}}\rho$ and let
$\psi_\rho:I\to\R$ be the corresponding cocycle for $T$ defined by
$\psi_\rho(x)=\int_0^{\tau(x)}f(\varphi^v_sx)\,ds$. By Theorem
\ref{cohthm}, since $[\rho]\in E_\omega^-(M,\R)$, the cocycle
$F^{v}_f$  for the vertical flow $(\varphi^{{v}}_t)_{t\in\R}$ is a
coboundary and thus, equivalently, by Lemma
\ref{equivalentcoboundaries}, the cocycle $\psi_\rho$ is a
coboundary for $T$ as well. Let $\gamma':=\mathcal{P}^{-1} [\rho]
\in \R\otimes_{\Q}K_1$ be the Poincar{\'e} dual of $[\rho]\in
K^1$. In view of Remark~\ref{remark:cocycle-hom-cohom}, the
cocycle $\psi_{\gamma'}:I\to\R$ given by
%there exists a smooth
%function $g:I\to\R$ such that $\psi_{\gamma'}=\psi_\rho+g-g\circ T$ with
\begin{equation}\label{psi0}
\psi_{\gamma'}(x)=\langle \gamma', \gamma_j, \rangle  \quad\text{
whenever }\quad x\in I_j
\end{equation}
is cohomologous to $-\psi_\rho$ and thus it is also a coboundary.

Clearly $\psi:I\to\Z$ can be considered as cocycle taking values
in $\R$ for the automorphism $T$. Then the group of essential
values $E_{\R}(\psi)=E_{\Z}(\psi)$ of this cocycle is a subgroup
of $\Z$. Let us consider the cocycle $\phi:I\to\R$ given by
$\phi:=\psi+\psi_{\gamma'}$.
%\[\phi:=\psi-\psi_{\gamma'}=\psi-\psi_\rho-g+g\circ T.\]
In view of (\ref{psi}) and (\ref{psi0}),
\begin{equation}\label{jcomp}
\phi(x)=\langle \gamma,\gamma_j\rangle +{\langle \gamma'
,\gamma_j\rangle } = \langle  \gamma +\gamma',\gamma_j \rangle
\quad\text{ if }\quad x\in I_j.
\end{equation}
Since $\sigma_1,\ldots,\sigma_{2g}\in H_1(M,\Z)$ form a basis of
$H_1(M,\Q)$, there exists a natural number $M$ and an $m\times
2g$-matrix $A=[a_{ji}]$ (of rank $2g$) with integer entries  such
that
\[\gamma_j=\frac{1}{M}\sum_{i=1}^{2g}a_{ji}\ \sigma_i\text{ for }\quad j=1,\ldots,m.\]
Since   $\gamma + \gamma' \in \R\otimes_\Q K_1$ and
$\sigma_3,\ldots,\sigma_{2g}\in K_1^{\perp}$ are symplectic
orthogonal to the subspace $\R\otimes_\Q K_1$, we have
\[\langle  \gamma+\gamma', \gamma_j\rangle =\frac{1}{M}\left(a_{j1}
\langle  \gamma + \gamma', \sigma_1\rangle +a_{j2}\langle \gamma +
\gamma',\sigma_2\rangle \right) \quad\text{ for }\quad j=1,\ldots,m.\]
%Let us write $A_ij = a_{ij}/M$  where $M$ is the common
%denominator and $a_{ij} \in \Z$ for all $1\leq i, j\leq m$.
%Hence
%\begin{equation}\label{zestgam}
%(\gamma\wedge \gamma_1,\ldots,\gamma\wedge
%\gamma_m)=\frac{1}{M}((\gamma\wedge\sigma_1)\cdot\overline{a}_1+(\gamma\wedge\sigma_2)\cdot\overline{a}_2),
%\end{equation}
%where $\overline{a}_1, \overline{a}_2\in\Z^m$ are non-zero.
Thus, from (\ref{jcomp}),
it follows that, for any $1\leq j \leq m$ and any  $x \in I_j$, we have
\begin{equation*}
\phi(x) =\langle \gamma + \gamma',\gamma_j\rangle  =
\frac{1}{M}\left({a}_{j1} \langle \gamma+\gamma',\sigma_1\rangle
+{a}_{j2} \langle \gamma+\gamma',\sigma_2\rangle \right) .
\end{equation*}
Since $[\rho]=\mathcal{P}\gamma'$,  in view of
\eqref{prop_duality} and \eqref{periods0},
\[\left( \langle \gamma',\sigma_1\rangle ,\langle \gamma',\sigma_2\rangle \right)=
-\left( \int_{\sigma_1}\rho ,\int_{\sigma_2}\rho \right)=
-\Upsilon([\rho])= -\left(\alpha ,\langle  \gamma, \sigma_2\rangle
\right)
\]
with  $\alpha \in\R\setminus \Q$.  Hence, since $\sigma_1=\gamma$,
$\langle \gamma+\gamma',\sigma_1\rangle=-\alpha$ and $\langle
\gamma+\gamma',\sigma_2\rangle=0$. Therefore,
\begin{equation*}
\phi(x)  = -\frac{ {a}_{j1}}{M}\alpha\quad\text{ if }\quad x\in
I_j.
\end{equation*}
%\begin{align*}
%(&\gamma\wedge\gamma_1-{(\sigma_0\wedge\gamma_1)}\cdot{(\gamma\wedge\sigma_1)},\ldots,\gamma\wedge\gamma_m-
%{(\sigma_0\wedge\gamma_m)}\cdot{(\gamma\wedge\sigma_1)})\\
%&=
%((\gamma-(\gamma\wedge\sigma_1)\cdot\sigma_0))\wedge\gamma_1,\ldots,
%(\gamma-(\gamma\wedge\sigma_1)\cdot\sigma_0))\wedge\gamma_m)\\
%&=\frac{1}{M}(((\gamma-(\gamma\wedge\sigma_1)\cdot\sigma_0)\wedge\sigma_1)\cdot\overline{a}_1+
%((\gamma-(\gamma\wedge\sigma_1)\cdot\sigma_0)\wedge\sigma_2)\cdot\overline{a}_2)\\
%&=\frac{1}{M}((\gamma\wedge\sigma_1-(\gamma\wedge\sigma_1)\cdot(\sigma_0\wedge\sigma_1))\cdot\overline{a}_1+
%(\gamma\wedge\sigma_2-(\gamma\wedge\sigma_1)\cdot(\sigma_0\wedge\sigma_2))\cdot\overline{a}_2)\\
%&=\frac{1}{M}((\gamma\wedge\sigma_2)-(\gamma\wedge\sigma_1)\cdot
%a)\cdot\overline{a}_2.
%\end{align*}
Thus, since  $\phi(x)$ is an  integer multiple of
$\overline{\alpha} := \alpha /M \notin \Q$ for any $x \in I$,
%where the integer multiple is given by  $a_{j1}$ if $x \in I_j$,
the cocycle $\phi:I\to\R$ takes values in $\overline{\alpha} \Z$,
hence $E_{\R}(\phi)\subset \overline{\alpha}  \Z$ (see Proposition \ref{basicessentialvalues}). Since
$\psi$ is cohomologous to $\phi$, it follows from Proposition \ref{basicessentialvalues} that
$E_{\R}(\psi)=E_{\R}(\phi)\subset \overline{\alpha}  \Z$. As
$E_{\R}(\psi)=E_{\Z}(\psi)\subset\Z$ and $\overline{\alpha} \Z\cap \Z = \{ 0\}$, we get
$E_{\Z}(\psi)=E_{\R}(\psi)=\{0\}$. By
Proposition~\ref{proposition:schmidt}, $T_\psi:I\times\Z\to
I\times\Z$ is not ergodic. In view of Remark~\ref{redtoskew}, it
follows that the vertical flow
$(\widetilde{\varphi}^{v}_t)_{t\in\R}$ is not ergodic.
\end{proofof}

\begin{proofof}{Corollary}{theointro}
Let $(M,\omega_0)$ be a square-tiled translation  surface of genus
two. Let $\mu_0$ be the canonical probability measure on
$\mathscr{L}_{\omega_0}$ (the $SL(2,\R)$-orbit of $(M,\omega_0)$
in $\mathcal{M}^1(M)$), which  is ergodic (see \S\ref{Veech:sec})
and KZ hyperbolic by Theorem \ref{Bain} (since $M$ has genus two).
Let $K_1= H_1^{(0)}(M,\Q)$ and $K_1^\perp= H_1^{st}(M,\Q)$ (see \S
\ref{Veech:sec}). The subspace $K_1$ is the kernel of the
homomorphism $p_*:H_1(M,\Q)\to H_1(\R^2/\Z^2,\Q)$ and, in view of
(\ref{kerhol}), it is the kernel of $\hol:H_1(M,\Q)\to \C$. Remark
that since $M$ has genus two, $\dim_\Q H_1(M, \R)=4$ and $\dim_\Q
K_1=2$.

By Remark  \ref{basisexists},  $\{H_1(M_\omega,\Q)=K_1 \oplus
K_1^{\perp})$, $\omega\in\mathscr{L}_{\omega_0}\}$ is an invariant
splitting which is constant over $\mathscr{L}_{\omega_0}$. Hence,
we can apply Theorem \ref{non-ergodicitycriterion}. The conclusion
follows by remarking that, in view of (\ref{kerhol}), the
recurrent $\mathbb{Z}$-covers are exactly the $\mathbb{Z}$-covers
$(\widetilde{M}_\gamma, \widetilde{\omega}_\gamma) $ given by
$\gamma \in H_1^{(0)}(M,\Q)\cap H_1(M,\Z)=K_1\cap H_1(M,\Z)$.
\end{proofof}

\section{Non-regularity}\label{nonregularity:sec}
In this section, we
prove the following Theorem:
\begin{theorem}\label{inv_fin_meas}
Let $\mu$ be any $SL(2,\R)$-invariant, KZ-hyperbolic probability ergodic
measure on  $\mathcal{M}^{(1)}(M)$. For $\mu$-almost every
$(M,\omega)$ the vertical flow of each $\Z$-cover
$(\widetilde{M}_\gamma,\widetilde{\omega}_{\gamma})$ given by a
non-zero $\gamma \in H_1(M , \Z)$ has no invariant subset of
positive finite measure.
\end{theorem}
Theorem \ref{inv_fin_meas} is derived from Theorem \ref{cohthm}
and Lemma \ref{notunstable} stated below,  via the representation
of directional flows on  $\Z$-cover  as special flows over
skew-products.

\begin{lemma}\label{notunstable}
Let $\mu$ be an $SL(2,\R)$-invariant probability measure on
$\mathcal{M}^{(1)}(M)$ ergodic for the Teichm\"uller flow. For
each non-zero $\gamma \in H_1(M , \Z)$ and for $\mu$-a.e $\omega
\in \mathcal{M}^{(1)}(M)$ the Poincar{\'e} dual class
$\mathcal{P}\gamma$ does not  belong to the stable space $
E_\omega^-(M,\R)$.
\end{lemma}
\begin{proof}
Consider any Oseledets regular $\omega\in\mathcal{M}^{(1)}(M)$ in
the set of $\mu$ full
measure given by Lemma \ref{comparisontimes:lemma} and let
$(t_k)_{k \in \N}$,  $\{\gamma_1^{(k)}, \dots, \gamma_{m}^{(k)}\}
\subset H_{1}(M, \mathbb{Z})$ and $c>0$ be given by Lemma
\ref{comparisontimes:lemma}. Then, by Poincar{\'e} duality, Lemma
\ref{comparisontimes:lemma} applied to $\mathcal{P}\gamma\neq 0$ gives
that
\begin{equation}\label{periods}
0<\widehat{\gamma}_k:
=\max_{1\leq j\leq m} \big| \langle{\gamma^{(k)}_j } , \gamma\rangle\big| =
\max_{1\leq j\leq m} \Big| \int_{\gamma_j ^{(k)}}  \mathcal{P}\gamma\Big|
\leq c \| \mathcal{P}\gamma \|_{G_{t_k}\omega}
\end{equation}
for every $k\in\N$. Therefore, $\widehat{\gamma}_k$ is a natural number
for any $k\in \mathbb{N}$. If  $\mathcal{P}\gamma \in
E_\omega^-(M,\R)$, by definition of the stable space (see
(\ref{stabledef})), the RHS of (\ref{periods}) tends to zero as $k\to\infty$,
hence $\widehat{\gamma}_k\to 0$ as $k\to\infty$,
which gives a contradiction. We conclude that
$\mathcal{P}\gamma$ does not belong to  $E_\omega^-(M,\R)$.
\end{proof}

\begin{proofof}{Theorem }{inv_fin_meas}
Let $\mu \in \mathcal{M}^{(1)}$ belong to the set of  full $\mu$
measure  given by Lemma \ref{notunstable} and let
$(\widetilde{M},\widetilde{\omega})=(\widetilde{M}_\gamma,\widetilde{\omega}_\gamma)$
for some non-zero $\gamma\in H_1(M,\Z)$. By
Lemma~\ref{lem_flow_auto}, the vertical flow
$(\widetilde{\varphi}^{{v}}_t)_{t\in\R}$ has a representation as a
special flow build over the skew product $T_{\psi}:I\times\Z\to
I\times\Z$, where $\psi(x)=\langle \gamma,\gamma_\alpha\rangle$ if
$x\in I_\alpha$, $\alpha\in\mathcal{A}$ and under a roof function
which takes finitely many positive values. Thus, the flow
$(\widetilde{\varphi}^{v}_t)_{t\in\R}$ has invariant subsets of
finite positive measure if and only if the skew product $T_{\psi}$
has. In view of Proposition~\ref{proposition:cocycle}, this
happens if and only if  the cocycle $\psi:I\to\Z$ for the IET $T$
is a coboundary. Thus, it is enough to show that  $\psi:I\to\Z$ is
not a coboundary.

Suppose that, contrary to our claim, $\psi:I\to\Z$ is
 a coboundary. Choose a smooth closed form $\rho\in\Omega^1(M)$
such that $[\rho]=\mathcal{P}\gamma$. Let us consider the cocycle
$F^v_{i_{X_v}\rho}$ for the flow $(\varphi_t^v)_{t\in\R}$ and the
corresponding cocycle $\psi_\rho:I\to\R$ for $T$ (see the
definition in Remark~\ref{remark:cocycle-hom-cohom}). By
Remark~\ref{remark:cocycle-hom-cohom}, the cocycle $\psi_\rho$ is
cohomologous to the cocycle $-\psi$, so also $\psi_\rho$ is a
coboundary. In view of Lemma \ref{equivalentcoboundaries}, it
follows that also  $F^v_{i_{X_v}\rho}$ is a coboundary. Since
$\mu$ is KZ-hyperbolic, by the second part of
Theorem~\ref{cohthm}, $\mathcal{P}\gamma=[\rho]\in
E_\omega^-(M,\R)$. On the other hand, by Lemma~\ref{notunstable},
$\mathcal{P}\gamma\notin  E_\omega^-(M,\R)$, which is a
contradiction.
\end{proofof}

\begin{corollary}\label{cor_UMEC}
Let $\mu$ be any $SL(2,\R)$-invariant, ergodic, KZ-hyperbolic
finite measure on  $\mathcal{M}^{(1)}(M)$ and let $H_1
(M,\mathbb{Q}) = K_1 \oplus K_1^\perp$ be a decompositions
satisfying the assumption of
Theorem~\ref{non-ergodicitycriterion}. Then for $\mu$-almost every
$(M,\omega)$ and every non-zero $\gamma \in K_1\cap H_1(M , \Z)$
the vertical flow of the $\Z$-cover
$(\widetilde{M}_\gamma,\widetilde{\omega}_{\gamma})$ is not
ergodic and it has uncountably many ergodic components and it has
no invariant subset of positive finite measure.
\end{corollary}

\begin{proof}
The absence of invariant subsets of positive finite measure follow
directly from Theorem~\ref{inv_fin_meas}. By the proof of
Theorems~\ref{non-ergodicitycriterion}~and~\ref{inv_fin_meas}, for
$\mu$-almost every $\omega\in\mathcal{M}^{(1)}(M)$ and every
non-zero $\gamma \in K_1\cap H_1(M , \Z)$ the vertical flow on
$(\widetilde{M}_\gamma,\widetilde{\omega}_{\gamma})$ has a special
representation over a skew product $T_\psi:I\times\Z\to I\times\Z$
such that $E_\Z(\psi)=\{0\}$ and $\psi$ is not a coboundary. In
view of Proposition~\ref{proposition:schmidt},
$\overline{E}_\Z(\psi)=\{0,\infty\}$, so the cocycle $\psi$ is
non-regular. By Corollary~\ref{corol_nonregul}, the skew product
and hence (by the reduction in \S\ref{reduction:sec}) also the
vertical flow on
$(\widetilde{M}_\gamma,\widetilde{\omega}_{\gamma})$ have
uncountably many ergodic components.
\end{proof}

\section{Final arguments}\label{Fubini:sec}
In this section we conclude the proofs of the main results stated
in the Introduction, that is Theorem \ref{stripbilliard} (see
\S\ref{strip}), Theoreom \ref{Ehrenfestthm} and Corollary
\ref{Ehrenfestcor} (see \S\ref{Ehrenfest:finalsec}) and Theorem
\ref{maintheorem} and Corollary \ref{non-divergence} (see
\S\ref{applications:sec}).  The arguments are essentially based on
a Fubini-type arguments. In \S\ref{applications:sec} we first
present a simple Fubini argument which holds in the case of
lattice surfaces (Proposition \ref{Fubinilattice}) and can be used
to prove Theorem \ref{maintheorem} and parts $(1)$ of Theorem
\ref{stripbilliard} and $(1), (2)$ of Theoreom \ref{Ehrenfestthm}
. The other parts of Theorem \ref{stripbilliard} and
\ref{non-divergence} require a different type of Fubini argument,
presented in \S\ref{strip} and \S\ref{Ehrenfest:finalsec}
respectively.

\subsection{A Fubini argument for lattice surfaces.}\label{applications:sec}
In this section we prove the following Proposition and then  use
it to prove Theorem \ref{maintheorem} and Corollary
\ref{non-divergence}.
\begin{proposition}\label{Fubinilattice}
Let  $(M, \omega_0)$ be a lattice surface and $\mu_0$ be the
canonical measure on its $SL(2, \R)$-orbit $\mathscr{L}_{\omega_0}$. Fix a non-zero
$\gamma\in H_1(M,\Z)$. Assume that for $\mu_0$-almost every
$\omega \in \mathscr{L}_{\omega_0}$ the vertical flow
$(\widetilde{\varphi}_t^{v})_{t\in\R}$ on $(\widetilde{M}_\gamma,
\widetilde{\omega}_\gamma)$ satisfy one (or more) of the following properties:
\begin{itemize}
\item[({P}-1)] is not ergodic;
\item[({P}-2)] has uncountably many ergodic components;
\item[({P}-3)] has no invariant sets
of finite measure.
\end{itemize} Then for almost every $\theta \in S^1$, the directional flow
$(\widetilde{\varphi}_t^{\theta})_{t\in\R}$ on
$(\widetilde{M}_\gamma, \widetilde{(\omega_0)}_\gamma)$ also satisfy the same property \rm{(P-1)}, \rm{(P-2)}, or \rm{(P-3)}. %, $(\mathcal{P}2)$ or $(\mathcal{P}3)$.
%Let $K_1\subset H_1(M, \mathbb{Q})$ be a subspace  invariant under
%the $SL(2, \R)$-action.  Assume that for $\mu_0$-almost every
%$\omega \in \mathcal{M}^{(1)}(M)$, for any  $\mathbb{Z}$-cover
%$(\widetilde{M}_\gamma, \widetilde{\omega}_\gamma)$ given by a
%non-zero $\gamma \in K_1$ the vertical flow
%$(\widetilde{\varphi}_t^{v})_{t\in\R}$ on $(\widetilde{M}_\gamma,
%\widetilde{\omega}_\gamma)$ is not ergodic and has no invariant
%sets of finite measure.
%
%Then for any  $\mathbb{Z}$-cover $(\widetilde{M}_\gamma,
%\widetilde{(\omega_0)}_\gamma)$ of $(M, \omega_0)$ given by a
%non-zero $\gamma \in K_1$, for almost every $\theta \in S^1$, the
%directional flow $(\widetilde{\varphi}_t^{\theta})_{t\in\R}$ on
%$(\widetilde{M}_\gamma, \widetilde{(\omega_0)}_\gamma)$ is not
%ergodic and it has no invariant sets of finite measure.
\end{proposition}

Let us first state to elementary Lemmas useful in the proofs. For
every $g\in SL(2,\R)$ and $\theta\in S^1$ let  us denote by $g
\cdot \theta\in S^1$  the action of $SL(2,\R)$ on $S^1$ determined
by $e^{i g\cdot \theta}=g(e^{i\theta})/|g(e^{i\theta})|$.

\begin{lemma}\label{changeg}
Let $(M,\omega)$ be a translation surface (not necessary compact).
Then for every $g\in SL(2,R)$ and $\theta\in S^1$ there exists
$s>0$ such that the directional flows $(\varphi^{g
\cdot\theta}_{st})_{t\in\R}$ on $(M,g\cdot \omega)$ and
$(\varphi^{\theta}_{t})_{t\in\R}$ on $(M,\omega)$ are
measure-theoretically isomorphic.
\end{lemma}
\begin{proof}
Let $s=s(g,\theta):=|g(e^{i\theta})|$. We claim  that $s X^{g\cdot\theta}_{g\cdot\omega}=X^{\theta}_{\omega}$. Indeed
\[i_{sX^{g\cdot\theta}_{g\cdot\omega}}\omega=sg^{-1}(i_{X^{g\cdot\theta}_{g\cdot\omega}}g\cdot\omega)
=sg^{-1}(e^{i g\cdot\theta})=g^{-1}(|g(e^{i\theta})|e^{i g\cdot\theta})=g^{-1}\circ g(e^{i\theta})=
e^{i\theta}\]
and since $X^{\theta}_{\omega}$ is defined by $i_{X^\theta_\omega}\omega = e^{i\theta}$, this proves the claim. From the claim, we also have
$\varphi^{g\cdot\omega,g\cdot\theta}_{st}=\varphi^{\omega,\theta}_{t}$
for every $t\in\R$. Since moreover, $\nu_{g\cdot\omega}=\nu_{\omega}$, the Lemma follows.
\end{proof}

\begin{lemma}\label{invarianceK1covers}
For every $\gamma\in  H_1(M,\Z)$ and $g\in SL(2,R)$ we have
$(\widetilde{M}_\gamma,\widetilde{g\cdot\omega}_\gamma)=
(\widetilde{M}_\gamma,g\cdot\widetilde{\omega}_{\gamma})$.
\end{lemma}
\begin{proof}
 Denote by
$p:\widetilde{M}_\gamma\to M$ the covering map. It is enough to remark that for every
$g\in SL(2,\R)$ we get
$\widetilde{g\cdot\omega}_\gamma=p^*(g\cdot\omega)=g\cdot\omega\circ p_*=
g\cdot p^*(\omega)=g\cdot\widetilde{\omega}_{\gamma}$.
\end{proof}
\begin{proofof}{Proposition}{Fubinilattice}
To avoid undue repetition, we will write that a directional flow
satisfies (P-i) for $i \in \{1,2,3\}$, where (P-i) could be any of
the three properties (P-1), (P-2) or (P-3) in the statement of the
Lemma. The same proof indeed  applies for all three properties.
Since $(M,\omega_0)$ a lattice surface, we recall (see
\S\ref{Veech:sec}) that the $SL(2,\R)$-orbit of $(M,\omega_0)$
(denoted by $\mathscr{L}_{\omega_0}$) is closed in
$\mathcal{M}^{(1)}(M)$ and can be identified to
$SL(2,\R)/SL(M,\omega_0)$ by the map $\Phi:
SL(2,\R)/SL(M,\omega_0) \to \mathscr{L}_{\omega_0}$ that sends $g
\,SL(M,\omega_0)\in SL(2,\R)/SL(M,\omega_0)$ to $g\cdot
\omega_0\in \mathscr{L}_{\omega_0}$. Denote by $\mu_0$ the
canonical measure on $\mathscr{L}_{\omega_0}$.

Using the Iwasawa NAK
decomposition, if we denote as usual by
\begin{equation*}
g_t= \left(\begin{array}{cc} e^t & 0 \\ 0 & e^{-t}
\end{array}\right), \quad h_s= \left(\begin{array}{cc} 1 & 0 \\ s &
1 \end{array}\right), \quad \rho_\theta= \left(\begin{array}{cc}
\cos \theta & -\sin \theta \\ \sin \theta & \cos\theta
\end{array}\right)
\end{equation*}
we can choose an open neighbourhood  $\mathscr{U}\subset \mathscr{L}_0$ of
$\omega_0 $ of the form
\begin{equation*}
\mathscr{U} = \{  \omega  \in \mathscr{L}_0: \,
\omega  = h_s g_t \rho_\theta \cdot  \omega_0  \,
\text{where} \,  (t,s, \theta)\in (-\epsilon, \epsilon)^2 \times
S^1\}
\end{equation*}
for some $\epsilon>0$. %Denote by $\mu$ the image of the volume
%measure $\upsilon$ on the homogeneous space
%$SL(2,\R)/SL(M,\omega)$ by the map $\Phi$.
By assumption, for $\mu_0$-a.e.  $ \omega  \in \mathscr{U}$,
%a full  $\upsilon$-measure set of
%there exists a Borel subset $X_0\subset
%SL(2,\R)/SL(M,\omega)$ with full $\upsilon$-measure such that if
%$g \,SL(M,\omega)\in X_0$ then
the vertical flow $(\widetilde{\varphi}^v_t)_{t\in\R}$ on
$(\widetilde{M}_\gamma,\widetilde{\omega}_\gamma)$ satisfies (P-i).
Moreover, since $\mu_0$ is the pull-back by $\Phi$ of the Haar
measure on $ SL(2,\R)/SL(M,\omega_0)$ which is locally equivalent
to the product Lebesgue measure in the coordinates $(t,s,\theta)$,
it follows that for Lebesgue almost every $(t,s,\theta)\in
(-\epsilon, \epsilon)^2 \times S^1$, the vertical flow
$(\widetilde{\varphi}^v_t)_{t\in\R}$ on $ (\widetilde{M}_\gamma,
\widetilde{(h_s g_t \rho_\theta \cdot\omega_0)}_\gamma)$, which by
Lemma~\ref{invarianceK1covers} is
metrically isomorphic to $(\widetilde{M}_\gamma, h_s g_t \rho_\theta
\cdot\widetilde{(\omega_0)}_\gamma)$, also satisfies (P-i).

%Let $(M,\omega)$ be a translation surface (not necessary compact).

%Similarly, one can show that for every
%$\theta\in S^1$ and $g\in SL(2,\R)$ the flow
%$\left(\phi^{g\cdot\omega,g\cdot\theta}_{t}\right)_{t\in\R}$ has a finite measure invariant set
%ergodic if and only if
%$\left(\phi^{\omega,\theta}_{t}\right)_{t\in\R}$ does.

Denote by $S_0\subset S^1$ the subset of all $\theta\in S_0$ for
which the directional flow $\widetilde{\varphi}^\theta_t$ on
$(\widetilde{M}_\gamma, \widetilde{(\omega_0)}_{\gamma})$ does not
satisfy (P-i). By Lemma~\ref{changeg}, if $\theta\in S_0$ then
also the vertical flow $\widetilde{\varphi}_t^{v}$ on
$(\widetilde{M}_\gamma, \rho_{\pi/2-\theta} \cdot
\widetilde{(\omega_0)}_{\gamma})$ does not satisfy (P-i).
Moreover, since the vertical direction $\pi/2\in S^1$ is fixed
both by $h_s$ and $g_t$, i.e. $h_s \cdot \frac{\pi}{2} =
\frac{\pi}{2} $ and $g_t \cdot \frac{\pi}{2} = \frac{\pi}{2} $ for
any $s, t \in \mathbb{R}$, Lemma~\ref{changeg} also implies that
the flow $\widetilde{\varphi}_t^{{v}}$ on $(\widetilde{M}_\gamma,
 h_s g_t\rho_{\pi/2-\theta} \cdot \widetilde{(\omega_0)}_{\gamma})$
does not satisfy (P-i) for all $(t,s)\in (-\epsilon,\epsilon)^2$. It follows
that for every $(t,s,\theta)\in
(-\epsilon,\epsilon)^2\times(\pi/2-S_0)$ the vertical flow
$\widetilde{\varphi}_t^{{v}}$ on $(\widetilde{M}_\gamma,
 h_s g_t\rho_{\theta} \cdot \widetilde{(\omega_0)}_{\gamma})$
does not satisfy (P-i). Therefore the set
$(-\epsilon,\epsilon)^2\times(\pi/2-S_0)$ has zero Lebesgue
measure and hence $S_0$ has zero Lebesgue measure. Thus, we
conclude that for any  $\mathbb{Z}$-cover $(\widetilde{M}_\gamma,
\widetilde{(\omega_0)}_\gamma)$ of $(M, \omega_0)$ given by a
non-zero $\gamma \in K_1\cap H_1(M,\Z)$, for almost every $\theta
\in S^1$, the directional flow
$(\widetilde{\varphi}^\theta_{t})_{t\in\R}$ on
$(\widetilde{M}_\gamma, \widetilde{(\omega_0)}_{\gamma})$
satisfies (P-i).
\end{proofof}

\begin{proofof}{Theorem}{maintheorem}
Let  $(M,\omega_0)$ is a square-tiled surface of genus $2$. The
canonical probability measure $\mu_0$  on $\mathscr{L}_{\omega_0}$
is ergodic (see \S\ref{Veech:sec}) and, by Theorem \ref{Bain}, is
KZ-hyperbolic. Moreover, setting $K_1= H_1^{(0)}(M,\Q)$ and
$K_1^\perp= H_1^{st}(M,\Q)$ (see \S \ref{Veech:sec}), one can
check, as in the proof of Corollary \ref{theointro}, that the
assumptions of Theorem
 \ref{non-ergodicitycriterion} hold and that,
in view of (\ref{kerhol}), the recurrent
$\mathbb{Z}$-covers are exactly the $\mathbb{Z}$-covers
$(\widetilde{M}_\gamma, \widetilde{\omega}_\gamma) $ given by
$\gamma \in K_1 \cap H_1(M,\Z)$.  Thus, by Corollary \ref{cor_UMEC}, for
$\mu_0$-almost every $\omega \in \mathscr{L}_{\omega_0}$, for any
recurrent $\mathbb{Z}$-cover $(\widetilde{M}_\gamma,
\widetilde{\omega}_\gamma)$ of $(M, \omega)$ given by a non-zero
$\gamma$ the vertical flow $(\widetilde{\varphi}_t^{v})_{t\in\R}$
on $(\widetilde{M}_\gamma, \widetilde{\omega}_\gamma)$  is not
ergodic and has no invariant set of finite measure and has
uncountably many ergodic components. Thus, the claim  follows from
Proposition~\ref{Fubinilattice}.
\end{proofof}

\begin{proofof}{Corollary}{non-divergence}
Denote by $Z_{(3,0)}$  the square-tiled translation surface
corresponding to the polygon drawn in Figure~\ref{example3} with
edges labeled by the same letter identified by translations. One
can verify that $Z_{(3,0)} \in\mathcal{H}(2)$.
%\begin{figure}[h]
%\includegraphics[width=0.3\textwidth]{example3.eps}
%\caption{Polygon $Z_{(3,0)}$.}\label{example3}
%\end{figure}
Consider the homology class $\gamma = [B]-[D]$ which  is non
trivial but has trivial holonomy. One can check that the
$\Z$-cover of $Z_{(3,0)}$ associated to $\gamma$ gives exactly the
infinite staircase translation surface $Z_{(3,0)}^\infty$. Thus,
Theorem~\ref{maintheorem} applied to this surface shows that the
directional flow on $Z_{(3,0)}^\infty$ is not ergodic and has no
invariant set of finite measure for almost every direction.
\end{proofof}
\begin{remark}
A similar  proof shows that any surface in the family
$Z_{(a,b)}^\infty$ with $(a,b) \in \mathbb{N}^2$, $b>2$, described
by Hubert-Schmithüsen in \cite{Hu-Sch} satisfy the same conclusion
of Corollary \ref{non-divergence}.
\end{remark}
 %As a Corollary of our work we have:
%\begin{proposition}
%The Veech group $SL(Z_3^\infty, \mathbb{R})$ is of convergence type.
%\end{proposition}

\subsection{Non-ergodicity for  billiards in the infinite
strip.}\label{strip}
\begin{proofof}{Theorem}{stripbilliard}
Let us consider the billiard flow on the table $T(l)$ in Figure
\ref{fig_bil}. Denote by $\Gamma$ the $4$-elements group of
isometries of $S^1$ generated by the reflections $\theta\mapsto
-{\theta}$, $\theta\mapsto \pi -{\theta}$.
%the billiard flow on $T(l)$ has the invariant subset $T(l)\times (\Gamma\theta)$ in the phase space.
Using the unfolding process described in  \cite{Ka-Ze} (see for
example \cite{Ma-Ta}),  one can verify that, for every direction
$\theta\in S^1$  the flow $(b^{\theta}_t)_{t\in\R}$ is isomorphic
to the directional flow
$(\widetilde{\varphi}^{\theta}_t)_{t\in\R}$ on a non-compact
translation surface $(\widetilde{M},\widetilde{\omega}_l)$, where
$(\widetilde{M},\widetilde{\omega}_l)$ is the translation surface
resulting from gluing, along segments with the same name, four
copies of $T(l)$, one for each element of $\Gamma$, according to
the action of $\Gamma$, as shown in the Figure~\ref{unfold1}.
\begin{figure}[h]
\includegraphics[width=.76\textwidth]{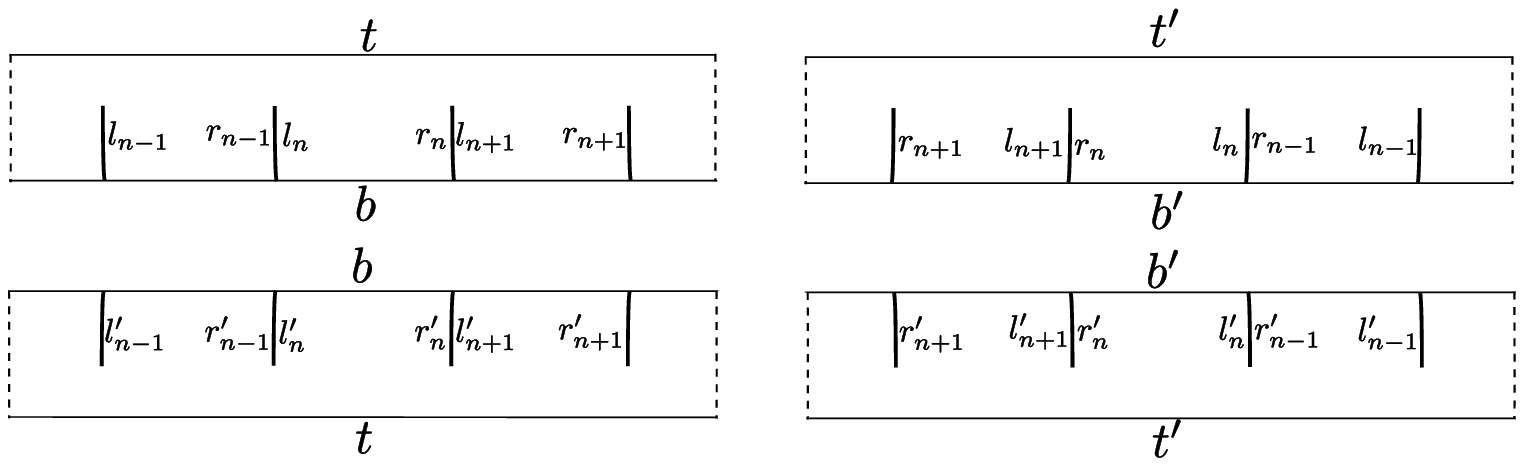}
\caption{}\label{unfold1}
\includegraphics[width=.66\textwidth]{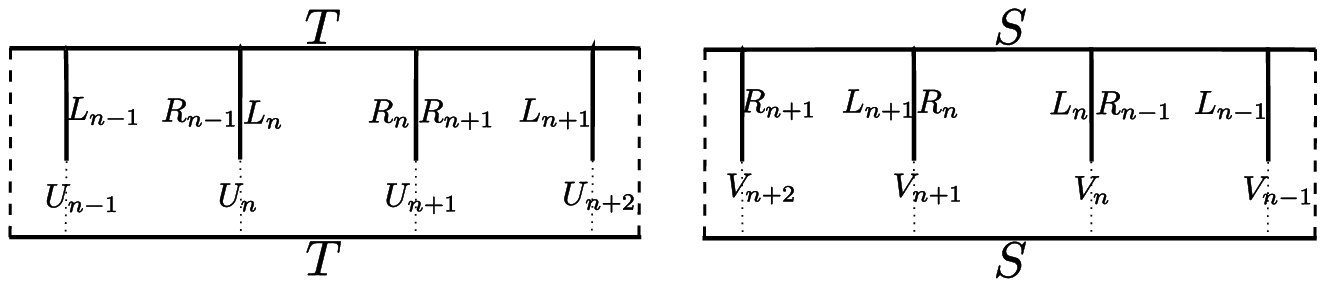}
\caption{}\label{unfold2}
\end{figure}
The surface $(\widetilde{M},\widetilde{\omega}_l)$ can be
represented as gluing two $\Z$-periodic polygons, as shown in the
Figure~\ref{unfold2}, where $R_n=r_n\cup r'_n$ and $L_n=l_n\cup
l'_n$.
Let  us cut these polygons along the segments marked as $U_n$,
$V_n$, $n\in\Z$, to obtain rectangles $P_n$, $P'_n$ and let us
glue $P_n$ and $P_n'$ along the segment $R_n$ (see the
Figure~\ref{unfold3}).
\begin{figure}[h]
\includegraphics[width=0.80\textwidth]{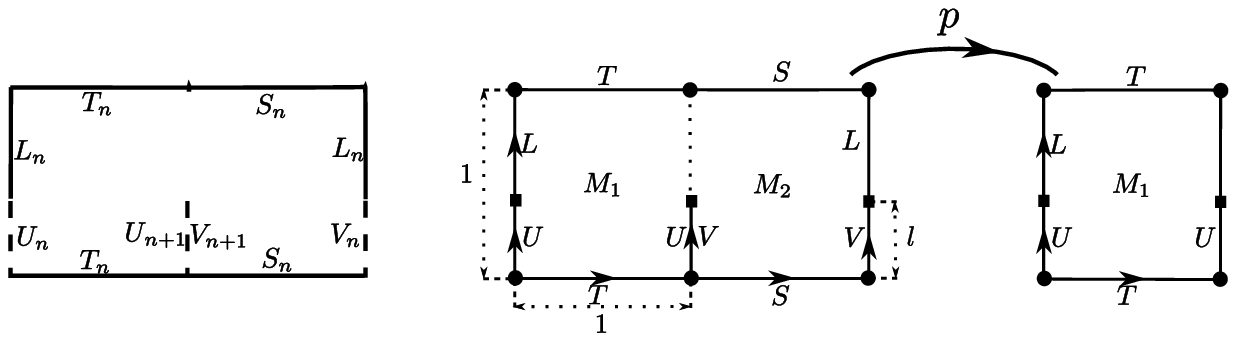}
\caption{}\label{unfold3}
\end{figure}
It follows that $(\widetilde{M},\widetilde{\omega}_l)$ is a
$\Z$-cover of the compact translation surface $(M,\omega_l)$
presented in the  Figure~\ref{unfold3}. More precisely,
$(\widetilde{M},\widetilde{\omega}_l)=
(\widetilde{M}_{\gamma},\widetilde{(\omega_l)}_{\gamma})$, where
$\gamma=[V-U]$ has trivial holonomy.

\subsubsection*{(1) Case $l$ rational.}
One can verify that for any $l\in (0,1)$, $(M,\omega_l)\in
\mathcal{H}(1,1)$, thus, in particular, $M$ has genus $2$. The
assumption that $l \in \mathbb{Q}$ guarantees that $(M,\omega_l)$
is square-tiled. Thus,  in this case we can apply
Theorem~\ref{maintheorem} that implies that for almost every
$\theta\in S^1$  the directional flow
$(\widetilde{\varphi}^\theta_t)_{t\in\R}$ on
$(\widetilde{M},\widetilde{\omega_l})$ and hence the billiard flow
$(b^\theta_t)_{t\in\R}$ on $T(l)$ is  not ergodic, has no
invariant sets of finite measure and has uncountably many ergodic
components.

\subsubsection*{(2) Full measure set of values of the parameter $l$.}
Let us remark that $(M, \omega_l)$  can be obtained from two
identical copies $(M_1, \omega_l^1)$, $(M_2, \omega_l^2)$
(corresponding to the two rectangles in Figure \ref{unfold3}) of a
genus $1$ translation surface with a slit (i.e. a straight segment
connecting two marked points), by identifying each side of the
slit in $(M_1, \omega_l^1)$ with the opposite side of the slit in
$(M_2, \omega_l^2)$. In particular, this shows that $(M,
\omega_l)$ is a \emph{branched $2$-cover of the torus} $(M_1,
\omega_l^1)$ with covering map given by the projection $p:M\to
M_1$. Denote by $\tau:M\to M$ the only non-trivial element of the
deck group of the covering $p:M\to M_1\approx \mathbb{T}^2$.
%For any $g \in SL(2,\R)$ to $(M, \omega_l)$, the  translation
%surface $(M, g\cdot  \omega_l)$ is again a branched cover of the
%torus since $(M , g\cdot \omega_l)$ can be clearly still
%constructed in the same way from the two identical slit tori
%$(M_1, g\cdot \omega_l^1)$, $(M_2, g\cdot \omega_l^2)$. Let
%$\mathscr{L}$ be the $SL(2,\R)$-orbit closure of  $(M, \omega_l)$
%in $\mathcal{H}(1,1)$.
%Clearly the locus $\mathscr{L} \subset \mathcal{H}(1,1)$ is a
%lower dimensional $SL(2, \R)$-invariant subset of
%$\mathcal{H}(1,1)$.
Denote by $\mathscr{L}$ the locus
\[\{ \omega \in \mathcal{H}^{(2)}(1,1):\tau^*\omega=\omega\}.\]
Equivalently, $\omega\in\mathscr{L}$ if and only if
$\omega=p^*\omega_0$ for some $\omega_0\in
\mathcal{H}^{(1)}(0,0)$, where  $\mathcal{H}^{(1)}(0,0)$ is the
stratum of a genus one translation surface with two marked points.
Therefore, $\mathscr{L}$ is the $2$-cover of the moduli space
stratum $\mathcal{H}^{(1)}(0,0)$ and therefore $\mathscr{L}$ has
dimension five, which is the dimension of
$\mathcal{H}^{(1)}(0,0)$. Moreover, $\mathscr{L}$ carries a
natural $SL(2, \R)$-invariant measure $\mu_{\mathscr{L}}$, which
is simply the pull-back of the canonical measure on the stratum
$\mathcal{H}^{(1)}(0,0)$ via the covering map $p$.  Let us
consider the decomposition $H_1(M,\Q)=K_1\oplus K_1^\perp$, where
\[K_1:=\{\gamma\in H_1(M,\Q):\tau_*\gamma=-\gamma\}\text{ and }
K_1^\perp:=\{\gamma\in H_1(M,\Q):\tau_*\gamma=\gamma\}.\] This is
an orthogonal decomposition. Indeed, if $\gamma_1\in K_1$ and
$\gamma_2\in K_1^\perp$ then
\[\langle\gamma_1,\gamma_2\rangle=\langle\tau_*\gamma_1,\tau_*\gamma_2\rangle=-\langle\gamma_1,\gamma_2\rangle
\quad\Longrightarrow\quad\langle\gamma_1,\gamma_2\rangle=0.\]
Moreover, $\dim_\Q K_1=\dim_\Q K_1^\perp=2$. Remark that the
homology class $\gamma=[V-U]$ which determines the
$\mathbb{Z}$-cover $(\widetilde{M}, \widetilde{\omega}_l)$ belongs
to $K_1$.

Let $p_*: H_1(M ,\Q) \to H_1(\mathbb{T}^2,\Q) $ be  the action
induced on $\Q$-homology by the covering map $p: M \to M_1$. If
$\tau_*\gamma=-\gamma$ then
$-p_*\gamma=p_*\tau_*\gamma=(p\circ\tau)_*\gamma=p_*\gamma$, hence
$K_1$ is a subspace of the kernel $\ker_\Q p_*$. Since $\dim_\Q
K_1=2=\dim_\Q \ker_\Q p_*$, we have $K_1= \ker_\Q p_*$.
%Let $K_1 \subset H_1(M, \mathbb{Q})$ be the kernel of $p_*$:
%\begin{equation*}
%K_1 = \{ \gamma \in H_1(M, \mathbb{Q}), \ \text{such \ that}  \
%p_* \gamma = 0 \in  H_1(M, \mathbb{Q})\}.
%\end{equation*}
%Clearly, $K_1$ has dimension two, since $dim_\Q H_1(M, \mathbb{Q})
%=4$  and $dim_\Q H_1( \mathbb{T}^2, \mathbb{Q}) =2$,  Moreover,
%$K_1$ is invariant under the $SL(2,\R)$-action on $\mathscr{L}$.
Let $\phi\in\Gamma(M)$ an element of the mapping-class group such
that $\omega_2=\phi^*\omega_1$ for
$\omega_1\,\Gamma(M)=\omega_2\,\Gamma(M)\in\mathscr{L}$. Then
there exists $\phi_0\in\Gamma(M_1)$ such that $p\circ
\phi=\phi_0\circ p$. It follows that  $p_*\gamma=0$ implies
$p_*(\phi_*\gamma)=(\phi_0)_*(p_*\gamma)=0$, so $\phi_*K_1= K_1$.
Since $K_1^\perp$ is the symplectic orthocomplement of $K_1$ in
$H_1(M,\Q)$, we obtain $\phi_*K_1^\perp=K_1^\perp$. Consequently,
\[\{ H_1(M_\omega,\Q)=K_1\oplus K_1^\perp, \; \omega\in\mathscr{L}\}\]
is an orthogonal invariant splitting which is constant on
$\mathscr{L}$. Let $\mathcal{K}_1$ and $\mathcal{K}_1^\perp$ be
the associated invariant subbundles over $\mathscr{L}$.

Since the canonical measure on $\mathcal{H}^{(1)}(0,0)$ is ergodic
for the Teichm\"uller flow (see \cite{Ma}) and $\mathscr{L}$ is a
connected cover of $\mathcal{H}^{(1)}(0,0)$ whose covering map is
equivariant with respect to the $SL(2,\R)$-action, it follows (for example by the Hopf argument) that
also the measure $\mu_{\mathscr{L}}$ on $\mathscr{L}$ is ergodic
for the Teichm\"uller flow. Thus, since $\mu_{\mathscr{L}}$ is an
$SL(2, \R)$-invariant  measure and ergodic for the Teichm\"uller
flow on $\mathcal{H}(1,1)$, which is a genus two stratum,
$\mu_{\mathscr{L}}$ is KZ-hyperbolic (see Theorem \ref{Bain}). In
particular, since there are no zero exponents, the Lyapunov
exponents of the invariant subbundle $\R\otimes_\Q\mathcal{K}_1$
(see \S\ref{Veech:sec}) are both non zero.
%  Let us remark that the symplectic orthogonal splitting $H_1(M, \R)= K_1 \oplus K_1^\perp$ induces a Poincar{\'e} dual splitting $H^1(M, \R)= K^1 \oplus K^1_\perp$  and hence a splitting of the Hodge bundle over $\mathscr{L}$ which is invariant under the action of the Kontsevich-Zorich cocycle.   The action of the Kontsevich-Zorich cocycle on $K^1_\perp \approx H^1(\mathscr{T}^2, \R)$ is via the usual $SL(2, \R)$ action, thus the corresponding Lyapunov exponents are trivial.  It follows that the exponents corresponding to $K^1$ are both non zero.
Thus, $\mathscr{L}$, $\mu_{\mathscr{L}}$ and  $\mathcal{K}_1$
satisfy all the assumptions of Theorem
\ref{non-ergodicitycriterion}. It follows that for there exists a
set $\mathscr{L}'\subset \mathscr{L}$ such that
$\mu_{\mathscr{L}}(\mathscr{L}')=1$ and for all $\omega \in
\mathscr{L}'$ and all non-zero $\gamma \in K_1\cap H_1(M,\Q)$, the
vertical flow $(\widetilde{\varphi}_t^{v})_{t\in\R}$ on
$(\widetilde{M}_\gamma, \widetilde{\omega}_\gamma)$ is not
ergodic, and by Corollary~\ref{cor_UMEC} that it has uncountably
many ergodic components. Let us now show that this allows to
deduce the desired conclusion by a Fubini argument.

%For any $0<l<1$, consider an open neighborhood $\mathcal{U}$ of $(M,\omega_l)$ in $\mathscr{L}$. Since $\mathscr{L}$ is a cover of $\mathcal{H}^{(1)}(0,0)$, local coordinates on $\mathscr{L}$ are given by the relative periods for the marked torus $(M_1, \omega_1)$ (see \S\ref{Veech:sec}). Thus, for a small $\mathcal{U}$, we can choose a base $\{\sigma , \gamma_1, \gamma_2  \}$ of $H_1(M_1,\Sigma, \R)$, for example, let $\gamma_1=[L]$, $\gamma_2=[T\cup S]$, be two sides of the rectangle in Figure \ref{} and let $\sigma=[U]$ be a relative period, and consider:
%\begin{equation}\label{coordinates}
%(x_1, x_2, x_3):= ( \int_{\sigma}\Re \omega , \int_{\gamma_1}\Re \omega,\int_{\gamma_2}\Re \omega), \ \  (y_1, y_2, y_3):= ( \int_{\sigma}\Im %\omega, \int_{\gamma_1}\Im \omega,\int_{\gamma_2}\Im \omega ).
%\end{equation}
Since  $\mathscr{L}$ is a $2$-cover of $\mathcal{H}^{(1)}(0,0)$,
local coordinates on $\mathscr{L}$ are given by the relative
periods for the marked torus $(M_1, \omega_l^1)$ (see
\S\ref{Veech:sec}).  We will deal with an open subset
$\mathcal{V}$  in $\mathscr{L}$ constructed as follows. Denote by
$\{ \gamma_1, \gamma_2,\gamma_3  \}$ the basis of
$H_1(M_1,\Sigma_1, \Z)$ given by $\gamma_1=[U]$ $\gamma_2=[U\cup
L]$, $\gamma_3=[T]$, see Figure \ref{unfold3}. Then $\{ \gamma_1,
\gamma_2,\gamma_3, \tau_*\gamma_1, \tau_*\gamma_2,\tau_*\gamma_3
\}$ is a family of generators of $H_1(M,\Sigma, \Z)$. Let us
consider
\begin{equation}\label{coordinates}
\begin{aligned} &(x_1, x_2, x_3):=  \Big( \int_{\gamma_1}\Re \omega ,
\int_{\gamma_2}\Re \omega,\int_{\gamma_3}\Re \omega\Big)=  \Big(
\int_{\tau_*\gamma_1}\Re \omega ,
\int_{\tau_*\gamma_2}\Re \omega,\int_{\tau_*\gamma_3}\Re \omega\Big),\\
&(y_1, y_2, y_3):= \Big( \int_{\gamma_1}\Im \omega,
\int_{\gamma_2}\Im \omega,\int_{\gamma_3}\Im \omega \Big)= \Big(
\int_{\tau_*\gamma_1}\Im \omega, \int_{\tau_*\gamma_2}\Im
\omega,\int_{\tau_*\gamma_3}\Im \omega \Big).
\end{aligned}
\end{equation}
Since we are considering  abelian differentials of area $2$, the
coordinates (\ref{coordinates}) are not all independent
($x_2y_3-x_3y_2=1$), but one of them, say $y_3$, is determined by
the area one requirement. Thus, $(\underline{x},
\underline{y}):=(x_1, x_2, x_3, y_1, y_2)$ are independent
coordinates on a subset of $\mathscr{L}$ and denote by
$\omega(\underline{x}, \underline{y})\in \mathscr{L}$ the
corresponding differential. Then $\omega(0,0,1,l,1)=\omega_l$ for
every $l\in(0,1)$.  Denote by $\mathcal{V}\subset\mathscr{L}$ the
open sets of all $\omega(\underline{x}, \underline{y})\in
\mathscr{L}$ with $x_1,x_2\neq 0$.

%We will consider neighbourhoods which have a product structure,
%that is of the form $\mathcal{U}=\{\omega(\underline{x},
%\underline{y}), \underline{x} \in  A, \underline{y} \in B\}$, were
%$A\subset \R^3, B \subset \R^2$ are open sets, and call them
%\emph{product neighbourhooods}.

%Since we are considering abelian differentials of \emph{unit
%area}, the coordinates (\ref{coordinates}) are not all
%independent, but one of them, say $x_3$, is determined by the area
%one requirement. Let us hence choose $(\underline{x},
%\underline{y}):=(x_1, x_2, y_1, y_2, y_3)$ as independent
%coordinates on $\mathcal{U}$ and let $\omega(\underline{x},
%\underline{y})$ be the corresponding differential. WIthouBy
%definition of the measure $\mu_\mathcal{L}$ and of the canonical
%measure $\mu_{\mathcal{H}^{(1)}(0,0)}$ (see \ref{Veech:sec}),
%since $\mu_{\mathscr{L}'}(\mathscr{L}'\cap \mathcal{U}) =
%\mu_{\mathscr{L}'}(\mathcal{U})$, for \emph{Lebesgue almost every}
%$(x_1, x_2, y_1, y_2, y_3)$ parametrizing $\mathcal{U}$,  $\omega(
%\underline{x}, \underline{y}) \in \mathscr{L}'$.
Fix a non-zero $\gamma\in K_1\cap H_1(M,\Z)$. Recall that, in view
of \S\ref{reduction:sec} (see Lemma~\ref{lem_flow_auto} and choose
$I$ as at the end of \S\ref{reduction:sec} so that
(\ref{parameters}) holds), for every $\omega\in\mathscr{L}$ there
exists a horizontal interval $I\subset M$ and $\gamma_\alpha\in
H_1(M,\Z)$, $\xi_\alpha\in H_1(M,\Sigma,\Z)$ for
$\alpha\in\mathcal{A}$ such that the vertical flow
$(\widetilde{\varphi}^v_t)_{t\in\R}$ on
$(\widetilde{M}_\gamma,\widetilde{\omega}_\gamma)$ has a special
representation built over the skew product $T_\psi:I\times\R\to
I\times\R$ such that for every $\alpha\in \mathcal{A}$
\[\lambda_\alpha =\int_{\xi_\alpha}\Re\omega\text{ and }
\psi(x)=\langle\gamma,\gamma_\alpha\rangle,\quad
Tx=x+\int_{\gamma_\alpha}\Re\omega
\text{ for } x\in I_\alpha.\]
For every $(M,\omega_0)\in\mathcal{V}$ we can
choose a neighbourhoood $\mathcal{U}\subset \mathcal{V}$ of
$\omega_0$ such that $\gamma_\alpha$ and $\xi_\alpha$,  for
$\alpha\in\mathcal{A}$, do not depend on $\omega\in\mathcal{U}$.

Let us adopt the following convention: let us say that a flow has
property (P-1) if it is not ergodic and property (P-2) if it has
uncountably many ergodic components. We claim that, if  $\omega_1
= \omega(\underline{x}_1, \underline{y}_1), \omega_2 =
\omega(\underline{x}_2, \underline{y}_2)\in \mathcal{U}$ with
$\underline{x}_1=\underline{x}_2$, then the vertical flow
$(\widetilde{\varphi}_t^{v})_{t\in\R}$ on
$(\widetilde{M}_\gamma\widetilde{(\omega_1)}_\gamma)$ has property
(P-i) for i $ \in \{1,2\}$ if and only if the vertical flow
$(\widetilde{\varphi}_t^{v})_{t\in\R}$ on
$(\widetilde{M}_\gamma\widetilde{(\omega_2)}_\gamma)$ has property
(P-i). Indeed, if $\underline{x}_1=\underline{x}_2$
 then
$\int_{\gamma_i}\Re\omega_1=\int_{\gamma_i}\Re\omega_2$ and
$\int_{\tau_*\gamma_i}\Re\omega_1=\int_{\tau_*\gamma_i}\Re\omega_2$
for $i=1,2,3$. Thus
$\int_{\gamma_\alpha}\Re\omega_1=\int_{\gamma_\alpha}\Re\omega_2$,
$\int_{\xi_\alpha}\Re\omega_1=\int_{\xi_\alpha}\Re\omega_2$
for all $\alpha\in\mathcal{A}$. It follows that both vertical
flows have special representations built over the same skew
product, which proves our claim.

Let us consider the diffeomorphism
$\Upsilon:(0,1)\times((0,2\pi)\setminus\{\pi/2,\pi,3\pi/2\})\times\R^3\to\R^5$
\[\Upsilon(l,\theta,t,y_1,y_2)=(-e^tl\cos\theta,-e^t\cos\theta,e^t\sin\theta,
e^{-t}(y_1+l\sin\theta),e^{-t}(y_2+\sin\theta)).\] The diffeomorphism $\Upsilon$ is defined so that we have
\bes
g_t\rho_{\pi/2-\theta}\omega_l=\omega(\Upsilon(l,\theta,t,0,0)), \qquad \forall l\in [0,1], \theta \in S^1, t \in \R.
\ees
Denote by $\mathcal{V}_0 \subset (0,1)\times((0,2\pi)\setminus\{\pi/2,\pi,3\pi/2\})\times\R \times \R^2$ (respectively $\mathscr{L}'_0$) the preimage of
$\mathcal{V}$ (respectively $\mathscr{L}'$) by the map
$(l,\theta,t,\underline{y})\mapsto\omega(\Upsilon(l,\theta,t,\underline{y}))$
and by $\mu_0$ the pullback of $\mu_{\mathscr{L}}$ by this map.
Since $\Upsilon$ is a diffeomorphism, the measure $\mu_0$ is
equivalent to the Lebesgue measure on $\mathcal{V}_0$, hence
$\mathcal{V}_0\setminus\mathscr{L}'_0$ has zero Lebesgue measure.

For $i=1,2$, denote by $\neg \mathcal{P}_i\subset(0,1)\times((0,2\pi)\setminus\{\pi/2,\pi,3\pi/2\})$ the set
of all $(l,\theta)$ such that the directional flow
$(\widetilde{\varphi}^\theta_t)_{t\in\R}$ on
$(\widetilde{M}_\gamma,\widetilde{(\omega_l)}_\gamma)$ does not have property (P-i).
We claim that $\neg \mathcal{P}_i$ has zero Lebesgue measure. If fact, we need to
show that for every $(l,\theta)\in \neg \mathcal{P}_i$ there exists a neighbourhood
$(l,\theta)\in \mathcal{U}$ such that $\neg \mathcal{P}_i\cap\mathcal{U}$ has zero
Lebesgue measure.

Fix $(l_0,\theta_0)\in \neg \mathcal{P}_i$. By Lemmas~\ref{invarianceK1covers}~and~\ref{changeg},
$(\widetilde{\varphi}^v_t)_{t\in\R}$
on $(\widetilde{M}_\gamma,\widetilde{(\rho_{\pi/2-\theta_0}\cdot\omega_{l_0})}_\gamma)$ is metrically isomorphic to $(\widetilde{\varphi}^{\theta_0}_{st})_{t\in\R}$ on
$(\widetilde{M}_\gamma,\widetilde{(\omega_l)}_\gamma)$ for some $s>0$,
and hence also does \emph{not} have property (P-i).
Since $\rho_{\pi/2-\theta_0}\cdot\omega_{l_0}\in\mathcal{V}$,
there exists a neighbourhood
of $\rho_{\pi/2-\theta_0}\cdot\omega_{l_0}\in\mathcal{U}$ such that for all
$\omega(\underline{x}_1,\underline{y}_1),\omega(\underline{x}_2,\underline{y}_2)\in\mathcal{U}$
with $\underline{x}_1=\underline{x}_2$ the vertical flows on
$(\widetilde{M}_\gamma,\widetilde{(\omega_1)}_\gamma)$ and $(\widetilde{M}_\gamma,\widetilde{(\omega_2)}_\gamma)$
have special representations over the same skew product.
Let $\mathcal{U}_1 \ni (l,\theta)$, $(-\vep,\vep)$ and $\mathcal{U}_2 \ni (0,0)$ be neighbourhoods
such that $\Upsilon(\mathcal{U}_1\times(-\vep,\vep)\times\mathcal{U}_2)\subset\mathcal{U}$.
We claim that
\begin{equation}\label{empty_inter}
(\neg \mathcal{P}_i\cap\mathcal{U}_1)\times(-\vep,\vep)\times\mathcal{U}_2\cap\mathscr{L}'_0=\emptyset.
\end{equation}
Indeed, if $(l,\theta)\in \neg \mathcal{P}_i\cap\mathcal{U}_1$ then $(\widetilde{\varphi}^v_t)_{t\in\R}$
on $(\widetilde{M}_\gamma,\widetilde{(\rho_{\pi/2-\theta}\cdot\omega_{l})}_\gamma)$ does not have property (P-i).
 Moreover, $\rho_{\pi/2-\theta}\cdot\omega_{l}=
\omega(\Upsilon(l,\theta,0,0,0))$ and $\Upsilon(l,\theta,0,0,0)\in\mathcal{U}$.
Therefore, for every $\underline{y}\in\mathcal{U}_2$
the vertical flow on $(\widetilde{M}_\gamma,\widetilde{\omega(\Upsilon(l,\theta,0,\underline{y}))}_\gamma)$
does not have property (P-i). Since every $g_t$ fixes the vertical direction, by
Lemmas~\ref{invarianceK1covers}~and~\ref{changeg},
the vertical flow on $(\widetilde{M}_\gamma,\widetilde{(g_t\cdot\omega(\Upsilon(l,\theta,0,\underline{y})))}_\gamma)$
does not have property (P-i) for every $t\in(-\vep,\vep)$. Since $g_t\cdot\omega(\Upsilon(l,\theta,0,\underline{y}))=\omega(\Upsilon(l,\theta,t,\underline{y}))$, it follows that
the vertical flow on
$(\widetilde{M}_\gamma,\widetilde{(\omega(\Upsilon(l,\theta,t,\underline{y})))}_\gamma)$
does not have property (P-i) for every $(l,\theta,t,\underline{y})\in(\neg \mathcal{P}_i\cap\mathcal{U}_1)\times(-\vep,\vep)\times\mathcal{U}_2\subset\mathcal{V}_0$,
which proves \eqref{empty_inter}. In view of the fact that $\mathcal{V}_0\setminus\mathscr{L}'_0$ has
zero Lebesgue measure, the product set $(\neg \mathcal{P}_i\cap\mathcal{U}_1)\times(-\vep,\vep)\times\mathcal{U}_2$
and hence $\neg \mathcal{P}_i\cap\mathcal{U}_1$ has zero Lebesgue measure.

Thus, we conclude that for every non-zero $\gamma\in K_1\cap
H_1(M,\Z)$ there exists a set $\Lambda \subset (0,1)$ of full
Lebesgue measure such that for every $l\in \Lambda$ for almost
$\theta \in S^1$ the directional flow
$(\widetilde{\varphi}^\theta_t)_{t\in\R}$ on the
$\mathbb{Z}$-cover
$(\widetilde{M}_\gamma,\widetilde{(\omega_l)}_\gamma)$ have both
properties (P-1) and  (P-2). This in particular applies to the
$\mathbb{Z}$-cover that is given by $\gamma=[V-U] \in K_1$.
Consequently, for any $l\in \Lambda$ the billiard flow
$(b_t^\theta)_{t\in\R}$ on $T(l)$ is not ergodic  and it has
uncountably many ergodic components  for almost every direction
$\theta \in S^1$.
\end{proofof}

\subsection{Non-ergodicity of the Erhenfest windtree model.}\label{Ehrenfest:finalsec}
Let us now prove Theorem \ref{Ehrenfestthm} and Corollary \ref{Ehrenfestcor}.
%We will use the recent work by \cite{DHL} in which the Lyapunov exponents for the the unfolding of the Ehrenfest billiard flow to a $\mathbb{Z}^2$-cover of a translation surface  and the corresponding Lyapunov exponents are

\begin{proofof}{Theorem}{Ehrenfestthm}
Let  us consider the $\mathbb{Z}$-periodic Ehrenfest billiard flow
$(e_t^\theta)_{t\in\R}$ on the tube $E_1(a,b)$ in Figure
\ref{Ehrenfesttube}. Let us denote  by $\Gamma$ the $4$-elements
group of isometries of the plane generated by $\langle \tau^h,
\tau^v\rangle $, where $\tau^h$ denotes the horizontal  reflection
$(x,y)\mapsto (x,-y)$ and $\tau^v$ denotes the vertical reflection
$(x,y)\mapsto (-x,y)$ ($\Gamma$ is  the Klein four-group
$\mathbb{Z}_2 \times \mathbb{Z}_2$). By the unfolding process (see
\cite{Ka-Ze}),  for every direction $\theta\in S^1$ the flow
$(e^{\theta}_t)_{t\in\R}$ on $E_1(a,b)$ is isomorphic to the
directional flow $(\widetilde{\varphi}^{\theta}_t)_{t\in\R}$ on a
non-compact translation surface
$(\widetilde{M},\widetilde{\omega}_{a,b})$ which is obtained by
gluing  four copies of $E_1(a,b)$, one for each element of the
group $\Gamma$, according to action of $\Gamma$. This translation
surface is a $\mathbb{Z}$-cover of a compact translation surface
$({M},{\omega}_{a,b})$ shown in Figure \ref{fundamental domain}
and the cover is given by $\sigma=v_{00}-v_{10}+v_{01}-v_{11}\in
H_1(M,\Z)$ (referring to the labelling of Figure \ref{fundamental
domain}). The surface $M$ is glued from four copies of a
fundamental domain $F(a,b):=E_1(a,b)\cap([0,1)\times(\R/\Z))$ for
the natural $\mathbb{Z}$-action (generated by the translation by
the vector $(1,0)$) on the tube $E_1(a,b)$.
\begin{figure}[h]
\includegraphics[width=1\textwidth]{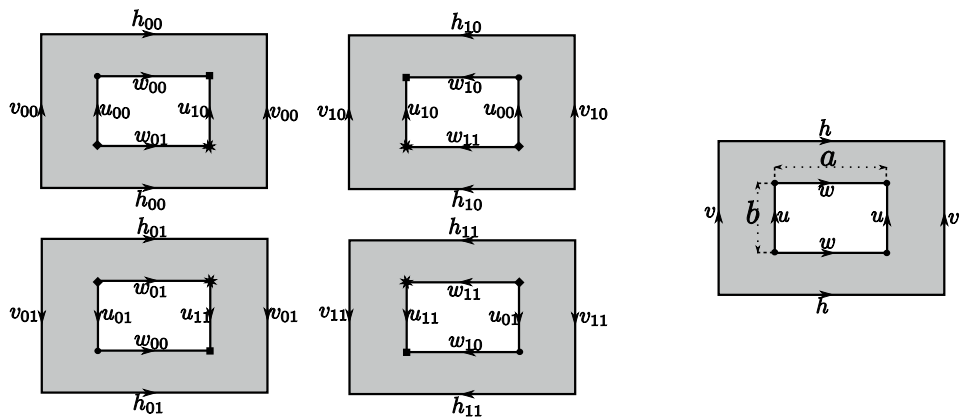}
\caption{Translation surfaces $(M,\omega_{a,b})$ and $(N,\nu_{a,b})$}\label{fundamental
domain}
\end{figure}
Thus, if we denote by $({N},{\nu}_{a,b})$ the translation surface
obtained from the fundamental domain $F(a,b)$ gluing the sides
according to the identifications in Figure \ref{fundamental
domain},  the translation surface $({M},{\omega}_{a,b})$ is a
cover of $({N},{\nu}_{a,b})$ with the deck group $\Gamma$. Let us
denote by $p:M\to N$ the covering map\footnote{We remark that this
surface  is the same that the surface is obtained by considering a
fundamental domain for the $\mathbb{Z}^2$-action on the planar
billiard table $E_2(a,b)$, which   is described in detail in
\cite{DHL} (see \S3).}. One can check that $({N},{\nu}_{a,b})$ has
genus two and belongs to the stratum $\mathcal{H}(2)$, while
$({M},{\omega}_{a,b})$ has genus $5$ and belongs to
$\mathcal{H}(2,2,2,2)$. By abuse of notation, we continue to write
$\omega_{a,b}$ for
$\omega_{a,b}/A(\omega_{a,b})=\omega_{a,b}/(4(1-ab))\in
\mathcal{H}^{(1)}(2,2,2,2)$. Let
\[\mathscr{L}=\{\omega\in \mathcal{H}^{(1)}(2,2,2,2):\omega=\frac{1}{4} p^*\nu,\,\nu\in \mathcal{H}^{(1)}(2)\}.\]
%be the $SL(2,\R)$-orbit closure  of
%$(M,{\omega}_{a,b})$ in $$. Since for any $g\in SL(2,\R)$ the
%translation surface $g\cdot (M,{\omega}_{a,b})$ is a
%$\Gamma$-cover of $g \cdot ({N},{\nu}_{a,b})$,
Then $\mathscr{L}$ is a closed $SL(2,\R)$-invariant subset of
$\mathcal{H}^{(1)}(2,2,2,2)$ which is a finite connected cover of
$\mathcal{H}^{(1)}(2)$ and $\omega_{a,b}\in\mathscr{L}$. The orbit
closures and the $SL(2,\R)$-invariant measures on
$\mathscr{H}^{(1)}(2)$ were classified by McMullen in \cite{McM}
and give a classification of orbit closures and the
$SL(2,\R)$-invariant measures on $\mathscr{L}$. From \cite{McM}
(see also \cite{DHL}), it follows that if $(a,b)$ satisfy
assumption $(1)$ or $(2)$ in Theorem \ref{Ehrenfestthm},
$(M,{\omega}_{a,b})$ is a Veech surface and its $SL(2,\R)$-orbit
is closed and carries the canonical $SL(2,\R)$-invariant measure.
Let us consider the $SL(2,\R)$-invariant measure $\mu_\mathscr{L}$
on $\mathscr{L}$ obtained by pull back by the finite covering map
of the canonical measure on $\mathcal{H}^{(1)}(2)$. Since the
canonical measure is ergodic and the cover $\mathscr{L}$ is
connected, each of these measures on $\mathscr{L}$ is ergodic.

Let $\tau^h_*, \tau^v_*$  be the maps induced on the homology
$H_1(M, \mathbb{Z})$ by the actions of the reflections $\tau^h,
\tau^v$ on $(M,{\omega}_{a,b})$. Consider the following orthogonal
decomposition
\begin{equation}\label{splittingEhrenfest}
\begin{split}
& H_1(M, \mathbb{Q})  = E^{++} \oplus E^{+-} \oplus E^{-+} \oplus
E^{--}, \quad \text{where, for} \ s_0,s_1 \in \{+,-\} , \\
& E^{s_0 s_1}=\{\gamma\in H_1(M, \mathbb{Q}):\tau_*^v(\gamma) = s_0 \gamma\text{ and }\tau_*^h(\gamma)
= s_1 \gamma\}.\end{split}\end{equation}
%For example, the $E^{++}$ contains all homology classes which are
%fixed by both reflections.
Remark that (\ref{splittingEhrenfest}) defines an invariant
orthogonal splitting constant  on $\mathscr{L}$.
%, that is, if
%$\gamma \in E^{s_0,s_1}$, $g_* \gamma \in E^{s_0,s_1}$ where $g$
%is the action on homology induced by the linear action of  any $g
%\in SL(2,\R)$ on an element of $\mathscr{L}$.

One can  check that the homology class $\sigma$ which determines
the $\mathbb{Z}$-cover $(\widetilde{M},\widetilde{\omega}_{a,b})$
of $({M},{\omega}_{a,b})$ belongs to the subspace $E^{-+}$ and
that the space $E^{-+}$ has dimension two (we refer for details to
\cite{DHL}, see Lemma 3 and Lemma 4). Moreover, the Lyapunov
exponents of the KZ cocycles for all the $SL(2,\R)$-invariant
ergodic measures on $\mathscr{L}$ were computed in \cite{DHL} (in
particular the exponents corresponding to $E^{-+}$) and turn out
to be all non-zero.

Given any parameter $(a,b)\in (0,1)^2$ let $\mu_{a,b}$ be the
canonical measure for a Veech surface (see \S\ref{Veech:sec}) if
$(a,b)$ satisfy the assumptions $(1)$ or $(2)$ or
$\mu_\mathscr{L}$ otherwise. Then, all the assumptions of Theorem
\ref{non-ergodicitycriterion} are satisfied by taking
$\mu:=\mu_{a,b}$ and $K_1:= E^{-+}$. It follows from
Corollary~\ref{cor_UMEC}  that there exists a set $\mathscr{L}'$
contained in the $SL(2,\R)$-orbit closure of $(M, \omega_{a,b})$
such that $\mu (\mathscr{L}')=1$ and for all $\omega \in
\mathscr{L}'$, for any   $\mathbb{Z}$-cover
$(\widetilde{M}_\gamma,\widetilde{\omega}_\gamma)$ with $\gamma\in
E^{-+}$ the vertical flow $(\widetilde{\varphi}^v_t)_{t \in \R}$
is \emph{not-}ergodic and it has uncountably many ergodic
components.

If $({M},{\omega}_{a,b})$ is a Veech surface, that is for $(a,b)$
as in $(1)$ or $(2)$, Proposition \ref{Fubinilattice} allows to
conclude the proof.  Therefore, from now on we consider the case
$\mu = \mu_\mathscr{L}$ and use a different Fubini argument to
prove the conclusion of the Theorem  for a full measure set of
parameters $(a,b)$. The arguments are similar to the proof of
Theorem \ref{stripbilliard} and also to the Fubini argument used
by \cite{DHL} in \S6.

Let us consider local coordinates
$(\underline{x},\underline{y})=(x_1, x_2,x_3, x_4,y_1, y_2,y_3,
y_4)$ on $\mathscr{L}$ given by period coordinates as follows
\[x_i=\int_{\gamma^i_{jk}}\Re \omega\text{ and }y_i = \int_{\gamma^i_{jk}}\Im \omega\text{ for }i=1,2,3,4\text{ and }j,k\in\{0,1\},\]
where $\gamma^1_{jk}=w_{jk}$, $\gamma^2_{jk}=u_{jk}$,
$\gamma^3_{jk}=h_{jk}$, $\gamma^4_{jk}=v_{jk}$ for $j,k\in\{0,1\}$
is a family of generators in $H_1(M,\Sigma,\Z)$. Since we are
considering  abelian differentials of \emph{unit area}, the
coordinates (\ref{coordinates}) are not all independent, but one
of them, say $y_4$, is determined by the area one requirement.
Thus, $(\underline{x}, \underline{y}):=(x_1, x_2, x_3, x_4, y_1,
y_2, y_3)$ are independent coordinates on a subset of
$\mathscr{L}$. Let $\omega(\underline{x}, \underline{y})$ be the
corresponding differential. Then
$\omega\big(\frac{1}{4(1-ab)}(a,0,1,0,0,b,0)\big)=\omega_{a,b}$
for every $(a,b)\in(0,1)^2$. Let as consider the local
diffeomorphism
$\Upsilon:(0,1)^2\times((0,2\pi)\setminus\{\pi/2,\pi,3\pi/2\})\times\R^4\to\R^7$,
\begin{align*}&\Upsilon(a,b,\theta,t,y_1,y_2,y_3) =\frac{1}{4(1-ab)} \cdot \\
&(e^t(a\sin\theta,
-b\cos\theta,\sin\theta,-\cos\theta),
e^{-t}(y_1+a\cos\theta,y_2+b\sin\theta,y_3+\cos\theta)).
\end{align*}
Then
$g_t\rho_{\pi/2-\theta}\omega_{a,b}=\omega(\Upsilon(a,b,\theta,t,0,0,0))$
and  the pullback of the measure $\mu_{\mathscr{L}}$ by the map
$(a,b,\theta,t,\underline{y})\mapsto\omega(\Upsilon(a,b,\theta,t,\underline{y}))$
is equivalent to the Lebesgue measure restricted to the domain of
the map.

As in the proof of Theorem \ref{stripbilliard}, let us say that a
flow has property (P-1) if it is not ergodic and (P-2) if it has
uncountably many ergodic components and let us denote by
$\neg\mathcal{P}_i \subset(0,1)^2\times(0,2\pi)$ the set of all
$(a,b,\theta)$ such that the directional flow
$(\widetilde{\varphi}^\theta_t)_{t\in\R}$ on
$(\widetilde{M}_\sigma,\widetilde{(\omega_{a,b})}_\sigma)$ does
not have property (P-i) for $i=1,2$. The same argument as in the
proof of Theorem~\ref{stripbilliard} shows that for every
$(a,b,\theta)\in \neg \mathcal{P}_i$ there exits neighbourhoods
$\mathcal{U}_1 \ni (a,b,\theta)$, $\mathcal{U}_2 \subset\R^4$ such
that for every $\omega\in \omega(\Upsilon((\neg
\mathcal{P}_i\cap\mathcal{U}_1)\times \mathcal{U}_2))$ the
vertical flow on
$(\widetilde{M}_\sigma,\widetilde{\omega}_\sigma)$ does not have
(P-i). Therefore the set $\omega(\Upsilon((\neg
\mathcal{P}_i\cap\mathcal{U}_1)\times
\mathcal{U}_2))\subset\mathscr{L}$ has zero $\mu_{\mathscr{L}}$
measure. It follows that $(\neg
\mathcal{P}_i\cap\mathcal{U}_1)\times \mathcal{U}_2$ and hence
$\neg \mathcal{P}_i\cap\mathcal{U}_1$ has zero Lebesgue measure.
Thus, for $i\in \{1,2\}$,  $\neg \mathcal{P}_i\subset
(0,1)^2\times(0,2\pi)$ has zero Lebesgue measure. Consequently,
for almost every $(a,b)\in(0,1)^2$ for almost every $\theta$ the
directional flow $(\widetilde{\varphi}^\theta_t)_{t\in\R}$ on
$(\widetilde{M}_\sigma,\widetilde{(\omega_{a,b})}_\sigma)$ is not
ergodic and has uncountably many ergodic components.
\end{proofof}

\begin{proofof}{Corollary}{Ehrenfestcor}
Let us remark that the billiard flow $(e_t^\theta)_{t\in \R}$ on
the planar Ehrenfest model  $E_2(a,b)$ projects on the the
billiard flow $(e_t^\theta)_{t\in \R}$ on the one-dimensional
Ehrenfest table  $E_1(a,b)$, via the map $\pi: \R^2 \to \mathbb{R}
\times \R/\Z$ given by $\pi(x,y) = (x, y +\Z)$. In other words,
$(e_t^\theta)_{t\in \R}$ on $E_1(a,b)$ is  a factor of
$(e_t^\theta)_{t\in \R}$ on $E_2(a,b)$.  It follows that if
$(e_t^\theta)_{t\in \R}$ on $E_1(a,b)$ is not ergodic and has
uncountably many ergodic components, also the flow
$(e_t^\theta)_{t\in \R}$ on   $E^2_{a,b}$ is not ergodic and has
uncountably many ergodic components. Thus,
Corollary~\ref{Ehrenfestcor} follows immediately from
Theorem~\ref{Ehrenfestthm}.
\end{proofof}

\appendix

\section{Stable space and coboundaries.}\label{coh:sec}
In this Appendix we include for completeness the proof  of Lemma
\ref{comparisontimes:lemma} and Theorem \ref{cohthm} (see
\S\ref{comparisontimes:lemma})  along the lines of \cite{Zo:how,
For-dev} (see also \cite{DHL}). Let us first introduce some
notation and describe how to construct a section $\mathcal{K}$ for the Teichm\"uller flow
which will be useful in both proofs. Some of the properties of
$\mathcal{K}$ will not be used in the proof of Lemma
\ref{comparisontimes:lemma}, but only in the proof of Theorem
\ref{cohthm}.

\subsubsection*{A section for the Teichm\"uller geodesic flow}
Let $\mu$ be any $SL(2,\R)$-invariant probability measure on the
moduli space  $\mathcal{M}^{(1)}(M)$ ergodic for the
Teichm\"uller flow. Since $\mu$  is $SL(2,\R)$-invariant, we can
assume that it is supported on a stratum $\mathcal{H}^{(1)}
=\mathcal{H}^{(1)} (k_1, \dots, k_\kappa) $ for some $k_1, \dots,
k_\kappa$. Let us remark that since $\mu$ is finite  and
ergodic for the Teichm\"uller flow,  by Oseledets' theorem,
$\mu$-almost every $\omega \in \mathcal{H}^{(1)}$ is Oseledets
regular for the Kontsevich-Zorich cocycle $(G^{KZ}_t)_{t\in\R}$.
Moreover,  there exists a $(G_t)_{t\in\R}$-invariant set
$\mathcal{H}_0\subset {\mathcal{H}}^{(1)}$ of $\mu$-measure one
such that each $\omega\in\mathcal{H}_0$ has no vertical and
horizontal saddle connections and  both the vertical and
horizontal flow on $(M, \omega)$ are ergodic (see \cite{Ma:IET}).
%In particular, both vertical and horizontal flows are also minimal
%(since the area induced by $\omega$ gives positive measure to open
%sets, see \cite{Co-Fo-Si}).

Choose a point $\omega_0\in \mathcal{H}_0 $ which is
Oseledets regular and  in the support of the measure $\mu$.
% (such a point exists since all the previous conditions holds on sets of $\mu$-measure one).
Consider the  vertical flow $(\varphi_t)_{t\in \R}$ on $(M,
\omega_0)$, where for brevity $\varphi_t:=
\varphi_t^{\omega_0,{v}}$. Let $M_{reg} =M_{reg,\omega_0}$ be the
set of points which are regular both for the vertical and
horizontal flow on $(M,\omega_0)$ (that, we recall, means that
both flows are defined for all times).  Remark that $M_{reg}$ has
full measure on $M$ and is invariant under $(G_t)_{t \in \R}$,
that is, $M_{reg, G_t\omega} = M_{reg} $ for all $t \in \R$.
Choose also a regular point $p_0 \in M_{reg}$.  The definition of
the section $\mathcal{K}$ depends on the choice of $\omega_0$ and
$p_0$, but $\omega_0$ and $p_0$ will play no role.

Let us denote by $I_{\omega_0}(p_0)$  the arc  of the horizontal
flow on $(M, \omega_0)$ of total length $1$ centered at  $p_0$.
%For any  $p \in M_{reg}$, denote by  $\gamma_s =
%\gamma_s(p,\omega)$ the \emph{unparametrized} curve given by the
%trajectory of  $(\varphi_t)_{t\in\R}$ of length $s$  starting at
%$p$.
%For  the function $f= i_{X_{v}} \rho $, since $\gamma_t$ is a
%vertical trajectory, we have that
%\begin{equation}\label{ergint}
%\int_0^t f(\varphi_s p) d s =  \int_{\gamma_t} \rho.
%\end{equation}
For any $q \in I_{\omega_0}(p_0)$ let us denote by
$\tau({\omega_0},q)$ the \emph{first return time} of $q$ to
$I_{\omega_0}(p_0)$ under the vertical flow $\varphi_t$.
%One can show
%(see Lemma 9.2'  in \cite{For-dev}) that there exists a measurable
%function $\overline{\tau}: \mathcal{M}^{(1)}(M)\to \R^+$ such that
%\be \tau(\omega,p)\leq \overline{\tau}(\omega) \qquad \text{ for\
%any}\ p \in M_{reg, \omega}.
%\end{equation}
 The Poincar{\'e} map of the flow
$(\varphi_t)_{t\in \R}$ to $I_{\omega_0}=I_{\omega_0}(p_0)$ is an IET that we
will denote by $T=T_{\omega_0,p_0}: I_{\omega_0} \to I_{\omega_0}$.
Let us denote by $I_j= I_j({\omega_0})$, $j=1, \dots, m$, the
subintervals exchanged by $T$, by $\lambda_j=
\lambda_j({\omega_0})$ their lengths and by $\tau_j=
\tau_j({\omega_0})$ the first return time of any $q \in I_i$ to
$I_{\omega_0}$.  Remark that
%since $p_0$ is regular, $p_0$ belongs tothe interior of one of the exchanged intervals. Moreover,
since $I_{\omega_0}(p_0)$  does not contain any singularity and the set of
singularities is discrete, there exists a  maximal $\overline{\delta} =
\overline{\delta}({\omega_0},p_0)$ such that the strip
 \begin{equation*}
\bigcup_{0\leq t < \overline{\delta}({\omega_0},p_0)} \varphi_t I_{\omega_0}(p_0)
 \end{equation*}
does not contain any singularities, and thus is isometric to an
Euclidean rectangle of height $\overline{\delta}$ and width $1$ in
the flat coordinates given by ${\omega_0}$.

For any $p \in M_{reg}$, denote by
$\gamma_s =\gamma_s(p,\omega_0)$ the \emph{unparametrized} curve
given by the trajectory of  $(\varphi_t)_{t\in\R}$ of length $s$
starting at $p$. For each $j=1, \dots, m$, let
$\widetilde{\gamma}_j=\widetilde{\gamma}_j({\omega_0})\in
H_1(M,\Z)$ be the homology class obtained by considering the
vertical trajectory $\gamma_{\tau_j}(q,{\omega_0})$ of a point $q
\in I_j$ up to the first return time to $I_{\omega_0}$ and closing
it up with a horizontal geodesic segment  contained in
$I_{\omega_0}$. One can show that $\{
\widetilde{\gamma}_j=\widetilde{\gamma}_j({\omega_0})$, $1\leq j
\leq m\}$ generate the homology $H_1(M, \R)$ (the proof is
analogous to the proof of Lemma 2.17, \S 2.9 in \cite{ViB}). In
particular,  their Poincar{\'e} duals classes $\{
\mathcal{P}\widetilde{\gamma}_j, 1\leq j \leq m\} $
generate  $H^1(M, \R)$. Thus, it follows\footnote{This same remark
is used in \cite{Zo}, see Lemma 6.2.} that there exists a constant
$c>0$
%(depending only on the basis)
such that
\begin{equation}\label{base}
 \frac{1}{c} \| \rho \|_{\omega_0}  \leq \max_{1\leq j\leq m} \left| \int_{\widetilde{\gamma}_j } \rho\right| =
\max_{1\leq j\leq m} |\langle \mathcal{P}\widetilde{\gamma}_j ,
\rho \rangle| \leq {c} \| \rho \|_{\omega_0} ,  \text{ for  all }
\rho \in H^1(M, \R).
\end{equation}
Since $\omega_0$ does not have neither vertical nor horizontal
saddle connections, for any $\omega \in \mathcal{H}^{(1)} (k_1,
\dots, k_\kappa) $ in a sufficiently small neighbourbood of
$\omega_0$ in the stratum the induced IET $T_\omega$ on
$I_\omega(p_0)$ has the same number $m$ of exchanged intervals and
the same combinatorial datum and furthermore the lengths
$\lambda_j(\omega)$, $j=1, \dots, m$ and the quantity
$\overline{\delta}(\omega, p_0)$ change continuously with $\omega$
and  the homology classes $\widetilde{\gamma}_j(\omega)$, $1\leq j
\leq m$ are locally constant. Therefore, by choosing $\mathcal{U}$
to be a small compact neighbourhood of $\omega_0$ in $
\mathcal{H}^{(1)} (k_1, \dots, k_\kappa) $,
$\widetilde{\gamma}_j(\omega)=\widetilde{\gamma}_j(\omega_0)$  for
any  $\omega \in \mathcal{U}$ and $1\leq j \leq m$ and  there
exists constants $A_\mathcal{U}>0$ and
$C_\mathcal{U}>1 $ %and $\tau_{\mathcal{U}}>0$
such that for any $\omega \in \mathcal{U}$ and $1\leq j\leq m$ one
has
\begin{eqnarray}\label{balance}
 \lambda_j(\omega) \,
\overline{\delta}(\omega,p_0)  \geq A_\mathcal{U},
\qquad \text{ and }\qquad
\frac{1}{C_\mathcal{U}}\leq  \tau_j(\omega)\leq
C_\mathcal{U} .
\end{eqnarray}
Furthermore, since $\mathcal{U}$ is compact, there exists a
constant $K$ such that for any $\omega_1, \omega_2 \in
\mathcal{U}$, and any $\rho \in H^1(M, \R)$ the Hodge norms
satisfy $\| \rho\|_{\omega_1} \leq K \| \rho\|_{\omega_2}$ (it
follows for example from \cite{For-dev}, \S2). Thus, (\ref{base})
holds uniformly for $\omega \in \mathcal{U}$, that is, there
exists $c_\mathcal{U}>1$ such that
\begin{equation}\label{baseuniform}
 \frac{1}{c_\mathcal{U}} \| \rho \|_\omega   \leq
\max_{1\leq j\leq m} \left| \int_{\widetilde{\gamma}_j (\omega)} \rho \right| \leq
c_\mathcal{U}\| \rho \|_\omega
\qquad \text{for \ all} \ \omega \in \mathcal{U}, \ \rho \in
H^1(M, \R).
\end{equation}
Since $\omega_0$ belongs to  the support of $\mu$,
$\mu(\mathcal{U})>0$.  Let  $\mathcal{S} \subset \mathcal{H}^{(1)}$ be a
hypersurface containing $\omega_0$ transverse to $(G_t)_{t\in \R}$
and let $\mathcal{K}\subset \mathcal{S}\cap \mathcal{U}$ be a
\emph{compact} subset with positive transverse measure  such that
every $\omega \in \mathcal{K}$ is Birkhoff generic and Oseledets
regular.
% and the function $\overline{\tau}: K \to \R^+$ is uniformly bounded by a constant $\tau_\mathcal{K}$ (such a set exists since by Luzin's theorem $\overline{\tau}$ is continuous on a compact set of positive measure).

\begin{comment}
Remark also that since  both the extremal and the flat length of a
fixed curve change continuously in $\omega$ and $\mathcal{K}$ is
compact, there exists a constant $c_{\mathcal{K}} $ such that
extremal and flat lengths  are comparable on $\mathcal{K}$, i.~e.~
\begin{equation}\label{comparelength} \frac{1}{c_{\mathcal{K}}}
l_\omega(\gamma) \leq  Ext_\gamma ( \omega) \leq c_{\mathcal{K}}
l_\omega(\gamma) \end{equation}
 for any curve
$\gamma$ and any $ \omega \in \mathcal{K}$.
\end{comment}

\begin{proofof}{Lemma}{comparisontimes:lemma}
Let $\mathcal{K}$ be the section constructed above starting from
the measure $\mu$.  Since $(G_t)_{t\in \R}$ is ergodic and
$\mathcal{K}$ has positive transverse measure, there exists  a
full $\mu$-measure set $\mathcal{M}'\subset \mathcal{H}^{(1)}$ such that
for any $\omega \in \mathcal{M}'$ the forward geodesics $\{ G_t
\omega , t >0\}$ visits $\mathcal{K}$ infinitely often.  For
$\omega \in \mathcal{M}'$,  let   $t_0$ be the minimum $t\geq 0$
such that $G_t(\omega ) \in \mathcal{K}$ and let  $\{t_k\}_{k\in
\N}$ be the sequence of successive returns to $\mathcal{K}$. For
each $k \in \N$, referring to the notation introduced above, let
us  denote by $I^k: = I_{G_{t_k}\omega}(p_0)$ and by
$\widetilde{\gamma}^{(k)}_j\in H_1(M,\R)$ the homology class
$\widetilde{\gamma}_j (G_{t_k}\omega)$. As we already remarked,
the set  $\{ \widetilde{\gamma}^{(k)}_j, j=1, \dots, m\}$
generates $H_1(M, \mathbb{R})$ and by construction each
$\widetilde{\gamma}^{(k)}_j$ belongs to $H_1(M, \mathbb{Z})$. Let
us show that they  $\{ \widetilde{\gamma}^{(k)}_j, j=1, \dots,
m\}$ satisfies the conclusion of Lemma~\ref{comparisontimes:lemma}.
%, that is that
%(\ref{base0}) holds for all $k \in \N$ with $c:= c_\mathcal{U}$
%(where $c_\mathcal{U}$  is the constant defined in the
%construction of $\mathcal{K}$).
%construction of the section $\mathcal{K}$
% \begin{equation}\label{defc} c := \max
%\left\{\frac{1}{C_\mathcal{U}},
%c_{\mathcal{K}}(\tau_\mathcal{K}+1)\right\} , \end{equation}
% where
%$c_\mathcal{K}, C_\mathcal{K}$ are the constants defined in the
%construction of the section $\mathcal{K}$.
Since $\widetilde{\gamma}^{(k)}_j =
\widetilde{\gamma}_j (G_{t_k}\omega)$ and $G_{t_k}\omega \in
\mathcal{K}\subset \mathcal{U}$, it follows from  (\ref{baseuniform}) that
\begin{equation}\label{baseuniform1}
\frac{1}{c_\mathcal{U}} \| \rho \|_{G_{t_k}\omega}
\leq \max_{1\leq j\leq m} \left| \int_{\widetilde{\gamma}^{(k)}_j }
\rho\right| \leq {c_\mathcal{U}} \| \rho \|_{G_{t_k}\omega}\quad\text{ for every }
\rho \in H^1(M, \R),\end{equation}
which gives  (\ref{base0}) with $c:=c_\mathcal{U}$.
\end{proofof}

\begin{proofof}{Theorem}{cohthm}
Let $\mathcal{K}$ be the section  constructed at the beginning of
the Appendix. Let $\mathcal{M}' \subset \mathcal{H}^{(1)}$ be the
set of $\omega $ such that the forward geodesic $\{ G_t \omega , t
>0\}$ visits $\mathcal{K}$ infinitely often and the vertical
and the horizontal flow on $(M,\omega)$ are ergodic.   The set
$\mathcal{M}'$ has full $\mu$ measure since $\mu$ is ergodic and
$\mathcal{K}$ has positive transverse measure. Let us show that
$\mathcal{M}'$ satisfy the conclusion of the theorem.  Let us
remark first that, since Oseledets regular points are flow
invariant, any $\omega \in \mathcal{M}'$ is Oseledets regular by
the definition of $\mathcal{K}$.

Fix $\omega \in \mathcal{M}'$ and let $t_0$ be the minimum $t\geq
0$ such that $G_t(\omega ) \in \mathcal{K}$ and let $\{t_k\}_{k\in
\N}$ be the sequence of successive returns to $\mathcal{K}$.  Let
$\rho$ be a closed smooth form such that $[\rho]\in
E_\omega^-(M,\R)$. Let  $(\varphi_t)_{t\in \R}$ be the vertical
flow on $(M, \omega)$ and consider the function $f= i_{X_{v}} \rho
$. We want to show that the associated cocycle $F_f^v$ is a
coboundary for  $(\varphi_t)_{t\in \R}$.

%Let $M_{reg}$ be the set of points  which are regular for vertical
%and horizontal flow.
For any $p \in M_{reg}$ let $\gamma_t(p,\omega)$ be the
unparametrized curve underlying a trajectory of length $t$ for the
vertical flow on $(M,\omega)$, so that
\begin{equation}\label{ergint1}
\int_0^t f(\varphi_s p) \,d s =  \int_{\gamma_t(p,\omega)} \rho.
\end{equation}
We will show that for every $p \in M_{reg}$ the ergodic integrals (\ref{ergint1})  are
bounded uniformly in $t\geq 0 $ (and hence
deduce that $F^v_f$  is a coboundary).  We will do so (as in
\cite{For-dev}) by decomposing the integral (\ref{ergint1}) along
a special sequence of  times, given in our proof by returns to
$\mathcal{K}$.

%Let  be the point chosen in the definition of the section $\mathcal{K}$ and Let  $p\in M_{reg}$ by any fixed regular point.
Referring to the
notation introduced in the construction of the section
$\mathcal{K}$, let $p_0 \in M_{reg}$ be the point chosen in the definition of $\mathcal{K}$ and for each  visit time $t_k$, let us denote by
\begin{equation}\label{notation}
I^k: = I_{G_{t_k}\omega}(p_0),   \, I^k_j = I_j(G_{t_k}\omega),
\, \widetilde{\gamma}^{k}_j := \widetilde{\gamma}_j (G_{t_k}\omega), \, \tau^{k}_j := \tau_j (G_{t_k}\omega).
\end{equation}
Since $(G_t)_{t \in \R}$ preserves horizontal leaves and for each
$k\geq 0$ we have $I^{k+1}=I_{G_{t_{k+1}} \omega} (p_0) \subset
I_{G_{t_{k}} \omega} (p_0)=I^k$. Let us remark that we can replace
$\rho$  by any form cohomologous to $\rho$. This follows since if
$\rho\in\Omega^1(M)$ is exact then $\rho=dh$ for some smooth
function $h:M\to\R$ and
$f=i_{X_{v}}\rho=i_{X_{v}}dh=\mathcal{L}_{X_{v}}h={X_{v}}h$, so
$F^{{v}}_f$ is a coboundary.  Let us then replace the form $\rho$
by any form cohomologous to $\rho$ vanishing on a neighborhood of
$I^0$. With the customary abuse of notation, the same symbol
$\rho$ will be used for the new form.

For each $k \in N$ and each $q \in I_j^k$, let
\begin{equation*}
\tau^k(\omega,q): =  \tau_j(G_{t_k} \omega)e^{t_k}=
\tau^{k}_je^{t_k}, \qquad \gamma^k_\omega(q):=
\gamma_{\tau^k(\omega,q)}(q, \omega).
\end{equation*}
The unparametrized curve $\gamma^k_\omega(q)$ will be  called a
$k^{th}$-\emph{principal return trajectory}.   Remark that
$\gamma^k_\omega(q) $, which is the support of a trajectory of
length $\tau^k(\omega,q)= \tau^{k}_je^{t_k}$  for the vertical
flow given by $\omega$, is the same unparametrized curve than the
support of a a trajectory of length $\tau^{k}_j$ for the vertical
flow given by $G_{t_k}(\omega)$.

Fix a regular point $p \in M_{reg}$.  For any $t \in \R$,
trajectory $\gamma_t:=\gamma_t(p, \omega)$ can be inductively
decomposed into principal return trajectories as  follows
(analogously to  Lemma 9.4 in \cite{For-dev}). Let $K \in \N$  be
the maximum $k \in \N$ such that $int(\gamma_t)  \cap I^k$ has at
least two elements (where $int(\gamma)$ denotes the curve $\gamma$
without its endpoints). Let $p^{K}_0 $, $p^{K}_1 , \dots,
p_{m_{K}}^K$  be all the points in the interior of $\gamma_t$
belonging to $I^{K}$, indexed in increasing order of $t$. Then if
$\alpha_{K}$ is the initial part of the trajectory $\gamma_t$ from
$p$ to $p^{K}_0$ and $\beta_{K}$ is the final part of $\gamma_t$
from  $p_{m_K}^K$ to the final point of $\gamma_t$, one can
decompose $\gamma_t$ as
\[ \gamma_t = \alpha_{K} \cup
\bigcup_{i=0}^{m_{K}-1} \gamma^{K}_\omega(p^{K}_i)\cup \beta_{K},
\]
where all $\gamma^{K}_\omega(p^{K}_i)$ are $K^{th}$ principal
return trajectories. Moreover,  denoting by $l_{\nu}(\cdot )$ the
length of an arc with respect to $\nu \in \mathcal{M}(M)$, the
reminder curves, by construction,  satisfy $l_{G_{t_K}\omega}(
\alpha_{K}), l_{G_{t_K}\omega}( \beta_{K}) \leq \max_j \tau_j^K$
or, equivalently, $l_{\omega}( \alpha_{K}), l_{\omega}( \beta_{K})
\leq e^{t_K}\max_j \tau_j^K$. Let us estimate the number $m_K$ of
$K^{th}$-principal returns. By definition of $K$,
$int(\gamma_t)\cap I_{K+1}$ has at most one element hence
$l_\omega(\gamma_t)\leq 2e^{t_{K+1}} \max_{j} \tau_j^{K+1}  $.
Since any $K^{th}$-principal return $\gamma_\omega^K(q)$ satisfy
$l_\omega(\gamma^K_\omega(q)) \geq e^{t_K} \min_{j} \tau_j^K $,
using (\ref{balance}) we get $m_{K}\leq 2C^2_\mathcal{U}
e^{t_{K+1}-t_K}$.

%Moreover, since by (\ref{}) the lengh of each return is at most $\tau_k$
Let $k:= K-1$. To decompose $ \alpha_{k+1}$ and $\beta_{k+1}$ in
$k^{th}$ principal return trajectories, let by convention
$p^{k}_0$ be the initial point of $\beta_{k+1}$ and let $p^{k}_1
$, $p^{k}_2 , \dots, p_{m_{k}}^{k}$ denote all points of
$int(\alpha_{k+1}) \cup int( \beta_{k+1}) \cap I^{k}$, indexed in
increasing order of $t$.   Then, we have
\[ \alpha_{k+1} \cup \beta_{k+1} = \alpha_{k} \cup  \bigcup_{i=0}^{m_{k}-1}
\gamma^{k}_\omega(p^{k}_i)\cup   \beta_{k}, \] where $\alpha_{k}$
is the initial part of $\alpha_{k+1}$ from $p$ to $p^{k}_1 $ and
$\beta_{k}$ is the final part of $\beta_{k+1}$ from
$p_{m_{k}}^{k}$ to the endpoint of $\beta_{k+1}$. To estimate
$m_{k}$, reasoning similarly to  the estimate of $m_K$ and using
the upper bound on the length of $\alpha_{k+1}$ and $\beta_{k+1}$
and the lower bound given by (\ref{balance}) on the length of
$k^{th}$ principal return trajectories  together with
(\ref{balance}), we get
%remark that by construction  $l_{G_{t_k}\omega}( \alpha_{K}), l_{G_{t_k}\omega}( \beta_{K}) \leq \max_j e^{t_{K}} \tau_j^{K}$, while by (\ref{}) any $(k)^{th}$ principal return trajectory has lenght at least $\min_j e^{t_{k}} \tau_j^{k}$. Thus, by (\ref{}),
\[m_{k}  \leq 2 \, \frac{ e^{t_{k+1}}  \max_{1\leq j \leq m}
\tau_j^{k+1}}{ e^{t_{k}} \min_{1\leq j\leq m} \tau_j^{k} } \leq 2
\, C^2_\mathcal{U}\,  {e^{t_{k+1} -t_{k}}}. \] Moreover,
$l_{\omega}( \alpha_{k}), l_{\omega}( \beta_{k}) \leq
e^{t_k}\max_j \tau_j^k$. Repeating for the same construction
described for $k=K-1$ for $k=K-2, \dots, 1,0$, we get the
decomposition
%a reminder curve $\beta_\omega^t(p)$
%that $\beta_\omega^t(p)$ and the principal return trajectories
%$\gamma^k_\omega(p^k_j)\subset \gamma_t$ for  $0\leq k \leq K$
%and $1 \leq j\leq m_k$ do not overlap, one has the decomposition
%\begin{equation}\label{decomp}
%\gamma_t = \bigcup_{k=1}^{K} \bigcup_{j=1}^{m_k} \gamma_\omega^{k} (p^k_j) \cup \beta_\omega^t(p),
%\end{equation}
%For any $0\leq k \leq \overline{k}$,
%let $p^k_1 $, $p^k_2 , \dots, p_{m_k}^k$ (where one can have
%$m_k=0$, in which case the set is by convention empty) be all the
%%points of $\gamma_t$ belonging to $I_{G_{t_k}\omega}(p) \backslash
%I_{G_{t_{k+1}}\omega}(p)$, indexed in increasing order of $t$.
%Then, one can find  a reminder curve $\beta_\omega^t(p)$,  such
%that $\beta_\omega^t(p)$ and the principal return trajectories
%$\gamma^k_\omega(p^k_j)\subset \gamma_t$ for  $0\leq k \leq \overline{k}$
%and $1 \leq j\leq m_k$ do not overlap, one has the decomposition
\begin{equation}\label{decomp}
\gamma_t = \alpha_0 \cup \bigcup_{k=0}^{K} \bigcup_{j=0}^{m_k-1} \gamma_\omega^{k} (p^k_j) \cup \beta_0,
\end{equation}
where for each $0\leq k\leq K$ and  $0\leq j\leq  m_k$ the
trajectory $\gamma_\omega^{k} (p^k_j)$ is a $k^{th}$ principal
return,
\begin{equation}\label{m_kupperbound}
m_k \leq 2C^2_\mathcal{U} e^{(t_{k+1}-t_k)} \quad\text{ and }\quad
l_\omega(\alpha_0), l_\omega(\beta_0) \leq e^{t_0 } \max_j
\tau^0_j .
\end{equation}
%where $C:= \max\{2 C_\mathcal{U}, \tau_\mathcal{U} \}$ is a constant independent on $p$.
Thus, recalling (\ref{ergint1}),
\begin{equation}\label{decompintegrals}
\int_0^t f(\varphi_s p) d s %= \int_{\gamma_t} \rho
= \sum_{k=0}^{K} \sum_{j=0}^{m_k-1} \int_{\gamma_\omega^{k}  (p^k_j)}
\rho + \int_{\beta_0}  \rho + \int_{\alpha_0}  \rho.
\end{equation}
By construction, all points $(p^k_j)_{j,k}$ belong to $ I_\omega
(p_0)$.
%since $(G_t)_{t \in \R}$ preserves horizontal leaves and
%for all $t\geq 0$ we have $I_{G_t \omega} (p_0) \subset
%I_\omega(p_0)$.
For each $0\leq k \leq K$ and $1 \leq j\leq m_k$, let
$\widetilde{\gamma}_\omega^{k} (p^k_j)$ stand for the homology
class of the closed curve obtained by closing up the trajectory
${\gamma}_\omega^{k} (p^k_j)$ by the shortest geodesic connecting
its final point with its initial point. Remark that since both
initial and final points of  ${\gamma}_\omega^{k} (p^k_j)$ are
contained in $ I_{G_{t_k}\omega(p)}=I^k$,
$\widetilde{\gamma}_\omega^{k} (p^k_j)$ consists of
${\gamma}_\omega^{k} (p^k_j)$ together with a horizontal segment
contained in $I^k$.

Remark now that  if $p^k_j \in I^k_l$,  the
closed curve $\widetilde{\gamma}^k_\omega(p^k_j)$ is a
representative of the homology class
$\widetilde{\gamma}_l(G_{t_k}\omega)\in H_1(M,\Z)$ defined in the
construction of $\mathcal{K}$ at the beginning of the Appendix.
Since $\rho$ vanishes on $I^k \subset I^0$ and  and
$\widetilde{\gamma}_\omega^{k} (p^k_j)$ and ${\gamma}_\omega^{k}
(p^k_j)$ differ by a horizontal segment contained in $I^k \subset
I^0$,   in view of  \eqref{baseuniform}, we get that
\begin{equation*}
%\left| \int_{{\gamma}_j^{k}}  \rho  \right| =
\left| \int_{{\gamma}_\omega^{k} (p^k_j)} \rho  \right| = \left|
\int_{\widetilde{\gamma}_\omega^{k} (p^k_j)}  \rho  \right|=
\left| \int_{\widetilde{\gamma}_l(G_{t_k}\omega)}  \rho
\right|\leq c_{\mathcal{U}} \|\rho \|_{G_{t_k}\omega}.
\end{equation*}
Since,  by assumption,  $[\rho] \in E_\omega^-(M,\R)$ (recall
\eqref{stabledef}), it follows that there exists constants $C_1,
\theta >0$ such that
\begin{equation*}
\left| \int_{{\gamma}_\omega^{k} (p^k_j)} \rho  \right| \leq %c \|\rho \|_{G_{t_k}\omega} \leq
C_1 e^{-\theta t_k}\text{ for all }k\geq 0.
\end{equation*}
Using this inequality together with (\ref{m_kupperbound})  to
estimate (\ref{decompintegrals}), we get that there exists $C_2>0$
such that for any  $t\geq 0$, one
has
\begin{equation}\label{decompintegrals1}
\left| \int_0^t f(\varphi_s p) d s \right| \leq  C_2
\sum_{k=0}^\infty e^{( t_{k+1}-t_k)}  e^{-\theta t_k} +C_2 = C_2
\sum_{k=0}^\infty  e^{ \left( \frac{ t_{k+1}-t_k}{t_k}  - \theta
\right)t_k } + C_2 .
\end{equation}
Since $\mathcal{K}$ has positive  transverse measure and $\omega$
is Birkhoff generic (since Birkhoff generic points are
$(G_t)_{t\in \R}$-invariant and $G_{t_1}\omega \in \mathcal{K}$
which by construction consists only of Birkhoff generic points),
by Birkhoff ergodic theorem  we have  $\lim_{k \to \infty} t_k /k
= 1/\mu_{tr}(\mathcal{K})$, where $\mu_{tr}(\mathcal{K})>0$ is the
transverse measure of $\mathcal{K}$. Thus, if $k$ is sufficiently
large, $(t_{k+1}-t_k)/t_k - \theta  \leq - \theta/2$, which shows
that the above series is convergent and the ergodic integrals in
(\ref{decompintegrals1}) are uniformly bounded for all  $t\geq 0$.
By Remark \ref{coboundarycondition}  this implies that $F^{{v}}_f$
is a coboundary.

\smallskip
Let us now prove the second part of Theorem \ref{cohthm}.  Let us
assume in addition from now on that $\mu$ is KZ-hyperbolic. Let
$\omega \in \mathcal{M}'$, $p \in M_{reg}$ and let $\rho \in
\Omega^1(M)$ be a smooth closed one form such that $[\rho]\notin
E_\omega^-(M,\R)$.
%As in the first part, we can assume without loss of generality that $\rho$ vanishes on $I_\omega(p)$.
For each $k\in \N$ and $j=1, \dots, m$, using the notation
introduced in (\ref{notation}) and setting $\overline{\delta}^{k}:
=  \overline{\delta}(G_{t_k} \omega,p_0)$ (where
$\overline{\delta}$ was also defined during the construction of
$\mathcal{K}$), define the set $R^{k}_j$ by
\begin{equation*}
R^k_j := \bigcup_{0\leq s < \overline{\delta}^k }
{\varphi^{G_{t_k} \omega,v}_s}\left( I^k_j\right)= \bigcup_{0\leq
s < e^{t_k}\overline{\delta}^k } {\varphi^{\omega,v}_s}
\left( I^k_j\right). %, \qquad \text{where} \ \\overline{\delta}^{k}: =  \overline{\delta}(G_{t_k \omega},p)
\end{equation*}
%and $\delta(\omega)$ was defined in the construction of the section $\mathcal{K}$.
Thus each $R^{k}_j$ is a rectangle in translation surface
coordinates given by $G_{t_k}\omega$ (and by $\omega$) with base
$I^k_j\subset I^k \subset I_{\omega}(p_0)$ and height
$\overline{\delta}^k$ with respect to $G_{t_k}\omega$ (or $e^{t_k}
\overline{\delta}^k$ with respect to $\omega$). Since
$G_{t_k}\omega \in \mathcal{U}$, by (\ref{balance}) the area of
each rectangle (which is invariant under $(G_t)_{t\in \R}$) is
uniformly bounded from below, that is
\begin{equation}\label{areas}
\nu_\omega (R^k_j) = \nu_{G_{t_k}\omega} (R^k_j) = \lambda_j(
G_{t_k}\omega) \overline{\delta}(G_{t_k }\omega,p_0 ) \geq
A_\mathcal{U} >0
%, \\  \forall j=1, \dots, m, \ k\in \N.
\end{equation}
for all $j=1, \dots, m$ and $k \in \N$.

Assume that $q\in I^k_j$ for some $j=1, \dots, m$ and $k \in \N$.
By definition,  $e^{t_k}\tau^k_j$ is the (vertical) length of
$\gamma^k_\omega(q)$ with respect to $\omega$ and
$\widetilde{\gamma}^k_j\in H_1(M,Z)$ is the homology class of the
curve consisting of $\gamma^k_\omega(q)$ closed by a horizontal
segment $J^k(q)$ contained in $I^k = I_{G_{t_k}\omega}(p_0)$, which
has horizontal length one in the flat metric given by
$G_{t_k}\omega$ Thus, for $f=i_{X_{v}} \rho$ we have
\begin{align}\label{difference1}
\begin{aligned}
\Big| \int_0^{e^{t_k}\tau_j^k} & f(\varphi_s q) \,d s -
\int_{\widetilde{\gamma}_j^k} \rho \Big|=
\Big|\int_{\gamma^k_\omega(q)} \rho     -
\int_{\widetilde{\gamma}_j^k}  \rho \Big|  = \Big|\int_{J^k(q)}
\rho \Big|\\
&\leq \int_0^{l_\omega(I^k)} |i_{X_{h}} \rho(\varphi^h_s
q)|\,ds\leq
%\vert \vert i_{X_{h}} \rho \vert \vert_\infty l_{G_{t_k}\omega}(I^k) =
l_\omega(I^k)\|i_{X_{h}} \rho \|_\infty =e^{-t_k}\|i_{X_{h}} \rho
\|_\infty.
\end{aligned}
\end{align} Let us now show that for any $p \in R^k_j$, setting
$c_\rho:= 3 \vert \vert i_{X_{h}} \rho \vert \vert_\infty$, we
have
\begin{equation}\label{samevalue}
\Big| \int_{\widetilde{\gamma}_j^k} \rho  -
\int_0^{e^{t_k}\tau_j^k} f(\varphi_s p) d s \Big| \leq c_\rho.
\end{equation}
Given $p\in R^k_j$,  since the height of $R^k_j$ in the
translation structure given by $\omega$ is $e^{t_k}
\overline{\delta}^k$, we can write $p = \varphi_{u} q$ for some
$0\leq u\leq e^{t_k} \overline{\delta}^k$ and $q \in I^k_j$. Thus,
since $\varphi_{e^{t_k}\tau_j^k-u}( p) =\varphi_{e^{t_k}\tau_j^k}
(q) = T_k (q)$, where $T_k:=T_{G_{t_k} \omega,p_0}$ is the first
return map of $(\varphi_t)_{t\in \R}$ to $I^k=I_{G_{t_k}
\omega}(p_0)$, we can write
\begin{equation}\label{equal}
\int_0^{e^{t_k}\tau_j^k}  f(\varphi_s q) \,d s  -
\int_0^{e^{t_k}\tau_j^k} f(\varphi_s p) \,d s = \int_0^{u}
f(\varphi_s q) \,d s - \int_{0}^{u} f(\varphi_s T_k (q) ) \,d s.
\end{equation}
Remark now that  $q, T_k(q),  \varphi_{u} q , \varphi_{u} T_k(q)$
are corners of a rectangle $R$ (since they are contained in the
rectangle of base $I^k$ and height $e^{t_k } \
\overline{\delta}^k$ in the translation structure given by
$\omega$). Denote by $\partial_vR$ and $\partial_hR$ the vertical
and the horizontal part of the boundary of $R$ respectively . Then
$\int_{\partial_vR}\rho$ is equal to the RHS of \eqref{equal} and
$\int_{\partial_hR}\rho$ is bounded by $2\| i_{X_{h}}\rho \|_{\infty}$.
Thus, since $\int_R d\rho=0$ ($\rho$ is closed and $R$ is simply
connected), by Stoke's theorem, $0=\int_{\partial
R}\rho=\int_{\partial_v R}\rho+\int_{\partial_h R}\rho$. It
follows  that
 \begin{equation*}
\Big|  \int_0^{u}   f(\varphi_s q) d s - \int_{0}^{u} f(\varphi_s
T_k(q) ) d s  \Big|  = \Big| \int_{\partial_v R} \rho \Big| =
\Big| \int_{\partial_h R}  \rho \Big| \leq 2\| i_{X_{v}\rho }
\|_{\infty}.
\end{equation*}
This, combined with  (\ref{equal}) and (\ref{difference1}), yields
(\ref{samevalue}).

Let us now prove that  the cocycle $F^v_f$ is not  a coboundary.
Assume by contradiction that  $F^v_f$ is a coboundary with a
measurable transfer function $u: M \to \R$. Then there exists a
constant $M_\mathcal{U}$, depending on $A_\mathcal{U}$, such that
the  set
\[\Lambda(M_{\mathcal{U}}) :=
\{ p \in M :|u(p)|\geq M_\mathcal{U}\}\quad\text{ satisfies }\quad
\nu_{\omega}(\Lambda(M_{\mathcal{U}}))\leq A_{\mathcal{U}}/2.\]
Thus, for any fixed $1\leq j \leq m$, for all $p$ in a set of
$\nu_\omega$-measure greater than $1-A_\mathcal{U}$ (more
precisely, for all $p \notin \Lambda(M_{\mathcal{U}}) \cup
\varphi_{-e^{t_k}\tau_j^k}(\Lambda(M_{\mathcal{U}}))$), we have
\begin{equation}\label{bounded}
\vert F^{v}_f ( e^{t_k}\tau_j^k,p)  \vert = \Big\vert
\int_0^{e^{t_k}\tau_j^k} f(\varphi^{v}_s p)\,ds \Big\vert =  \vert
u(\varphi_{e^{t_k}\tau_j^k} p) - u(p) \vert \leq 2
M_{\mathcal{U}}.
\end{equation}
Since   $\nu_\omega (R_j^k) \geq A_\mathcal{U}$  (see
\eqref{areas}), there exists $p_j \in R_j^k$ satisfying
(\ref{bounded}). Repeating the same argument for each $1\leq j
\leq m$ and recalling (\ref{baseuniform1}), which holds since
$G_{t_k}\omega \in \mathcal{U}$, and (\ref{samevalue})  we get
\[
\frac{1}{{c_\mathcal{U}}}\|\rho \|_{G_{t_k} \omega} \leq
\max_{1\leq j\leq m}\Big| \int_{\widetilde{\gamma}_j^k}
\rho\Big|\\ \leq \max_{1\leq j\leq m}\Big|
\int_0^{e^{t_k}\tau_j^k}  f(\varphi_s p_j) \,d s\Big| +c_\rho \leq
{2M_\mathcal{U}+ c_\rho}.
\]
Thus, $\liminf_{t \to +\infty }\vert \vert \rho \vert \vert _{G_{t}
\omega}<\infty$. Since  $\mu$ is KZ-hyperbolic, recalling the
definition of the stable space (\ref{stabledef}), this implies
that $ [\rho]\in E_\omega^-(M,\R)$, contrary to the assumptions.
Thus, we conclude that $F^{v}_f$ cannot be a coboundary.
\end{proofof}

\section*{Acknowledgments}
We  would like to  thank Vincent Delacroix, Giovanni Forni   and
Pascal Hubert for useful discussions and suggestions that helped
us  improve the paper and J.-P. Conze, P. Hooper, M. Lema\'nczyk,
C. Matheus and B. Weiss for useful discussions.   The second
author is currently supported by an RCUK Academic Fellowship and
an EPSRC First Grant, whose support is fully acknowledged.

%Both authors would like to \marginpar{Should we thank Bedlewo? can
%you write that if you wish?} thank ?? the Program bla bal in
%Bedlewo where part of this work was carried

% Moreover, $Z_3^\infty$ has an
%infinite strip. On the other hand, by Theorem~\ref{maintheorem},
%the directional flow on $Z_3^\infty$ is not ergodic for almost
%every direction which gives a counter-example to Theorem~1 in
%\cite{Hu-We}.

\end{document}